\documentclass{amsart}
\usepackage{amsmath,amssymb,amsthm,bm,bbm}
\usepackage{amscd}
\usepackage{graphicx,enumerate}
\usepackage{multirow}
\usepackage{layout}
\usepackage[all]{xy}
\usepackage[OT2,T1]{fontenc}
\usepackage[dvipsnames]{xcolor}
\usepackage{ulem}
\DeclareSymbolFont{cyrletters}{OT2}{wncyr}{m}{n}
\DeclareMathSymbol{\Sha}{\mathalpha}{cyrletters}{"58}
\usepackage[OT2,OT1]{fontenc}
\newcommand\cyr{\renewcommand\rmdefault{wncyr}T
\renewcommand\sfdefault{wncyss}
\renewcommand\encodingdefault{OT2}
\normalfont\selectfont}
\DeclareTextFontCommand{\textcyr}{\cyr}

\usepackage{hyperref}
\usepackage{url}

\theoremstyle{plain}
\newtheorem{theorem}{Theorem}[section]
\newtheorem*{theorem-nn}{Theorem}
\newtheorem{lemma}[theorem]{Lemma}
\newtheorem{proposition}[theorem]{Proposition}
\newtheorem*{proposition-nn}{Proposition}

\newtheorem*{question}{Question}

\theoremstyle{definition}
\newtheorem{definition}[theorem]{Definition}
\newtheorem{example}[theorem]{Example}
\newtheorem{remark}[theorem]{Remark}

\newtheorem*{acknowledgments}{Acknowledgments}

\theoremstyle{remark}

\newcommand{\bZ}{\mathbbm{Z}}\newcommand{\bQ}{\mathbbm{Q}}
\newcommand{\bC}{\mathbbm{C}}
\newcommand{\bG}{\mathbbm{G}}\newcommand{\bF}{\mathbbm{F}}
\newcommand{\bA}{\mathbbm{A}}\newcommand{\bP}{\mathbbm{P}}

\newcommand{\T}{\mathcal{T}}

\newcommand{\Hom}{\mathrm{Hom}}
\newcommand{\Ext}{\mathrm{Ext}}
\newcommand{\Gal}{\mathrm{Gal}}

\newcommand{\GL}{{\rm GL}}
\newcommand{\SL}{{\rm SL}}
\newcommand{\PGL}{{\rm PGL}}
\newcommand{\PSL}{{\rm PSL}}
\newcommand{\Aut}{\mathrm{Aut}}

\newcommand{\Id}{\mathrm{Id}}

\title[Norm one tori of tensor products of \'etale algebras and Hasse norm principle]
{Rationality problem for norm one tori of tensor products of \'etale algebras and Hasse norm principle}

\author[M. Florence]{Mathieu Florence}
\address{Universit\'e Paris Cit\'e and Sorbonne Universit\'e, CNRS, IMJ-PRG, F-75005 Paris, France}
\email{mflorence@imj-prg.fr}

\author[A. Hoshi]{Akinari Hoshi}
\address{Department of Mathematics, Niigata University, Niigata 950-2181, Japan}
\email{hoshi@math.sc.niigata-u.ac.jp}

\author[A. Yamasaki]{Aiichi Yamasaki}
\address{Department of Mathematics, Kyoto University, Kyoto 606-8502, Japan}
\email{aiichi.yamasaki@gmail.com}

\thanks{{\it Key words and phrases.} 
Algebraic tori, norm one tori, rationality problem, flabby resolution, 
stably rational, retract rational. 
This work was partially supported by JSPS KAKENHI Grant Numbers 
19K03418, 20H00115, 20K03511, 24K00519, 24K06647.
}

\subjclass[2010]{Primary 11E72, 12F20, 13A50, 14E08, 20C10, 20G15.}

\begin{document}
\maketitle
\begin{abstract}
Let $k$ be a field. 
Let $A=\prod_{i=1}^r K_i$ and $B=\prod_{j=1}^s E_j$ be \'etale $k$-algebras 
where $K_i$ and $E_j$ are finite separable field extensions of $k$ with 
$[K_i:k]=m_i$ and $[E_j:k]=n_j$. 
Let $\T_A=R^{(1)}_{A/k}(\bG_m)$ 
be the norm one torus of the \'etale $k$-algebra $A$. 
We prove that 
if $\gcd(m_i,n_j\mid 1\leq i\leq r, 1\leq j\leq s)=1$ and 
$\T_A$ and $\T_B$ are 
stably $($resp. retract$)$ $k$-rational, then 
the algebraic $k$-torus $\T_A\otimes \T_B$ 
and the norm one torus $\T_{A\otimes B}$ 
are stably $($resp. retract$)$ $k$-rational. 
In particular, if $k$ is a global field, then 
the Hasse norm principle holds for $(A\otimes B)/k$. 
We introduce a new invariant of $G$-lattices, the permutation order, 
whose triviality is equivalent to invertibility, 
and use it to study the rationality of tensor products $T_1\otimes T_2$ 
of algebraic $k$-tori. 
As an application, we obtain large families of field extensions $K/k$ 
for which the Hasse norm principle holds. 
\end{abstract}
\tableofcontents
%
\section{Introduction}\label{S1}

Let $k$ be a field and $A$ be an \'etale $k$-algebra,
i.e. a finite product of finite separable field extensions of $k$.
The norm one torus $\T_A:=R^{(1)}_{A/k}(\bG_m)$ is a linear algebraic group
over $k$. Its group of $k$-points is
\begin{align*}
\T_A(k)=\mathrm {Ker}(N_{A/k}:A^\times \to k^\times)
\end{align*}
where $N_{ A/k}$ is
the \textit{norm} of the finite  $k$-algebra $A$.\\
\noindent Norm one tori have been intensively studied, especially with regards
to rationality questions, as well as to other interesting invariants and
classical arithmetic questions,
e.g. $R$-equivalence and Hasse norm principle.
In this paper, we investigate some meaningful situations, where properties of 
algebraic tori are preserved {\it upon tensor product}, in two different ways.
Let us formulate this rigorously. 
For algebraic $k$-tori $T_1$ and $T_2$ with character lattices 
$\widehat{T}_1$ and $\widehat{T}_2$, denote by $T_1 \otimes T_2$ 
the algebraic $k$-torus whose character lattice is 
$\widehat{T}_1 \otimes_{\mathbb Z} \widehat{T}_2$.
\begin{question}
Let $\mathcal P$ be a property that an algebraic $k$-variety $X$
may possess $($for instance, stable $k$-rationality$)$.
Assume that $\mathcal P$  holds for two algebraic $k$-tori $T_1$ and $T_2$.
Does $\mathcal P$ hold for  $T_1 \otimes T_2$?\\
\noindent
If $T_1=\T_A$ and $T_2=\T_B$ $($for \'etale $k$-algebras $A$ and $B$$)$,
does $\mathcal P$ hold for $\T_{A\otimes_k B}$?
\end{question}
In general, both questions are of course answered in the negative.
However, when $\mathcal P$ is `stable rationality' or `retract rationality', 
there are interesting situations for which the answer is positive.
Accordingly, our first main result is Theorem \ref{th6.15}.
It states that  $T_1 \otimes T_2$ indeed has $\mathcal P$,
provided the \textit{permutation orders}  of the character lattices of $T_1$ and $T_2$
are coprime.
Our second main result is Theorem \ref{th7.5}.
It gives a sufficient condition on the \'etale $k$-algebras $A$ and $B$
(satisfied if their degrees  are coprime),
ensuring a positive answer to the second question above.
Proofs are constructive, by explicit and elementary manipulations of
flabby resolutions of $G$-lattices (with $G$ a finite group).
In that process, we introduce the useful invariant of a $G$-lattice referred to above:
its permutation order (see Definition \ref{DefiPermutationOrder}).
To our knowledge,  this simple invariant had not yet been systematically studied.\\
As an application, we provide  new examples of stably $k$-rational
norm one tori $\T_K$ where $K/k$ are finite separable field extensions.
If $k$ is a global field, we note that the Hasse norm principle holds
for these field extensions.\\

The paper is structured as follows. Section 2 gathers definitions and
recollections of well-known concepts and (cohomological) techniques.
These are used later on, to investigate rationality properties of
algebraic $k$-tori.
For the reader's convenience, most known results on that topic,
are recalled in Section 3.
Similarly, Section 4 contains classical material on the Hasse norm principle (for torsors under algebraic $k$-tori over global fields).
In Section 5, a bunch of new cohomological tools are developed. Among these,  some explicit results on tensor products of extensions of $G$-lattices, are presented in full detail.
We typically perform concrete constructions at the level of such extensions,
rather than merely dealing with their isomorphism classes
(living in suitable cohomology groups).
This is a computer-friendly approach to the cohomology theory of finite
groups.
In Section 6, we apply these constructions to flabby resolutions,
in order to prove our first main result,
Theorem \ref{th6.15}.
In Section 7, focus is laid on norm one tori,
in order to prove our second main result, Theorem \ref{th7.5}.
Concrete examples are provided.
Finally, Section 8 is an application of this work
to the Hasse norm principle.

\begin{acknowledgments}
The authors would like to thank Colliot-Th\'el\`ene
for giving them valuable remarks to clarify Definition \ref{deftensorT}. 
They also would like to thank Federico Scavia 
for giving them valuable comments to Proposition \ref{prop6.5} and 
Proposition \ref{prop6.6}. 
\end{acknowledgments}

\section{Notation and recollections}\label{S2}
In this paper, $k$ is a field, with fixed separable closure $\overline{k}$. 
Let $\mathcal{G}={\rm Gal}(\overline{k}/k)$ be the absolute Galois group 
of $k$.

\subsection{$G$-sets and $G$-lattices}
Let $G$ be a profinite group (often a finite group). 
A $G$-set is a set $X$, equipped with a continuous action of $G$. 
Here `continuous' means that the stabiliser of every element of $X$ is open 
in $G$. 
This condition is empty if $G$ is finite. 
In this paper, 
all actions of profinite groups are continuous.\\
A $G$-lattice is a free $\bZ$-module $M$, endowed with a (continuous) 
$G$-action. [Thus, some open subgroup $G_0 \subset G$ acts trivially on $M$.] 
The lattice dual to $M$ is $M^\circ:=\Hom_{\bZ\text{-}{\rm mod}}(M,\bZ)$, 
equipped with its 
$G$-action, given by $(g\phi)(m)=\phi(g^{-1}m)$ 
$(g\in G, \phi\in M^\circ, m\in M)$. 
More generally, for another $G$-lattice $N$, 
${\rm Hom}_{\bZ}(M,N)$ is a $G$-lattice, via the $G$-action 
\begin{align*}
(gu)(m)=g\left(u(g^{-1}m)\right)\ (g\in G, u\in {\rm Hom}(M,N), m\in M).
\end{align*}
Accordingly, the tensor product $G$-lattice $M\otimes N$ 
(where $\otimes=\otimes_\bZ$) is a $G$-lattice as well, 
with $G$-action given by $g(m \otimes n)=(gm) \otimes (gn)$. 
The following basic construction is essential. 
\begin{definition}\label{def8.1}
Let $X$ be a finite $G$-set. 
Denote by $\bZ[X]$ the $G$-lattice 
which is the free $\bZ$-module with basis $\{[x]\mid x \in X\}$, 
on which $G$ acts by permutation. 
Consider the 
exact sequence of $G$-lattices 
\begin{align*}
(E_X): 0\to I_X\to\bZ[X]\xrightarrow{\varepsilon}\bZ\to 0
\end{align*}
where $\varepsilon(\sum_{x\in X}a_x[x])=\sum_{x\in X}a_x$ 
is the \textit{augmentation map} and $I_X:={\rm Ker}(\varepsilon)$. 
Take $J_X:=I_X^\circ$, its dual sequence reads as
\begin{align*}
(F_X): 0 \to\bZ^\circ\simeq \bZ\xrightarrow{\varepsilon^\circ}
\bZ[X]^\circ\simeq \bZ[X]\to J_X\to 0
\end{align*}
with 
\begin{align*}
\varepsilon^\circ: \bZ&\to \bZ[X],\\
n&\mapsto nN_X
\end{align*} 
where $N_X:=\sum_{x\in X}[x]$ is called the \textit{norm} (element). \\
Say that $(E_X)$ (resp. $(F_X)$) is the augmentation (resp. co-augmentation) 
sequence, w.r.t. the finite $G$-set $X$.
\end{definition}

\begin{remark}
For two finite $G$-sets $X$ and $Y$, 
there is an 
isomorphism 
\begin{align*}
\varphi: &\ \bZ[X] \otimes \bZ[Y] \xrightarrow{\sim} \bZ[X\times Y],\\
&(\sum_{x\in X}a_x[x])\otimes(\sum_{y\in Y}b_y[y])\mapsto 
\sum_{x\in X}\sum_{y\in Y}a_xb_y[(x,y)].
\end{align*}
\end{remark}

\begin{definition}[The $G$-lattice $M_G$ of a finite subgroup $G$ of 
$\GL(n,\bZ)$]\label{defMG} 
Let $G$ be a finite subgroup of $\GL(n,\bZ)$. 
The $G$-lattice $M_G$ of rank $n$ 
is defined to be the $G$-lattice with $\bZ$-basis $\{u_1,\ldots,u_n\}$, 
on which $G$ acts by $u_i^{\sigma}=\sum_{j=1}^n a_{i,j}u_j$ for any $
\sigma=[a_{i,j}]\in G$. 
\end{definition}

\begin{definition}\label{def6.8}
Let $G_1$ (resp. $G_2$) be a finite group 
and $M_1$ (resp. $M_2$) be a $G_1$-lattice (resp. $G_2$-lattice). 
Let $G\leq G_1\times G_2$ be a subdirect product of $G_1, G_2$ 
with surjections 
$\varphi_1: G\rightarrow G_1$ and $\varphi_2: G\rightarrow G_2$.
We define the $G$-lattice 
$M_1\otimes M_2$, {\it tensor product of $M_1$, $M_2$}, 
by $g(m_1\otimes m_2)=\varphi_1(g)(m_1)\otimes \varphi_2(g)(m_2)$ 
$(m_1\in M_1, m_2\in M_2)$. 
Note that the action of $G$ on $M_1\otimes M_2$ may not be faithful 
and $M_1\otimes M_2$ may be regarded as $G/N$-lattice where $N\lhd G$ 
is the action kernel of $G$. 
\end{definition}

\begin{definition}
Let $G_1\leq \GL(m,\bZ)$ and $G_2\leq \GL(n,\bZ)$. 
Let $G\leq G_1\times G_2$ be a subdirect product of $G_1, G_2$ 
with surjections 
$\varphi_1: G\rightarrow G_1$ and $\varphi_2: G\rightarrow G_2$. 
For $g_1\in G_1$, $g_2\in G_2$, consider the Kronecker product 
$g_1\otimes g_2\in \GL(mn,\bZ)$, i.e.  
\begin{align*}
g_1\otimes g_2=
\left[\begin{array}{ccc}
   a_{11}\,g_2 & \cdots & a_{1m}\,g_2 \\
   \vdots & \ddots & \vdots\\
   a_{m1}\,g_2 & \cdots & a_{mm}\,g_2
\end{array}\right]\in \GL(mn,\bZ)
\end{align*}
where $g_1=[a_{ij}]_{1\leq i,j\leq m}\in \GL(m,\bZ)$ and 
$g_2\in \GL(n,\bZ)$ 
with $(g_1\otimes g_2)(g_1^\prime\otimes g_2^\prime)
=(g_1g_1^\prime)\otimes (g_2g_2^\prime)\in \GL(mn,\bZ)$. 
Define 
\begin{align*}
G_1\otimes G_2=\{g_1\otimes g_2\mid g_1=\varphi_1(g)\in G_1, g_2=\varphi_2(g)\in G_2\}
\leq \GL(mn,\bZ).
\end{align*}
Let $M_{G_1\otimes G_2}$ be as in Definition \ref{defMG}. 
Note that the action of $G$ on $M_{G_1\otimes G_2}$ may not be faithful. 
For example, when $m=n=1$, 
there is a 
surjection $G \to G_1\otimes G_2$, 
whose kernel is contained in $\{(\pm 1)\}\leq\GL(1,\bZ)$. 
As $G$-lattices, 
\begin{align*}
M_{G_1}\otimes M_{G_2}\simeq M_{G_1\otimes G_2}.
\end{align*}
\end{definition}

In the sequel, the letter $G$ always denotes a finite group.

\subsection{Classical terminology for $G$-lattices, and related basic results}
A $G$-lattice $M$ is a {\it permutation} $G$-lattice 
if it has a $\bZ$-basis permuted by $G$. This means that
 $M\simeq \oplus_{1\leq i\leq m}\bZ[G/H_i]$ 
for some subgroups $H_1,\ldots,H_m$ of $G$. Equivalently, in more concise terms, $M \simeq \bZ[X]$, for a finite $G$-set $X$.  \\
The $G$-lattice  $M$ is called {\it stably permutation} 
if $M\oplus P\simeq P^\prime$ 
for some permutation $G$-lattices $P$ and $P^\prime$. \\The $G$-lattice 
$M$ is called {\it invertible} 
if it is a direct summand of a permutation $G$-lattice, 
i.e. $P\simeq M\oplus M^\prime$ for some permutation $G$-lattice 
$P$ and a $G$-lattice $M^\prime$. \\
The $G$-lattice  $M$ is called {\it flabby} (or {\it flasque}) if $\widehat H^{-1}(H,M)=0$ 
for any subgroup $H$ of $G$ (here $\widehat H$ denotes Tate modified cohomology). 
It is called {\it coflabby} (or {\it coflasque}) if $H^1(H,M)=0$
for any subgroup $H$ of $G$. \\
Say that two $G$-lattices $M_1$ and $M_2$ are {\it similar} 
if there exist permutation $G$-lattices $P_1$ and $P_2$ such that 
$M_1\oplus P_1\simeq M_2\oplus P_2$. 
The set of similarity classes is a commutative monoid 
with respect to the sum $[M_1]+[M_2]:=[M_1\oplus M_2]$ 
and the zero $0=[P]$ where $P$ is a permutation $G$-lattice. \\
For a $G$-lattice $M$, there exists a {\it flabby resolution} of $M$; that is, an exact sequence of $G$-lattices 
\[
0 \rightarrow M \rightarrow P \rightarrow F \rightarrow 0
\]
where $P$ is permutation and $F$ is flabby. 
Flabby resolutions can be constructed explicitly;
see Endo and Miyata \cite[Lemma 1.1]{EM75}, 
Colliot-Th\'el\`ene and Sansuc \cite[Lemma 3]{CTS77}, 
Manin \cite[Appendix, page 286]{Man86}.
The similarity class $[F]$ of $F$ is determined uniquely 
and is called {\it the flabby class} of $M$. 
We denote the flabby class $[F]$ of $M$ by $[M]^{fl}$. 
We say that $[M]^{fl}$ is invertible if $[M]^{fl}=[E]$ for some 
invertible $G$-lattice $E$. 
For a $G$-lattice $M$, 
it is not difficult to see that 
\begin{align*}
\textrm{permutation}\ \ 
\Rightarrow\ \ 
&\textrm{stably\ permutation}\ \ 
\Rightarrow\ \ 
\textrm{invertible}\ \ 
\Rightarrow\ \ 
\textrm{flabby\ and\ coflabby}\\
&\hspace*{8mm}\Downarrow\hspace*{34mm} \Downarrow\\
&\hspace*{7mm}[M]^{fl}=0\hspace*{10mm}\Rightarrow\hspace*{5mm}[M]^{fl}\ 
\textrm{is\ invertible}.
\end{align*}
The above implications in each step cannot be reversed 
(see, for example, \cite[Section 1]{HY17}). 
For  basic facts on flabby $G$-lattices, 
see Colliot-Th\'el\`ene and Sansuc \cite{CTS77}, 
Swan \cite{Swa83}, \cite{Swa10}, Voskresenskii \cite[Chapter 2]{Vos98}, 
Lorenz \cite[Chapter 2]{Lor05}, Hoshi and Yamasaki \cite[Chapter 2]{HY17}.

\subsection{\'Etale $k$-algebras}\label{ss23}
This notion is an enhancement of Galois correspondence. 
Let $A$ be a commutative $k$-algebra and 
$X(A):={\rm Hom}_{k}(A,\overline{k})$ be the set of 
$k$-algebra homomorphisms from $A$ to $\overline{k}$. 
Then $X(A)$ becomes a $\mathcal{G}$-set, 
via the  $\mathcal{G}$-action on the target. 
Conversely, let $X$ be a \textit{finite} $G$-set. 
There is a 
$\mathcal{G}$-action on 
$\overline{k}^X=\mathrm{Maps}(X,\overline k)$. 
Define the $k$-algebra $k \{X \}:=H^0(\mathcal G,\overline{k}^X)$, 
to be the sub-$k$-algebra of $\mathcal{G}$-fixed points. 
Then $k \{X \}\simeq K_1\times\cdots\times K_r$ where each $K_i/k$ 
is a finite  separable field extension. 
Such a $k$-algebra is called \textit{\'etale}. 
The associations $A \mapsto X(A)$ and $X \mapsto k \{X \}$ are 
mutually inverse anti-equivalences, from the category of \'etale $k$-algebras 
(morphisms being $k$-algebra homomorphisms) and that of finite 
$\mathcal{G}$-sets (morphisms being $\mathcal{G}$-equivariant maps). 
Furthermore, this equivalence is compatible to monoidal structures. 
More precisely, it turns a product (resp. tensor product) of $k$-algebras, 
into a disjoint union (resp. product) of $\mathcal{G}$-sets. 
This is expressed by the correspondences
\begin{align*}
&X(A\times B)\leftrightarrow X(A)\cup X(B),\\
&X(A\otimes B)\leftrightarrow X(A)\times X(B).
\end{align*}
For details (and more) see 
Bourbaki {\cite[V. \S\S 6--8]{Bou81}, \cite[V. \S\S 6--8]{Bou90}, 
Waterhouse \cite[Theorem, page 47, page 48]{Wat79}, 
Knus, Merkurjev, Rost and Tignol \cite[Proposition 18.3, Theorem 18.4]{KMRT98}. 
A fundamental observation (to be used tacitly throughout) is that the $\mathcal{G}$-action on $X(A)$ is transitive if and only if $A$ is a field. 
More precisely, for a finite extension $L/k$, 
with $L\subset \overline k$, one has $X(L)=\mathcal{G}/\mathcal{H}$ 
where $\mathcal{H}=\Gal(\overline{k}/L)$. 

\subsection{Varieties over $k$} 
In this paper, 
a  $k$-variety is a $k$-scheme $X$ that is reduced, separated and 
of finite-type. 
Furthermore, we always assume that $X$ is quasi-projective 
(= an open subvariety of a projective $k$-variety). 

\subsection{Algebraic tori} 
An algebraic $k$-torus $T$ (of dimension $n$) is a $k$-group scheme, 
such that the base-change 
$T\times_k \overline{k}:=T\times_{{\rm Spec}\, k}\,{\rm Spec}\, \overline{k}$ 
is isomorphic to $(\bG_{m,\overline{k}})^n$ as a $\overline{k}$-group scheme. 
Equivalently, it is a $k$-form of the split torus $(\bG_{m,k})^n$. 
Observe that an algebraic $k$-torus is an affine $k$-variety.

Let $T$ be an algebraic $k$-torus. 
Its \textit{character lattice} is 
$\widehat{T}:={\rm Hom}(T\times_k \overline{k},\bG_{m,\overline{k}})$ 
(homomorphisms of $\overline{k}$-group schemes). 
It is a $\mathcal{G}$-lattice 
(note that ${\rm Hom}(\bG_{m,\overline{k}},\bG_{m,\overline{k}})\simeq \bZ$). 
By Galois correspondence, there exists a finite Galois extension $L/k$ 
with Galois group $G:={\rm Gal}(L/k)$ such that 
$T$ splits over $L$, i.e. $T\times_k L\simeq (\bG_{m,L})^n$. 
The kernel of the $\mathcal{G}$-action  on $\widehat{T}$ yields the minimal 
such extension $L/K$, called \textit{the minimal splitting field} of $T$. 
The association $T \mapsto \widehat T$ is an anti-equivalence (duality), 
between the category of algebraic $k$-tori which split over $L$ 
(morphisms are homomorphisms of $k$-group schemes), 
and the category of $G$-lattices (morphisms are $\bZ[G]$-linear maps). 
In fact, if $T$ is an algebraic $k$-torus which splits over $L$, 
then the character lattice $\widehat{T}$ 
may be regarded as a $G$-lattice with 
$\widehat{T}\simeq {\rm Hom}(T\times_k L,\bG_{m,L})$ 
where $G\simeq \mathcal{G}/\mathcal{H}$ and 
$\mathcal{H}={\rm Gal}(\overline{k}/L)$. 
Conversely, for a given $G$-lattice $M$, 
$T:={\rm Spec}(L[M]^G)$ becomes an algebraic 
$k$-torus which splits over $L$, and 
such that $\widehat{T}\simeq M$ as $G$-lattices. 
Here $L[M]^G$ denotes the invariant ring of the group algebra $L[M]$, 
under the action of $G$. 
For details, see Ono \cite[Section 1.2]{Ono61}, 
Voskresenskii \cite[page 27, Example 6]{Vos98} and 
Knus, Merkurjev, Rost and Tignol \cite[page 333, Proposition 20.17]{KMRT98}). 

\begin{definition}\label{deftensorT}
For algebraic $k$-tori $T_1$, $T_2$ with character lattices 
$\widehat{T}_1$, $\widehat{T}_2$, via the equivalence above, 
we define {\it the tensor product $T=T_1\otimes T_2$ of $T_1$, $T_2$} 
by $\widehat{T}\simeq \widehat{T}_1\otimes \widehat{T}_2$ 
where $\widehat{T}_1\otimes \widehat{T}_2$ is 
the tensor product of $\widehat{T}_1$, $\widehat{T}_2$ 
as in Definition \ref{def6.8} 
(the character lattice 
$\widehat{T}$ 
can be regarded as a $G$-lattice 
for some subdirect product $G\leq G_1\times G_2$ 
of $G_1$, $G_2$ with surjections 
$\varphi_1: G\rightarrow G_1$, $\varphi_2: G\rightarrow G_2$ 
for $G_1$-lattice $\widehat{T}_1$ 
and $G_2$-lattice $\widehat{T}_2$).  
\end{definition}

There is a useful interpretation of the preceding equivalence. 
Namely, isomorphism classes of $k$-tori of dimension $n$ 
correspond bijectively to the elements of the set 
$H^1(k,\GL_n(\bZ)):=H^1(\mathcal{G},\GL_n(\bZ))$, because 
${\rm Aut}((\bG_{m,\overline{k}})^n)=\GL_n(\bZ)$. 
Thus, an algebraic $k$-torus $T$ of dimension $n$ is uniquely 
determined by its associated (continuous) integral 
representation $h : \mathcal{G}\rightarrow \GL_n(\bZ)$ 
(defined up to conjugacy). 
Here the group $G=h(\mathcal{G})$ is a finite subgroup of $\GL_n(\bZ)$ 
(see Voskresenskii \cite[page 57, Section 4.9]{Vos98}). 
The minimal splitting field $L$ of $T$ corresponds to 
$\mathcal{H}={\rm Ker}(h)$.

\subsection{Weil (scalar) restriction}
Let $K/ k$ be a  finite  separable  extension of $k$. 
Let $Y$ be a $K$-variety (recall that $Y$ is quasi-projective by convention). 
Denote by $R_{K/k}(Y)$ the Weil (scalar) restriction of $Y$ 
(see Ono \cite[Section 1.4]{Ono61}, 
Voskresenskii \cite[page 37, Section 3.12]{Vos98}). 
It is a $k$-variety, characterised by the formula, 
for every $k$-algebra $A$: 
\begin{align*}
R_{K/k}(Y)(A)=Y(A \otimes_k K).
\end{align*}
In particular, a $k$-point of $R_{K/k}(Y)$ is the same thing 
as a $K$-point of $Y$.\\
The $\overline k$-variety $R_{K/k}(Y) \times_k {\overline k}$ 
is canonically isomorphic to 
the product of the $\overline k$-varieties $Y_\sigma $, 
for all $\sigma \in X(K)$. 
Here $Y_\sigma$ denotes the $\overline k$-variety obtained 
from the $K$-variety $Y$, upon the base-change 
$K \xrightarrow{\sigma} \overline k$. 

The functorial nature of the Weil (scalar) restriction, 
ensures that if $Y$ is a $K$-group scheme, then $R_{K/k}(Y)$ 
is a $k$-group scheme. 
From there, one sees that if $T$ is an algebraic $K$-torus, 
then $R_{K/k}(T)$ is an algebraic $k$-torus. 
In case $T=\bG_{m,K}$, the algebraic $k$-torus $R_{K/k}(\bG_{m,K})$ will 
be of interest to us.  
It is sometimes called a \textit{quasi-trivial} torus, 
and is simply denoted by $R_{K/k}(\bG_m)$. 
Its character lattice is the permutation $G$-lattice $\bZ[X]$, 
for $X=X(K)$. 
It is straightforward that this construction generalises to 
an \'etale $k$-algebra $A$, in place of $K$. 
For example, if $A=K_1\times\cdots\times K_r$ 
(for $K_i/k$ finite separable field extensions), then 
$R_{A/k}(\bG_{m})=R_{K_1/k}(\bG_m) \times \cdots \times R_{K_r/k}(\bG_m)$. 

\subsection{Norm one tori $\T_K=R^{(1)}_{K/k}(\bG_m)$ of field extensions $K/k$ and $\T_A=R^{(1)}_{A/k}(\bG_m)$ of \'etale algebras $A/k$}\label{ssRatNorm1} 
A reference for this construction is Ono \cite[Section 1.4]{Ono61}. 
Keep notation of the preceding paragraphs. 
Let $A$ be an \'etale $k$-algebra. 
Consider the finite $G$-set $X:=X(A)$.

Via the duality between algebraic $k$-tori which split over $L$ 
and $G$-lattices where $G={\rm Gal}(L/k)$, 
the sequence $(F_X)$ 
as in Definition \ref{def8.1} 
yields an extension of algebraic $k$-tori, denoted by 
\begin{align*}
1 \to \T_A:=R^{(1)}_{A/k}(\bG_m) \to R_{A/k}(\bG_m) 
\xrightarrow{N_{A/k}} \bG_{m,k} \to 1.
\end{align*}
Here $N_{A/k}$ is the \textit{multi-norm} (or just norm). 
It is a homomorphism of algebraic $k$-tori. 
Its  effect on $k$-rational points is the (usual) norm  of the finite 
$k$-algebra $A$, also denoted by $N_{A/k}: A^\times \to k^\times$. 
The algebraic $k$-torus $\T_A=R^{(1)}_{A/k}(\bG_m)$ is the 
\textit{$($multi-$)$norm one torus} of the \'etale algebra $A/k$. 
Its character lattice $\widehat{T}$ is the $G$-lattice $J_X$. 
The function field $k(\T_A)$ of 
$\T_A$ is isomorphic to the fixed field $L(J_X)^G$. 

Similarly, the sequence $(E_X)$ as in Definition \ref{def8.1} 
yields an extension (of algebraic $k$-tori), that reads as 
\begin{align*}
1 \to \bG_{m,k} \xrightarrow{\iota} R_{A/k}(\bG_m) \to R_{A/k}(\bG_m)/\bG_{m,k} \to 1
\end{align*}
where $\iota$ is the natural embedding. 
Thus, the character lattice of the algebraic $k$-torus 
$R_{A/k}(\bG_m)/\bG_{m,k}$ is the $G$-lattice $I_X$ and 
the function field $k(R_{A/k}(\bG_m)/\bG_{m,k})$ is isomorphic 
to the fixed field $L(I_X)^G$. 

\subsubsection{Case of field extensions $K/k$} 
Assume that $A=K \subset \overline k$ is a field, 
with Galois closure $L \subset \overline k$. 
Define $G:=\Gal(L /k)$, $H:=\Gal(L /K)$ and $\T_K:=R^{(1)}_{K/k}(\bG_m)$. 
The $k$-variety $\T_K$ is biregularly isomorphic to the norm hypersurface 
\begin{align*}
f(x_1,\ldots,x_n)=1
\end{align*}
where 
$f(x_1,\ldots,x_n)\in k[x_1,\ldots,x_n]$ 
is the polynomial of total degree $n=\vert G/H \vert$ 
defined as 
\begin{align*}
f(x_1,\ldots,x_n)=N_{K/k}(x_1w_1+\cdots+x_nw_n)=\prod_{\overline g\in G/H}\overline{g}(x_1w_1+\cdots+x_nw_n) 
\end{align*}
with $\{w_1,\ldots,w_n\}$ a basis of $K/k$ and 
$N_{K/k}: K^\times\to k^\times$ is the norm map. 
%

\subsubsection{Case of \'etale algebras $A/k$}\label{ss272}
Let $A$ be an \'etale $k$-algebra with ${\rm dim}_k\, A=n$. 
Write  $A=\prod_{i=1}^r K_i$ where $K_i/k$ $(1\leq i\leq r)$ 
is a finite separable field extension 
with $[K_i:k]=n_i$. One has $\sum_{i=1}^r n_i=n$. 
Let $L_i/k$ $(1\leq i\leq r)$ 
be the Galois closure of $K_i/k$ in $\overline{k}$, 
with Galois groups $G_i={\rm Gal}(L_i/k)$, $H_i={\rm Gal}(L_i/K_i)$. 
Then we have $\bigcap_{\sigma\in G_i} H_i^\sigma=\{1\}$ 
where $H_i^\sigma=\sigma^{-1}H_i\sigma$ and hence 
$H_i$ contains no non-trivial normal subgroup of $G_i$. 
Let $L=L_1\cdots L_r\subset \overline{k}$ be the composite field of 
$L_1,\ldots,L_r$, 
i.e. the smallest field which contains all $L_i$, 
with Galois group $G={\rm Gal}(L/k)$. 
We see that $G$ is a subdirect product of $G_1,\ldots,G_r$, i.e. 
a subgroup $G\leq G_1\times\cdots\times G_r$ 
with surjections $\varphi_i: G\rightarrow G_i$ $(1\leq i\leq r)$. 
We may regard $X(A)={\rm Hom}(A,\overline{k})$ as a $G$-set, 
i.e. $G$ acts on $X(A)=\bigcup_{i=1}^r X(K_i)$ 
via the surjections $\varphi_i$. 
We have
$G\simeq\mathcal{G}/\mathcal{H}$ 
where  $\mathcal{H}={\rm Gal}(\overline{k}/L)$. 

When $r=2$, the set of all subdirect products of $G_1$, $G_2$ 
can be described explicitly, see Lemma \ref{lem2.2}. 

There is also an explicit description of the (multi-)norm one torus 
$\T_A=R^{(1)}_{A/k}(\bG_m)$ with $\widehat{\T}_A\simeq J_X$ 
as a norm hypersurface. 
It is obtained as in the field case. 
Just replace $N_{K/k}$ by $N_{K_1/k} \times\cdots\times N_{K_r/k}$. 
Namely, the norm one torus $\T_A=R^{(1)}_{A/k}(\bG_m)$ 
is biregularly isomorphic to the norm hypersurface 
\begin{align*}
\prod_{i=1}^rf_i(x_{i,1},\ldots,x_{i,n_i})=1
\end{align*}
where 
$f_i(x_{i,1},\ldots,x_{i,n_i})\in k[x_{i,1},\ldots,x_{i,n_i}]$ 
$(1\leq i\leq r)$ is the polynomial of total degree $n_i$ 
defined as 
\begin{align*}
f_i(x_{i,1},\ldots,x_{i,n_i})=N_{K_i/k}(x_{i,1}w_{i,1}+\cdots+x_{i,n_i}w_{i,n_i})=\prod_{\overline{g}\in G_i/H_i}\overline{g}(x_{i,1}w_{i,1}+\cdots+x_{i,n_i}w_{i,n_i})
\end{align*}
with $\{w_{i,1},\ldots,w_{i,n_i}\}$ a basis of $K_i/k$ 
and 
$N_{K_i/k}:K_i^\times\to k^\times$ is the norm map. 

\subsection{Rationality and variants}
Let $K$ and $K^\prime$ 
be finitely generated field extensions of $k$. \\The two fields 
$K$ and $K^\prime$ are called 
{\it stably $k$-isomorphic} if 
$K(y_1,\ldots,y_m)\simeq K^\prime(z_1,\ldots,z_n)$ over $k$, 
for some algebraically independent elements $y_1,\ldots,y_m$ over $K$ 
and $z_1,\ldots,z_n$ over $K^\prime$. \\
The field $K$ is called {\it rational over $k$} 
(or {\it $k$-rational} for short) 
if $K$ is purely transcendental over $k$, 
i.e. $K$ is isomorphic to $k(x_1,\ldots,x_n)$, 
the rational function field over $k$ with $n$ variables $x_1,\ldots,x_n$ 
for some integer $n$. \\
The field $K$ is called {\it stably $k$-rational} 
if it is stably isomorphic to a $k$-rational field. \\The field $K$ is {\it a direct factor of a $k$-rational field} 
(or {\it factor $k$-rational} for short) 
if there exists a field $K^\prime$ such that 
$K\otimes_k K^\prime$ is a domain whose quotient field is $k$-rational 
(see also Kang \cite[page 23]{Kan12}). \\
When $k$ is an infinite field, $K$ is called {\it retract $k$-rational} 
if there exists a $k$-algebra $R$ contained in $K$ such that 
(i) $K$ is the quotient field of $R$, and (ii) 
the identity map $1_R : R\rightarrow R$ factors through a localized 
polynomial ring over $k$. Condition (ii) means there exists 
an element $f\in k[x_1,\ldots,x_n]$ 
(polynomial ring over $k$) and  $k$-algebra 
homomorphisms $\varphi : R\rightarrow k[x_1,\ldots,x_n][1/f]$ 
and $\psi : k[x_1,\ldots,x_n][1/f]\rightarrow R$ satisfying 
$\psi\circ\varphi=1_R$ (see Saltman \cite[Definition 3.1, page 180]{Sal84}). \\
The field $K$ is called {\it $k$-unirational} 
if $k\subset K\subset k(x_1,\ldots,x_n)$ for some integer $n$. 
It is well-known that
\begin{center}
$k$-rational \ \ $\Rightarrow$\ \ 
stably $k$-rational\ \ $\Rightarrow$ \ \ 
factor $k$-rational\ \ $\Rightarrow$ \ \ 
retract $k$-rational\ \ $\Rightarrow$ \ \ 
$k$-unirational.
\end{center} 
For an 
integral $k$-variety $X$, 
one says that $X$ is $k$-rational if and only if 
its function field $K=k(X)$ is $k$-rational, and similarly 
for all variants recalled above. 
For example, $X$ is stably $k$-birational to $X'$ 
if and only if the function fields 
$K=k(X)$ and $K'=k(X')$ are stably $k$-isomorphic. 
We recall that 
$X$ is $k$-rational if and only if $X$ is birational 
to the projective space $\bP^n$ of dimension $n$ for some integer $n$. 
$X$ is stably $k$-rational if and only if $X\times \bP^m$ is $k$-rational 
for some integer $m$. 
$X$ is factor $k$-rational if and only if 
there exists an integral $k$-variety $Y$ 
such that $X\times Y$ is $k$-rational. 
$X$ is retract $k$-rational if and only if 
there exists a dominant rational map 
$f: \bP^n\dashrightarrow X$ for some integer $n$ and 
a rational map $g: X\dashrightarrow \bP^n$ such that 
$f\circ g=\Id_X$. 
$X$ is $k$-unirational if and only if 
there exists a dominant rational map $f: \bP^n\dashrightarrow X$ 
for some integer $n$. 
For equivalent definitions in the language of algebraic geometry, 
see e.g. 
Manin \cite{Man74}, \cite{Man86}, 
Manin and Tsfasman \cite{MT86}, 
Colliot-Th\'{e}l\`{e}ne and Sansuc \cite[Section 1]{CTS07}, 
Kang \cite{Kan12},
Beauville \cite{Bea16}, 
Merkurjev \cite[Section 3]{Mer17}. 

\subsection{Some properties of algebraic tori related to rationality}

Let $T$ be an algebraic $k$-torus with character lattice 
$\widehat T\simeq {\rm Hom}(T\times_k L,\bG_{m,L})$ where 
$L$ is the minimal splitting field of $T$. 
It is well-known that $T$ is $k$-unirational 
(see Voskresenskii \cite[page 40, Example 21]{Vos98}).

It is known that 
$T$ is retract $k$-rational 
if and only if 
 $T$ is factor $k$-rational 
if and only if 
there exists an algebraic $k$-torus $T^\prime$ 
such that $T\times T^\prime$ is $k$-rational 
by Colliot-Th\'{e}l\`{e}ne and Sansuc \cite[Proposition 7.4]{CTS87}. 
%

The flabby class $[M]^{fl}$ 
plays a crucial role in the rationality problem 
for $L(M)^G\simeq k(T)$ 
as follows 
(see Colliot-Th\'el\`ene and Sansuc \cite[Section 2]{CTS77}, 
\cite[Proposition 7.4]{CTS87}, 
Voskresenskii \cite[Section 4.6]{Vos98}, 
Kunyavskii \cite[Theorem 1.7]{Kun07}, 
Colliot-Th\'el\`ene \cite[Theorem 5.4]{CT07}, 
Hoshi and Yamasaki \cite[Section 1]{HY17}, \cite[Section 1]{HY1}). 

%
\begin{theorem}[{Voskresenskii \cite[Section 4, page 1213]{Vos69}, \cite[Section 3, page 7]{Vos70}, see also 
\cite{Vos74}, \cite[Section 4.6]{Vos98}, Kunyavskii \cite[Theorem 1.9]{Kun07} and 
Colliot-Th\'el\`ene \cite[Theorem 5.1, page 19]{CT07} for any field $k$}]\label{thVos69}
Let $k$ be a field 
and $\mathcal{G}={\rm Gal}(\overline{k}/k)$. 
Let $T$ be an algebraic $k$-torus, 
$X$ be a smooth $k$-compactification of $T$ 
and $\overline{X}=X\times_k\overline{k}$. 
Then there exists an exact sequence of $\mathcal{G}$-lattices 
\begin{align*}
0\to \widehat{T}\to \widehat{Q}\to {\rm Pic}\,\overline{X}\to 0
\end{align*}
where $\widehat{Q}$ is permutation 
and ${\rm Pic}\ \overline{X}$ is flabby. 
\end{theorem}

Let $L$ be the minimal splitting field of 
$T$ with Galois group $G={\rm Gal}(L/k)\simeq \mathcal{G}/\mathcal{H}$ 
where $\mathcal{G}={\rm Gal}(\overline{k}/k)$ 
and $\mathcal{H}={\rm Gal}(\overline{k}/L)$. 
By Theorem \ref{thVos69}, 
we obtain a flabby resolution of 
$\widehat{T}\simeq {\rm Hom}(T\times_k L,\bG_{m,L})$: 
\begin{align*}
0\to \widehat{T}\to \widehat{Q}\to {\rm Pic}\, X_L\to 0
\end{align*}
as $G$-lattices with $[\widehat{T}]^{fl}=[{\rm Pic}\ X_L]$ 
where $\widehat{Q}$ is permutation 
and $X_L=X\times_k L$ 
(see also Voskresenskii \cite[Section 1]{Vos74}). 
By the inflation-restriction exact sequence 
$0\to H^1(G,{\rm Pic}\, X_L)\xrightarrow{\rm inf} 
H^1(k,{\rm Pic}\,\overline{X})\xrightarrow{\rm res} 
H^1(L,{\rm Pic}\,\overline{X})$, 
we get 
${\rm inf}: H^1(G,{\rm Pic}\, X_L)\xrightarrow{\sim} 
H^1(k,{\rm Pic}\,\overline{X})$ 
because $H^1(L,{\rm Pic}\,\overline{X})=0$. 

%
\begin{theorem}
\label{th2.7}
Let $L/k$ be a finite Galois extension with Galois group $G={\rm Gal}(L/k)$ 
and $M$ and $M^\prime$ be $G$-lattices. 
Let $T$ and $T^\prime$ be algebraic $k$-tori with $\widehat{T}\simeq M$ 
and $\widehat{T}^\prime\simeq M^\prime$, 
i.e. $L(M)^G\simeq k(T)$ and $L(M^\prime)^G\simeq k(T^\prime)$.\\
{\rm (1)} $(${\rm Endo and Miyata} \cite[Theorem 1.6]{EM73}$)$ 
$[M]^{fl}=0$ if and only if $k(T)$ is stably $k$-rational.\\
{\rm (2)} $(${\rm Voskresenskii} \cite[Theorem 2]{Vos74}$)$ 
$[M]^{fl}=[M^\prime]^{fl}$ if and only if $k(T)$ and $k(T^\prime)$ 
are stably $k$-isomorphic.\\
{\rm (3)} $(${\rm Saltman} \cite[Theorem 3.14]{Sal84}$)$ 
$[M]^{fl}$ is invertible if and only if $k(T)$ is 
retract $k$-rational.
\end{theorem}
%

\begin{remark}
For algebraic $k$-tori, retract $k$-rationality is easier to decide 
than stable $k$-rationality. 
By Theorem \ref{th2.7}, retract $k$-rationality (resp. stable $k$-rationality) 
amounts to checking whether some $G$-lattice $F$ is invertible 
(resp. stably permutation). 
For $G$-lattices, invertibility is `local', in the sense that $F$ is 
$G$-invertible if and only if $F$ is $G_p$-invertible 
for every $p$-Sylow subgroup $G_p\leq G$ 
(see Endo and Miyata \cite[Lemma 1.4]{EM75}). 
Assuming that $F$ is invertible, $F$ being stably permutation is then 
a much stronger property, which is not local. 
Actually, there is no known general procedure to check it. 
In specific situations (e.g. for norm one tori in Theorem \ref{thS}), 
stable $k$-rationality can be disproved 
by computing relevant cohomology groups. 
\end{remark}
\begin{remark}
There is no known criterion for an algebraic $k$-torus $T$ to be $k$-rational.  The fact that $T$ is (or not) stably/retract $k$-rational, is determined 
by the $\mathcal G$-lattice $\widehat T$. 
To our knowledge, it is very much unclear that such should be the case 
for $k$-rationality. 
Voskresenskii conjectured that stably $k$-rational tori should be 
$k$-rational. 
While no counter-examples are known, 
there is no evidence to support this widely open conjecture.
\end{remark}

If $M_1\oplus M_2\simeq M_3$ is an exact sequence of $G$-lattices, 
then $L(M_2)^G\simeq L(M_1\oplus M_3)^G$, i.e. 
$k(T_2)\simeq k(T_1\times T_3)$ with $\widehat{T}_i=M_i$ $(i=1,2,3)$, 
this means that $T_2$ and $T_1\times T_3$ are birationally $k$-equivalent. 
In particular, we have $[M_2]^{fl}=[M_1]^{fl}+[M_3]^{fl}$. 

In general, it follows from 
Lenstra \cite[Proposition 1.5]{Len74} that 
if $0\to M_1\to M_2\to M_3\to 0$ is an exact sequence of $G$-lattices 
with $M_3$ invertible, 
then  
$L(M_2)^G\simeq L(M_1\oplus M_3)^G$, i.e. 
$k(T_2)\simeq k(T_1\times T_3)$ with $\widehat{T}_i=M_i$ $(i=1,2,3)$ 
(see also Ono \cite[Proposition 1.2.2]{Ono63}, 
Endo and Miyata \cite[Theorem 1.6, Proposition 1.10]{EM73}, 
Swan \cite[Lemma 3.1]{Swa10}, 
Hoshi, Kang and Kitayama \cite[Proof of Theorem 6.5]{HKK14}). 
In particular, we have $[M_2]^{fl}=[M_1]^{fl}+[M_3]^{fl}$. 

\begin{proposition}[{Endo and Miyata \cite[Proposition 1.1, Corollary 1.3]{EM73}, Lenstra \cite[Proposition 1.5]{Len74}}]\label{prop2.8}
Let $M$ be a $G$-lattice and 
$0\to M\to P\to F\to 0$ be a flabby resolution of $M$ 
with $P$ permutation and $F$ flabby. 
Let $T$ and $S$ be algebraic $k$-tori with $\widehat{T}\simeq M$, 
$\widehat{S}\simeq F$.\\
{\rm (1)} 
{\rm (Endo and Miyata \cite{EM73})}
If $F$ is stably permutation, then $T$ is stably $k$-rational;\\
{\rm (2)} 
{\rm (Lenstra \cite{Len74})} 
If $F$ is invertible, then $S\times T$ and $T_P$ are 
birationally $k$-equivalent 
where $T_P$ is an algebraic $k$-torus which is $k$-rational 
with 
$\widehat{T}_P\simeq P$ and $[P]^{fl}=0$. 
In particular, $S\times T$ is $k$-rational with 
$[\widehat{S\times T}]^{fl}=[\widehat{T}]^{fl}+[\widehat{S}]^{fl}
=[M]^{fl}+[F]^{fl}=[P]^{fl}=0$. 
\end{proposition}
%

Let $G$ be a finite group and $M$ be a $G$-lattice. We define 
\begin{align*}
\Sha^i_\omega(G,M):={\rm Ker}\left\{H^i(G,M)\xrightarrow{{\rm res}}\bigoplus_{g\in G}H^i(\langle g\rangle,M)\right\}\quad (i\geq 1) .
\end{align*}

Let $X$ be a smooth $k$-compactification of $T$, 
i.e. smooth projective $k$-variety $X$ 
containing $T$ as a dense open subvariety, 
and set $\overline{X}=X\times_k\overline{k}$ and $X_K=X\times_k K$. 
There exists such a smooth $k$-compactification of an algebraic $k$-torus $T$ 
over any field $k$, that can be  constructed explicitly,  without the general theory of resolution of singularities  (see Colliot-Th\'{e}l\`{e}ne, Harari and Skorobogatov 
\cite[Corollaire 1]{CTHS05}). 

\begin{theorem}[{Colliot-Th\'{e}l\`{e}ne and Sansuc \cite[Proposition 9.5 (ii)]{CTS87}, see also \cite[Proposition 9.8]{San81}, \cite[page 98]{Vos98}, \cite[Corollaire 1]{CTHS05}, \cite[Theorem 2.3]{BP20}}]\label{thCTS87}
Let $k$ be a field 
and $K/k$ be a finite Galois extension with Galois group $G={\rm Gal}(K/k)$. 
Let $T$ be an algebraic $k$-torus which splits over $K$ and 
$X$ be a smooth $k$-compactification of $T$. 
Then we have 
\begin{align*}
\Sha^2_\omega(G,\widehat{T})\simeq 
H^1(G,{\rm Pic}\, X_K)\simeq {\rm Br}(X)/{\rm Br}(k)
\end{align*}
where 
${\rm Br}(X)$ is the Brauer group of $X$ (in \'etale cohomology). 
\end{theorem}

In other words, 
we have 
$H^1(k,{\rm Pic}\, \overline{X})\simeq H^1(G,{\rm Pic}\, X_K)\simeq 
H^1(G,[\widehat{T}]^{fl})\simeq \Sha^2_\omega(G,\widehat{T})\simeq {\rm Br}(X)/{\rm Br}(k)$. 
Of interest is also the unramified Brauer group ${\rm Br}_{\rm nr}(k(X)/k)={\rm Br}(X)\subset {\rm Br}(k(X))$ 
(see 
Colliot-Th\'{e}l\`{e}ne and Skorobogatov 
\cite[Proposition 6.2.7]{CTS21}). 
In particular, 
for an algebraic $k$-torus $T$ which splits over $K$, 
with character lattice $\widehat{T}\simeq {\rm Hom}(T\times_k K,\bG_{m,K})$, 
we have: 
\begin{align*}
T\ \textrm{is\ retract}\ k\textrm{-rational}\ \ 
\Rightarrow\ \ 
\Sha^2_\omega(G,\widehat{T})=0. 
\end{align*}

\section{Known results of stably/retract $k$-rational algebraic $k$-tori $T$}\label{S3}

\subsection{Case of invertible character lattices $\widehat{T}\simeq {\rm Hom}(T\times_k L,\bG_{m,L})$}

Simple examples of 
rational $k$-tori $T$ with ${\rm dim}\, T=n$ 
are $T=R_{A/k}(\bG_m)$, for an $n$-dimensional \'etale $k$-algebra $A$. 
Setting $X:=X(A)$, one then has $\widehat{T}\simeq \bZ[X]$ and $|X|=n$. 
In other words, $\widehat{T}\simeq  \oplus_{i=1}^r\bZ[G/H_i]$ 
for various subgroups $H_i\leq G\leq S_{n_i}$, with $[G:H_i]=n_i$ 
and $n=\sum_{i=1}^r n_i$. 
There exists exactly $1$ 
(resp. $2$, $4$, $11$, $19$, $56$) such $k$-tori $T=R_{A/k}(\bG_m)$ 
in dimension $1$ (resp. $2$, $3$, $4$, $5$, $6$) 
(see Hoshi and Yamasaki \cite[Theorem 6.3, Table 8]{HY17}). 
Hoshi and Yamasaki \cite[Theorem 6.3, Table 8]{HY17}
also determine all of the stably permutation $G$-lattice $M_G$ with 
${\rm rank}_\bZ\, M_G\leq 6$ as in Theorem \ref{th6.16}. 

Let $S_n$ (resp. $A_n$, $D_n$, $C_n$) be the symmetric 
(resp. the alternating, the dihedral, the cyclic) group 
of degree $n$ of order $n!$ (resp. $n!/2$, $2n$, $n$). 
Let $Q_{4m}$ be the 
generalized quaternion group of order $4m$ $(m\geq 2)$. 
Let $p$ be a prime number and 
$F_{pl}\simeq C_p\rtimes C_l$ $(2<l\mid p-1)$ 
be the Frobenius group of order $pl$ where $l\mid p-1$ 
means that $l$ is a (positive) divisor of $p-1$. 

\begin{theorem}[{Hoshi and Yamasaki \cite[Theorem 6.2, Theorem 6.3, Table 8]{HY17}}]\label{th6.16}
Let $G$ be a finite subgroup of $\GL(n,\bZ)$ and 
$M_G$ be the $G$-lattice as in Definition \ref{defMG}. 
For $1\leq n\leq 6$, the following conditions are equivalent:\\
{\rm (1)} $M_G$ is stably permutation, i.e. $[M_G]=0$;\\
{\rm (2)} $M_G$ is invertible;\\ 
{\rm (3)} $M_G$ is flabby and coflabby.\\ 
Indeed, there exist $1=1+0$ $($resp. $2=2+0$, $4=4+0$, 
$15=11+4$, $23=19+4$, $106=56+50$$)$ 
stably permutation $($resp. invertible, flabby and coflabby$)$ 
$G$-lattices of rank $n=1$ $($resp. $2$, $3$, $4$, $5$, $6$$)$. 
%
They consist of $1$ $($resp. $2$, $4$, $11$, $19$, $56$$)$ 
permutation $G$-lattices $M_G\simeq \oplus_{i=1}^r\bZ[G/H_i]$ 
where $H_i\leq G\leq S_n$ with $[G:H_i]=n_i$ of rank $n=\sum_{i=1}^r n_i$ 
with $n=1$ $($resp. $2$, $3$, $4$, $5$, $6$$)$ which are 
given as in 
Hoshi and Yamasaki \cite[Example 6.1]{HY17} 
and 
$0$ $($resp. $0$, $0$, $4$, $4$, $50$$)$
not permutation but 
stably permutation $($resp. invertible, flabby and coflabby$)$ 
$G$-lattices of rank $n=1$ $($resp. $2$, $3$, $4$, $5$, $6$$)$ which are 
given as in 
Hoshi and Yamasaki \cite[Table 8]{HY17}. 
\end{theorem}

The corresponding algebraic $k$-torus $T$ with 
$\widehat{T}\simeq M_G$ and ${\rm dim}_k\, T=n$ as in Theorem \ref{th6.16} 
is $k$-rational (resp. stably $k$-rational) if $M_G$ is permutation 
(resp. $M_G$ is stably permutation). 

\subsection{Case of dim $T\leq 5$} 
It is easy to see that all  $1$-dimensional algebraic $k$-tori $T$ 
are $k$-rational. 
These are: the trivial torus $\bG_{m,k}$ and the norm one torus 
$\T_K=R^{(1)}_{K/k}(\bG_{m,K})$ of a separable quadratic field extension $K/k$.

Voskresenskii \cite{Vos67} showed that 
all  $2$-dimensional algebraic $k$-tori $T$ are $k$-rational. 
There exist $13$ cases of algebraic $k$-tori 
of dimension $2$ which corresponds to 
$\bZ$-class of subgroups (up to conjugacy) $G\leq \GL(2,\bZ)$. 
Kunyavskii \cite{Kun90} gave a 
rational (stably rational, retract rational) classification of 
$3$-dimensional algebraic $k$-tori: 
among $73$ cases, 
there exist $58$ (resp. $15$) 
$k$-rational (resp. not retract $k$-rational) cases in dimension $3$. 
\begin{theorem}[{Kunyavskii \cite[Theorem 1, Theorem 2]{Kun90}, see Kang \cite[page 25,  fifth paragraph]{Kan12} and Hoshi and Yamasaki \cite[Lemma 7.3]{HY1} for the last statement}]\label{thKu}
Let $L/k$ be a finite Galois extension with Galois group 
$G={\rm Gal}(L/k)\leq \GL(3,\bZ)$ 
and $M_G$ be the $G$-lattice as in Definition \ref{defMG}. 
Let $T$ be an algebraic $k$-torus with $\widehat{T}\simeq M_G$ 
and $\dim_k\, T=3$. 
Then $k(T)\simeq L(M_G)^G$ is not $k$-rational 
if and only if 
$G$ is conjugate to one of the $15$ groups which are given 
as in {\rm Table} $1$. 
Moreover, if $k(T)\simeq L(M_G)^G$ is not $k$-rational, 
then it is not retract $k$-rational, 
i.e. $[M_G]^{fl}$ is not invertible. 
\end{theorem}
%
\begin{center}
Table $1$: $L(M)^G$ not retract $k$-rational, 
rank $M=3$, $M$: indecomposable ($15$ cases)\vspace*{2mm}\\
\fontsize{9pt}{12pt}\selectfont
\begin{tabular}{lll} 
${}^tG$ in \cite{Kun90} & GAP ID& $G$ \\\hline
$U_1$ & $(3,3,1,3)$ & $C_2^2$\\
$U_2$ & $(3,3,3,3)$ & $C_2^3$\\
$U_3$ & $(3,4,4,2)$ & $D_4$\\
$U_4$ & $(3,4,6,3)$ & $D_4$\\
$U_5$ & $(3,7,1,2)$ & $A_4$
\end{tabular}\quad
\begin{tabular}{lll}
${}^tG$ in \cite{Kun90} & GAP ID& $G$ \\\hline
$U_6$ & $(3,4,7,2)$ & $D_4\times C_2$\\
$U_7$ & $(3,7,2,2)$ & $A_4\times C_2$\\
$U_8$ & $(3,7,3,3)$ & $S_4$\\
$U_9$ & $(3,7,3,2)$ & $S_4$\\
$U_{10}$ & $(3,7,4,2)$ & $S_4$
\end{tabular}\quad
\begin{tabular}{lll} 
${}^tG$ in \cite{Kun90} & GAP ID& $G$ \\\hline
$U_{11}$ & $(3,7,5,3)$ & $S_4\times C_2$\\
$U_{12}$ & $(3,7,5,2)$ & $S_4\times C_2$\\
$W_1$ & $(3,4,3,2)$ & $C_4\times C_2$\\
$W_2$ & $(3,3,3,4)$ & $C_2^3$\\
$W_3$ & $(3,7,2,3)$ & $A_4\times C_2$
\end{tabular}
\end{center}

For the GAP ID, see Hoshi and Yamasaki \cite[Chapter 3]{HY17}, 
see also Hoshi and Yamasaki \cite[Theorem 1.2, Table 1, 
the second paragraph in page 3, Theorem 13.4, Table 10]{HY17}. 
In particular, if we adopt the action of $G$ as in Definition \ref{defMG}, 
we should take the transpose ${}^tG$ of the matrix group $G$ 
as in Kunyavskii \cite{Kun90}. 

There exist $710$ (resp. $6079$) cases of algebraic $k$-tori 
of dimension $4$ (resp. $5$) which corresponds to 
$\bZ$-class of subgroups (up to conjugacy) 
$G\leq \GL(4,\bZ)$ (resp. $\GL(5,\bZ)$)). 
Hoshi and Yamasaki \cite[Theorem 1.9 and Theorem 1.12]{HY17} 
gave a stably rational (retract rational) classification of 
algebraic $k$-tori of dimension $4$ and $5$: 
\begin{theorem}[{Hoshi and Yamasaki \cite[Theorem 1.9]{HY17}}]\label{th1M}
Let $L/k$ be a finite Galois extension with Galois group 
$G={\rm Gal}(L/k)\leq \GL(4,\bZ)$  
and $M_G$ be the $G$-lattice as in Definition \ref{defMG}. 
Let $T$ be an algebraic $k$-torus with $\widehat{T}\simeq M_G$ 
and $\dim_k\, T=4$. 
Then\\ 
{\rm (1)} $k(T)\simeq L(M_G)^G$ is stably $k$-rational, 
i.e. $[M_G]^{fl}=0$, if and only if 
$G$ is conjugate to one of the $487$ groups which are not in 
\cite[{\rm Tables} $2$, $3$ and $4$]{HY17}.\\
{\rm (2)} $k(T)\simeq L(M_G)^G$ is not stably but retract $k$-rational, 
i.e. $[M_G]^{fl}$ is not zero but invertible, 
if and only if 
$G$ is conjugate to one of the $7$ groups which are 
given as in \cite[{\rm Table} $2$]{HY17}.\\
{\rm (3)} $k(T)\simeq L(M_G)^G$ is not retract $k$-rational, 
i.e. $[M_G]^{fl}$ is not invertible, 
if and only if 
$G$ is conjugate to one of the $216$ groups which are given as 
in \cite[{\rm Tables} $3$ and $4$]{HY17}. 
\end{theorem}
\begin{theorem}[{Hoshi and Yamasaki \cite[Theorem 1.12]{HY17}}]\label{th2M}
Let $L/k$ be a finite Galois extension with Galois group 
$G={\rm Gal}(L/k)\leq \GL(5,\bZ)$ 
and $M_G$ be the $G$-lattice as in Definition \ref{defMG}. 
Let $T$ be an algebraic $k$-torus with $\widehat{T}\simeq M_G$ 
and $\dim_k\, T=5$. 
Then\\ 
{\rm (1)} $k(T)\simeq L(M_G)^G$ is stably $k$-rational, 
i.e. $[M_G]^{fl}=0$, if and only if 
$G$ is conjugate to one of the $3051$ groups which are not in 
\cite[{\rm Tables} $11$, $12$, $13$, $14$ and $15$]{HY17}.\\
{\rm (2)} $k(T)\simeq L(M_G)^G$ is not stably but retract $k$-rational, 
i.e. $[M_G]^{fl}$ is not zero but invertible, 
if and only if 
$G$ is conjugate to one of the $25$ groups which are given as 
in \cite[{\rm Table} $11$]{HY17}.\\
{\rm (3)} $k(T)\simeq L(M_G)^G$ is not retract $k$-rational, 
i.e. $[M_G]^{fl}$ is not invertible, 
if and only if 
$G$ is conjugate to one of the $3003$ groups which are given as 
in \cite[{\rm Tables} $12$, $13$, $14$ and $15$]{HY17}. 
\end{theorem}

\subsection{Norm one tori $\T_K=R^{(1)}_{K/k}(\bG_m)$ of field extensions $K/k$}
The rationality problem for norm one tori $\T_K=R^{(1)}_{K/k}(\bG_m)$ 
with $\widehat{\T}_K\simeq J_{G/H}$ and ${\rm dim}_k\, \T_K=n-1$ 
is investigated 
by Endo and Miyata \cite{EM75}, 
Colliot-Th\'{e}l\`{e}ne and Sansuc \cite{CTS77}, \cite{CTS87}, 
H\"{u}rlimann \cite{Hur84}, 
Le Bruyn \cite{LeB95}, 
Cortella and Kunyavskii \cite{CK00}, 
Lemire and Lorenz \cite{LL00}, 
Endo \cite{End11}, 
Hoshi and Yamasaki \cite{HY17}, 
\cite{HY21}, \cite{HY24}, \cite{HY1}, \cite{HY2}, 
Hasegawa, Hoshi and Yamasaki \cite{HHY20}. 
Let us recall some significant results.

\begin{theorem}[{Endo and Miyata \cite[Theorem 1.5]{EM75}, Saltman \cite[Theorem 3.14]{Sal84}}]\label{th1.11}
Let $K/k$ be a finite Galois field extension with Galois group 
$G={\rm Gal}(K/k)$. 
Then the following conditions are equivalent:\\
{\rm (1)} $R^{(1)}_{K/k}(\bG_m)$ is retract $k$-rational;\\
{\rm (2)} All the Sylow subgroups of $G$ are cyclic. 
\end{theorem}

\begin{remark}\label{rem3.6}
If all the Sylow subgroups of $G$ are cyclic, 
then $G$ is metacyclic 
(see e.g. Gorenstein \cite[Theorem 6.2, page 258]{Gor80}) 
and can be presented as 
\begin{align*}
G=\langle a,b\mid a^m=1, b^n=1, bab^{-1}=a^r\rangle\simeq C_m\rtimes C_n
\end{align*}
where 
$r^n\equiv 1\ ({\rm mod}\ m)$, $m$ is odd, $0\leq r<m$, 
$\gcd(m,n)=\gcd(m,r-1)=1$ 
(see Zassenhaus \cite[Satz 5, page 198]{Zas35}, 
see also Hall \cite[Theorem 9.4.3]{Hal59}, 
Robinson \cite[10.1.10, page 290]{Rob96}). 
\end{remark}

Let $\gcd(m,n)$ be the greatest common divisor of integers $m,n$. 

\begin{theorem}[{Endo and Miyata \cite[Theorem 2.3]{EM75}, Colliot-Th\'{e}l\`{e}ne and Sansuc \cite[Proposition 3]{CTS77}}]\label{th1.12}
Let $K/k$ be a finite Galois field extension with Galois group 
$G={\rm Gal}(K/k)$. 
The following conditions are equivalent:\\
{\rm (1)} $R^{(1)}_{K/k}(\bG_m)$ is stably $k$-rational;\\
{\rm (2)} all  Sylow subgroups of $G$ are cyclic and $H^4(G,\bZ)\simeq \widehat H^0(G,\bZ)$ 
where $\widehat H$ denotes Tate cohomology;\\
{\rm (3)} $G\simeq C_m$ or $G\simeq C_n\times \langle\sigma,\tau\mid\sigma^k=\tau^{2^d}=1,
\tau\sigma\tau^{-1}=\sigma^{-1}\rangle$ where $d\geq 1, k\geq 3$, 
$n,k$: odd, and $\gcd(n,k)=1$;\\
{\rm (4)} $G=\langle s,t\mid s^m=t^{2^d}=1, tst^{-1}=s^r, m: odd,\ 
r^2\equiv 1\pmod{m}\rangle$.
\end{theorem}
\begin{theorem}[Endo {\cite[Theorem 2.1]{End11}}]\label{th1.13}
Let $K/k$ be a finite non-Galois, separable field extension 
and $L/k$ be the Galois closure of $K/k$ 
with Galois group $G={\rm Gal}(L/k)$. 
If $G$ is nilpotent, then 
$R^{(1)}_{K/k}(\bG_m)$ is not retract $k$-rational.
\end{theorem}
\begin{theorem}[Endo {\cite[Theorem 3.1]{End11}}]\label{th1.14}
Let $K/k$ be a finite non-Galois, separable field extension 
and $L/k$ be the Galois closure of $K/k$ with Galois groups 
$G={\rm Gal}(L/k)$ and $\{1\}\lneq H={\rm Gal}(L/K)\lneq G$. 
If all  Sylow subgroups of $G$ are cyclic, then 
$R^{(1)}_{K/k}(\bG_m)$ is retract $k$-rational, 
and the following conditions are equivalent:\\
{\rm (1)} 
$R^{(1)}_{K/k}(\bG_m)$ is stably $k$-rational;\\
{\rm (2)} 
$G\simeq D_n$ with $n$ odd $(n\geq 3)$ 
or $G\simeq C_m\times D_n$ where $m,n$ are odd, 
$m,n\geq 3$, $\gcd(m,n)=1$, and $H\leq D_n$ is of order $2$;\\
{\rm (3)} 
$H\simeq C_2$ and $G\simeq C_r\rtimes H$, $r\geq 3$ odd, where 
$H$ acts non-trivially on $C_r$. 
\end{theorem}
\begin{remark}\label{rem15}
Endo \cite[page 84, line 7]{End11} explained that 
the retract $k$-rationality of $R^{(1)}_{K/k}(\bG_m)$ 
in Theorem \ref{th1.14} 
is already obtained in Endo and Miyata \cite[Theorem 1.5]{EM75} 
and Saltman \cite[Theorem 3.14]{Sal84}. 
Indeed, we can check that $F=[J_{G/H}]^{fl}$ is invertible 
 by combining the following results:\\
(1) (\cite[Lemma 1.4]{EM75}) $F$ is invertible if and only if $F|_{G_p}$ is invertible for any $p\mid |G|$ 
where $G_p$ is a $p$-Sylow subgroup of $G$;\\
(2) (\cite[Theorem 1.5: $(1)\Leftrightarrow (2)$]{EM75}) $F|_{G_p}$ is invertible if and only if $F|_{G_p}$ is coflabby 
because $G_p$ is cyclic;\\
(3) $F|_{G_p}$ is coflabby. 
Indeed, since $G_p$ is cyclic, it is standard that 
$H^1(H,F)=\widehat{H}^{-1}(H,F)$ for any $H\leq G_p$ 
(see e.g. Neukirch, Schmidt and Wingberg \cite[Theorem 1.6.12, page 68]{NSW00}). 
\end{remark}

When $G\simeq S_n$ or $A_n$ and $[G:H]=[K:k]=n$, we have:
\begin{theorem}[{Colliot-Th\'{e}l\`{e}ne and Sansuc \cite[Proposition 9.1]{CTS87}, \cite[Theorem 3.1]{LeB95}, \cite[Proposition 0.2]{CK00}, \cite{LL00}, Endo \cite[Theorem 4.1]{End11}, see also \cite[Remark 4.2 and Theorem 4.3]{End11}}]\label{thS}
Let $K/k$ be a non-Galois separable field extension 
of degree $n$ and $L/k$ be the Galois closure of $K/k$. 
Assume that ${\rm Gal}(L/k)=S_n$, $n\geq 3$, 
and ${\rm Gal}(L/K)=S_{n-1}$ is the stabilizer of one of the letters 
in $S_n$. 
Then we have:\\
{\rm (1)}\ 
$R^{(1)}_{K/k}(\bG_m)$ is retract $k$-rational 
if and only if $n$ is a prime number;\\
{\rm (2)}\ 
$R^{(1)}_{K/k}(\bG_m)$ is $($stably$)$ $k$-rational 
if and only if $n=3$.
\end{theorem}
\begin{theorem}[Endo {\cite[Theorem 4.4]{End11}, Hoshi and Yamasaki \cite[Corollary 1.11]{HY17}}]\label{th1.17}
Let $K/k$ be a non-Galois separable field extension 
of degree $n$ and $L/k$ be the Galois closure of $K/k$. 
Assume that ${\rm Gal}(L/k)=A_n$, $n\geq 4$, 
and ${\rm Gal}(L/K)=A_{n-1}$ is the stabilizer of one of the letters 
in $A_n$. 
Then we have:\\
{\rm (1)}\ 
$R^{(1)}_{K/k}(\bG_m)$ is retract $k$-rational 
if and only if $n$ is a prime number.\\
{\rm (2)}\ $R^{(1)}_{K/k}(\bG_m)$ is stably $k$-rational 
if and only if $n=5$.
\end{theorem}

When $G\leq S_p$ with prime $[G:H]=[K:k]=p$, we have:

\begin{theorem}[{Hoshi and Yamasaki \cite[Theorem 1.9]{HY21}, see Theorem \ref{th1.11} and Theorem \ref{th1.14} for (1), see Theorem \ref{th1.17} for (4)}]\label{th1.18}
Let $p\geq 3$ be a prime number, 
$K/k$ be a separable field extension of degree $p$ 
and $L/k$ be the Galois closure of $K/k$. 
Let $G={\rm Gal}(L/k)$ be a transitive subgroup of $S_p$ 
and $H={\rm Gal}(L/K)$ with $[G:H]=p$. 
Then norm one tori $\T_K=R_{K/k}^{(1)}(\bG_m)$ with 
$\widehat{\T}_K\simeq J_{G/H}$ and $\dim_k\, \T_K=p-1$ 
are retract $k$-rational, and 
a stably rational classification of $\T_K$ is given as follows:\\ 
{\rm (1)} $\T_K$ is stably $k$-rational 
for $G\simeq C_p\leq S_p$ and $G\simeq D_p\leq S_p$;\\
{\rm (2)} $\T_K$ is not stably $k$-rational 
for $G\simeq C_p\rtimes C_m\leq S_p$ with $3\leq m\mid p-1$;\\
{\rm (3)} 
$\T_K$ is not stably $k$-rational for $G\simeq S_p$ where $p\geq 5$;\\
{\rm (4)} $\T_K$ is stably $k$-rational for $G\simeq A_5\leq S_5$ 
and $\T_K$ is not stably $k$-rational for $G\simeq A_p\leq S_p$ where $p\geq 7$;\\
{\rm (5)} $\T_K$ is not stably $k$-rational 
for $G\simeq \PSL_2(\bF_{11})\leq S_{11}$;\\
{\rm (6)} $\T_K$ is not stably $k$-rational 
for $G\simeq M_{11}\leq S_{11}$ and $G\simeq M_{23}\leq S_{23}$;\\
{\rm (7)} 
$\T_K$ is not stably $k$-rational 
for $\PSL_d(\bF_q)\leq G\leq 
{\rm P\Gamma L}_d(\bF_q)\simeq \PGL_d(\bF_q)\rtimes C_e$
where $d\geq  3$, $p=\frac{q^d-1}{q-1}$ and $q=l^e$ is a prime power;\\
{\rm (8)} 
$\T_K$ is not stably $k$-rational 
for $\PSL_2(\bF_{2^e})< G\leq 
{\rm P\Gamma L}_2(\bF_{2^e})\simeq \PSL_2(\bF_{2^e})\rtimes C_e$ 
where $p=2^e+1$ is a Fermat prime. 
\end{theorem}
\begin{remark}
We do not know whether $\T_K$ is stably $k$-rational 
in the case {\rm (8)} in Theorem \ref{th1.18} when 
$G=\PSL_2(\bF_{2^e})$ and $p\geq 17$. 
Note that for Fermat primes $p=3$ and $5$, 
$\T_K$ is stably $k$-rational for $G=\PSL_2(\bF_{2^e})$ 
by Theorem \ref{th1.18} (1), (4) (note that 
$\PSL_2(\bF_2)\simeq D_3\simeq S_3$, 
$\PSL_2(\bF_4)=\PGL_2(\bF_4)\simeq A_5$).
\end{remark}

When $G\leq S_{2^e}$ with $[G:H]=[K:k]=2^e$ $(e\geq 1)$, we have:
\begin{theorem}[{Hasegawa, Hoshi and Yamasaki \cite[Theorem 1.1]{HHY20}}]
Let $K/k$ be a separable field extension of degree $n$ 
and $L/k$ be the Galois closure of $K/k$. 
Let $G={\rm Gal}(L/k)$ be a transitive subgroup of $S_n$ 
where $n=2^e$ $(e\geq 1)$ 
and $H={\rm Gal}(L/K)$ with $[G:H]=n$. 
Then $R_{K/k}^{(1)}(\bG_m)$ is stably $k$-rational 
if and only if $G\simeq C_n$. 
Moreover, if $R_{K/k}^{(1)}(\bG_m)$ is not stably $k$-rational, 
then it is not retract $k$-rational. 
\end{theorem}

Some special cases where $G\leq S_n$ are also given in Hasegawa, Hoshi and Yamasaki \cite{HHY20}: 

\begin{theorem}[{Hasegawa, Hoshi and Yamasaki \cite[Theorem 1.3]{HHY20}}]
Let $K/k$ be a separable field extension of degree $n$ 
and $L/k$ be the Galois closure of $K/k$. 
Let $G={\rm Gal}(L/k)$ be a transitive subgroup of $S_n$ 
and $H={\rm Gal}(L/K)$ with $[G:H]=n$. 
Assume that $n=q+1$ where $q=l^e\equiv 1\pmod{4}$ 
is an odd prime power 
and 
$\PSL_2(\bF_q)\leq G\leq {\rm P\Gamma L}_2(\bF_q)\simeq 
\PGL_2(\bF_q)\rtimes C_e$. 
Then $R_{K/k}^{(1)}(\bG_m)$ is not retract $k$-rational. 
\end{theorem}
\begin{theorem}[{Hasegawa, Hoshi and Yamasaki \cite[Theorem 1.4]{HHY20}}]
Let $p$ be a prime number, 
$K/k$ be a separable field extension of degree $2p$ 
and $L/k$ be the Galois closure of $K/k$. 
Assume that $G={\rm Gal}(L/k)$ is a primitive 
subgroup of $S_{2p}$ 
and $H={\rm Gal}(L/K)$ with $[G:H]=2p$. 
Then $R_{K/k}^{(1)}(\bG_m)$ is not retract $k$-rational. 

More precisely, $R_{K/k}^{(1)}(\bG_m)$ is 
not retract $k$-rational for the following primitive 
groups $G\leq S_{2p}$:\\ 
{\rm (1)} $G=S_{2p}$ or $G=A_{2p}\leq S_{2p}$;\\
{\rm (2)} $G=S_5\leq S_{10}$ or $G=A_5\leq S_{10}$;\\
{\rm (3)} $G=M_{22}\leq S_{22}$ 
or $G=\Aut(M_{22})\simeq M_{22}\rtimes C_2\leq S_{22}$ 
where $M_{22}$ is the Mathieu group of degree $22$;\\ 
{\rm (4)} $\PSL_2(\bF_q)\leq G\leq {\rm P\Gamma L}_2(\bF_q)\simeq 
\PGL_2(\bF_q)\rtimes C_e$ 
where $2p=q+1$ and $q=l^e$ is an odd prime power.
\end{theorem}
%
Let $nTm$ be the $m$-th transitive subgroup of the symmetric group $S_n$ 
of degree $n$ up to conjugacy (see Butler and McKay \cite{BM83}, \cite{GAP}).
The number of transitive subgroups $nTm$ of $S_n$ 
$(2\leq n\leq 16)$ up to conjugacy is given as follows 
(see Butler and McKay \cite{BM83} for $n\leq 11$, 
Royle \cite{Roy87} for $n=12$, 
Butler \cite{But93} for $n=14,15$, 
Hulpke \cite[Tabelle 1]{Hul96}, \cite[Table 1]{Hul05} for $n=16$ 
and \cite{GAP}):\\
\begin{center}
\begin{tabular}{r|rrrrrrrrrrrrrrr}
$n$ & $2$ & $3$ & $4$ & $5$ & $6$ & $7$ & $8$ & $9$ & $10$ & $11$ & 
$12$ & $13$ & $14$ & $15$ & $16$\\\hline
$\#$ of $nTm$ & $1$ & $2$ & $5$ & $5$ & $16$ & $7$ & $50$ & $34$ & $45$ & $8$ &
$301$ & $9$ & $63$ & $104$ & $1954$\vspace*{3mm}\\
\end{tabular}
\end{center}

A classification of stably/retract rational norm one tori 
$\T_K=R^{(1)}_{K/k}(\bG_m)$ in dimension $n-1$ where $[K:k]=n\leq 16$
is given completely as follows:
%
\begin{theorem}[{Hoshi and Yamasaki \cite[Theorem 1.10, Theorem 1.14, Theorem 8.5]{HY17} for $n=5,6,7,11$, \cite[Theorem 1.9, Theorem 1.11]{HY21} for $n=8,9,10,13$, \cite[Theorem 1.1]{HY2} for the stable $k$-rationality of $G=9T27\simeq \PSL_2(\bF_8)$,  Hasegawa, Hoshi and Yamasaki \cite[Theorem 1.2]{HHY20} for $n=12,14,15,16$ and the stable $k$-rationality of $G=10T11$ $\simeq A_5\times C_2$}]\label{th1.23}
Let $K/k$ be a separable field extension of degree $n$ 
and $L/k$ be the Galois closure of $K/k$. 
Let $G={\rm Gal}(L/k)=nTm$ be a transitive subgroup of $S_n$ 
and $H={\rm Gal}(L/K)$ with $[G:H]=n$. 
Let $\T_K=R_{K/k}^{(1)}(\bG_m)$ be norm one tori of $K/k$ 
with $\widehat{\T}_K\simeq J_{G/H}$ and $\dim_k\, \T_K=n-1$. 
Assume that $2\leq n\leq 16$. 
Then\\ 
{\rm (1)} $\T_K$ is stably $k$-rational if and only if $G$ is given as in Table $2$;\\
{\rm (2)} $\T_K$ is not stably but retract $k$-rational if and only if $G$ is given as in Table $3$;\\
{\rm (3)} $\T_K$ is not retract $k$-rational if and only if $G$ is not in 
Tables $2$ and Table $3$;
\end{theorem}
\begin{center}
Table $2$: $\T_K=R_{K/k}^{(1)}(\bG_m)$ 
is stably $k$-rational where $G={\rm Gal}(L/k)=nTm\leq S_n$ $(2\leq n\leq 16)$\vspace*{2mm}\\
\begin{tabular}{l} 
$G=nTm$ $(2\leq n\leq 16)$: $\T_K=R_{K/k}^{(1)}(\bG_m)$ is stably $k$-rational 
\\\hline
$2T1\simeq C_2$\\
$3T1\simeq C_3$, $3T2\simeq S_3$\\
$4T1\simeq C_4$\\
$5T1\simeq C_5$, $5T2\simeq D_5$, $5T4\simeq A_5$\\ 
$6T1\simeq C_6$, $6T2\simeq S_3$, $6T3\simeq D_6$\\
$7T1\simeq C_7$, $7T2\simeq D_7$\\ 
$8T1\simeq C_8$\\
$9T1\simeq C_9$, $9T3\simeq D_9$, $9T27\simeq \PSL_2(\bF_8)$\\
$10T1\simeq C_{10}$, $10T2\simeq D_5$, $10T3\simeq D_{10}$, 
$10T11\simeq A_5\times C_2$\\
$11T1\simeq C_{11}$, $11T2\simeq D_{11}$\\ 
$12T1\simeq C_{12}$, $12T5\simeq Q_{12}\simeq C_3\rtimes C_4$, $12T11\simeq S_3\times C_4$\\
$13T1\simeq C_{13}$, $13T2\simeq D_{13}$\\ 
$14T1\simeq C_{14}$, $14T2\simeq D_{7}$, $14T3\simeq D_{14}$\\ 
$15T1\simeq C_{15}$, $15T2\simeq D_{15}$, $15T3\simeq D_5\times C_3$, $15T4\simeq S_3\times C_5$, $15T5\simeq A_5$, $15T7\simeq D_5\times S_3$,\\
$15T16\simeq A_5\times C_3\simeq \GL_2(\bF_4)$, $15T23\simeq A_5\times S_3$\\ 
$16T1\simeq C_{16}$
\end{tabular}
\end{center}\vspace*{0mm}

\begin{center}
Table $3$: $\T_K=R_{K/k}^{(1)}(\bG_m)$ 
is not stably but retract $k$-rational where $G={\rm Gal}(L/k)=nTm\leq S_n$ $(2\leq n\leq 16)$\vspace*{2mm}\\
\begin{tabular}{l} 
$G=nTm$ $(2\leq n\leq 16)$: $\T_K=R_{K/k}^{(1)}(\bG_m)$ is not stably but retract $k$-rational\\\hline
$5T3\simeq F_{20}$, $5T5\simeq S_5$\\
$7T3\simeq F_{21}$, $7T4\simeq F_{42}$, $7T5\simeq \PSL_2(\bF_7)$, $7T6\simeq A_7$, $7T7\simeq S_7$\\
$10T4\simeq F_{20}$, $10T5\simeq F_{20}\times C_2$, $10T12\simeq S_5$, $10T22\simeq S_5\times C_2$\\
$11T3\simeq F_{55}$, $11T4\simeq F_{110}$, $11T5\simeq \PSL_2(\bF_{11})$, 
$11T6\simeq M_{11}$, $11T7\simeq A_{11}$, $11T8\simeq S_{11}$\\
$13T3\simeq F_{39}$, $13T4\simeq F_{52}$, $13T5\simeq F_{78}$, $13T6\simeq F_{156}$, $13T7\simeq \PSL_2(\bF_{13})$, $13T8\simeq A_{13}$, $13T9\simeq S_{13}$\\
$14T4\simeq F_{42}$, $14T5\simeq F_{21}\times C_2$, $14T7\simeq F_{42}\times C_2$, $14T16\simeq \PSL_2(\bF_7)\rtimes C_2$,\\
$14T19\simeq\PSL_2(\bF_7)\times C_2$, 
$14T46\simeq S_{7}$, $14T47\simeq A_{7}\times C_2$, $14T49\simeq S_{7}\times C_2$\\
$15T6\simeq C_{15}\rtimes C_4$, $15T8\simeq F_{20}\times C_3$, $15T10\simeq S_5$, $15T11\simeq F_{20}\times S_3$,\\
$15T22\simeq (A_5\times C_3)\rtimes C_2\simeq \GL_2(\bF_4)\rtimes C_2$, $15T24\simeq S_5\times C_3$, $15T29\simeq S_5\times S_3$
\end{tabular}
\end{center}\vspace*{2mm}

In Table $3$, 
$F_{pl}\simeq C_p\rtimes C_l$ $(2<l\mid p-1)$ 
is the Frobenius group of order $pl$ where $p$ is a prime number 
and $M_{11}$ is the Mathieu group of degree $11$.\\

For particular groups $G$, and for general $H\leq G$, we have:

\begin{theorem}[{Hoshi and Yamasaki \cite[Theorem 1.1]{HY2}}]\label{th1.24}
Let $k$ be a field, $K/k$ be a finite 
separable field extension 
of degree $n$ 
and $L/k$ be the Galois closure of $K/k$ with 
$G={\rm Gal}(L/k)$ and $H={\rm Gal}(L/K)\lneq G$. 
Let $\T_K=R^{(1)}_{K/k}(\bG_m)$ be the norm one torus of $K/k$ 
with $\widehat{\T}_K\simeq J_{G/H}$ and ${\rm dim}_k\, \T_K=n-1=[K:k]-1=[G:H]-1$.\\ 
{\rm (1)} When $G\simeq A_4\simeq \PSL_2(\bF_3)$, 
$A_5\simeq\PSL_2(\bF_5)\simeq \PSL_2(\bF_4)\simeq \PGL_2(\bF_4)\simeq \SL_2(\bF_4)$, 
$A_6\simeq\PSL_2(\bF_9)$, 
$\T_K$ is not retract $k$-rational except for the two cases $(G,H)\simeq (A_5, V_4)$, $(A_5, A_4)$ 
with $|G|=60$, $n=[G:H]=15$, $5$. 
For the two exceptional cases, $\T_K$ is stably $k$-rational;\\ 
{\rm (2)} When $G\simeq S_3\simeq\PSL_2(\bF_2)\simeq \PGL_2(\bF_2)\simeq 
\SL_2(\bF_2)\simeq \GL_2(\bF_2)$, 
$S_4\simeq\PGL_2(\bF_3)$, $S_5\simeq\PGL_2(\bF_5)$, $S_6$, 
$\T_K$ is not retract $k$-rational except for the six cases 
$(G,H)\simeq (S_3,\{1\})$, $(S_3,C_2)$, 
$(S_5, V_4)$ satisfying $V_4\leq D(S_5)\simeq A_5$, 
$(S_5, D_4)$, $(S_5, A_4)$, $(S_5, S_4)$ with $|S_3|=6$, $n=[S_3:H]=6$, $3$, 
$|S_5|=120$, $n=[S_5:H]=30$, $15$, $10$, $5$. 
For the two exceptional cases $(S_3,\{1\})$, $(S_3, C_2)$, 
$\T_K$ is stably $k$-rational. 
For the four exceptional cases $(S_5, V_4)$ satisfying 
$V_4\leq D(S_5)\simeq A_5$, $(S_5, D_4)$, $(S_5, A_4)$, $(S_5, S_4)$, 
$\T_K$ is not stably but retract $k$-rational;\\
{\rm (3)} When $G\simeq \GL_2(\bF_3)$, $\GL_2(\bF_4)\simeq A_5\times C_3$, $\GL_2(\bF_5)$, 
$\T_K$ is not retract $k$-rational except for the case $(G,H)\simeq  (\GL_2(\bF_4), A_4)$ 
satisfying $A_4\leq D(G)\simeq A_5$ with $|G|=180$, $n=[G:H]=15$. 
For the exceptional case, $\T_K$ is stably $k$-rational;\\
{\rm (4)} When $G\simeq \SL_2(\bF_3)$, $\SL_2(\bF_5)$, $\SL_2(\bF_7)$, 
$\T_K$ is not retract $k$-rational;\\
{\rm (5)} When $G\simeq \PSL_2(\bF_7)\simeq \PSL_3(\bF_2)$, 
$\T_K$ is not retract $k$-rational except for the two cases $H\simeq D_4$, $S_4$ with $|G|=168$, $n=[G:H]=21, 7$. 
For the two exceptional cases, 
$\T_K$ is not stably but retract $k$-rational;\\
{\rm (6)} When 
$G\simeq \PSL_2(\bF_8)\simeq \PGL_2(\bF_8)\simeq \SL_2(\bF_8)$, 
$\T_K$ is not retract $k$-rational except for the two cases $H={\rm Sy}_2(G)\simeq (C_2)^3$, 
$N_G({\rm Sy}_2(G))\simeq (C_2)^3\rtimes C_7$ 
with $|G|=504$, $n=[G:H]=63$, $9$.   
For the two exceptional cases, $\T_K$ is stably $k$-rational. 
\end{theorem}


\subsection{Norm one tori $\T_A=R^{(1)}_{A/k}(\bG_m)$ of \'etale algebras $A/k$}

Norm one tori $R^{(1)}_{A/k}(\bG_m)$ of \'etale algebras $A/k$ 
(multinorm one tori $R^{(1)}_{A/k}(\bG_m)$ 
if $r\geq 2$, see Section \ref{ssRatNorm1}) are 
investigated by Colliot-Th\'{e}l\`{e}ne and Sansuc \cite[page 208]{CTS77},
Endo \cite{End01}, \cite{End11}. 
\begin{theorem}[{Endo \cite[Theorem 2]{End01}}]\label{th4.1}
Let $p$ be a prime number, 
$P=(\bZ/p\bZ)^{\oplus m}$ and 
$P_i\leq P$ $(1\leq i\leq r)$ be distinct subgroup of index $p$. 
Let $\varepsilon=(\varepsilon_1^{h_1},\ldots,\varepsilon_r^{h_r}): 
\oplus_{i=1}^r \bZ[P/P_i]^{\oplus h_i}\to \bZ$ 
be the multiaugmentation map where 
$\varepsilon_i: \bZ[P/P_i]\to \bZ$ $(1\leq i\leq r)$ 
is the augmentation map, 
$\widetilde{I}={\rm Ker}(\varepsilon)$ and 
$\widetilde{J}=(\widetilde{I})^{\circ}={\rm Hom}(\widetilde{I},\bZ)$. 
Let $\widetilde{L}$ be a field with Galois group 
${\rm Gal}(\widetilde{L}/k)\simeq P$ and 
$K_i=\widetilde{L}^{P_i}$ with $[K_i:k]=p$ $(1\leq i\leq r)$. 
Let $\T_A=R^{(1)}_{A/k}(\bG_m)$ be the 
norm one torus of an \'etale algebra $A/k$ 
where $A=\prod_{i=1}^rK_i^{h_i}$ and $J\simeq \widehat{\T}_A$. 
%
Let $L=K_1\cdots K_r\subset \overline{k}$ be the composite field of 
$K_1,\ldots, K_r$ with Galois group $G={\rm Gal}(L/k)\simeq P/P^\prime$ 
where $P^\prime=\bigcap_{i=1}^r P_i$.\\
{\rm (1)} 
When $p=2$, $k(\T_A)\simeq \widetilde{L}(\widetilde{J})^P\simeq L(J)^G$ 
is stably $($retract$)$ $k$-rational 
if and only if $r=1,2$;\\
{\rm (2)}
When $p\geq 3$, $k(\T_A)\simeq \widetilde{L}(\widetilde{J})^P\simeq L(J)^G$ 
is stably $($retract$)$ $k$-rational 
if and only if $r=1$.
\end{theorem}

\begin{theorem}[{Endo \cite[Proposition 1.3, Corollary 1.4]{End11}, see also Endo and Miyata \cite[Proposition 1.1]{EM73}, Hajja and Kang \cite[Theorem 1]{HK95}, Miyata \cite[Lemma]{Miy71}, Saltman \cite[Proposition 3.6 (a)]{Sal84}}]\label{th4.2}
Let $G$ be a finite group, $H_1,\ldots,H_t\leq G$ $(t\geq 2)$ be 
subgroups with $H_{t-1}\geq H_t$ and 
$\varepsilon=(\varepsilon_1,\ldots,\varepsilon_t): 
\oplus_{i=1}^t \bZ[G/H_i]\to \bZ$, 
$\varepsilon^\prime=(\varepsilon_1,\ldots,\varepsilon_{t-1}): 
\oplus_{i=1}^{t-1} \bZ[G/H_i]\to \bZ$ 
be the multiaugmentation maps 
where 
$\varepsilon_i: \bZ[G/H_i]\to \bZ$ $(1\leq i\leq t)$ 
is the augmentation map, 
$I={\rm Ker}(\varepsilon)$, $I^\prime={\rm Ker}(\varepsilon^\prime)$, 
$J=I^\circ={\rm Hom}(I,\bZ)$ and 
$J^\prime=(I^\prime)^{\circ}={\rm Hom}(I^\prime,\bZ)$. 
Then we have $I\simeq I^\prime\oplus\bZ[G/H_t]$ and 
$J\simeq J^\prime\oplus\bZ[G/H_t]$ 
and hence $L(J)^G\simeq L(J^\prime)^G(u_1,\ldots,u_s)$ 
$($$L(J)^G$ is rational over $L(J^\prime)$$)$ where 
$s=[G:H_t]$ and ${\rm Gal}(L/k)\simeq G$. 
In particular, $\T_A=R^{(1)}_{A/k}(\bG_m)$ and 
$\T_{A^\prime}=R^{(1)}_{A^\prime/k}(\bG_m)$ 
are stably birationally $k$-equivalent where 
$A=\prod_{i=1}^tK_i$, $A^\prime=\prod_{i=1}^{t-1}K_i$, 
$\widehat{T}_A\simeq J$, 
$\widehat{T}_{A^\prime}\simeq J^\prime$. 
For example, if $H_i=H$ for any $1\leq i\leq t$, 
i.e. $K_i=K$ for any $1\leq i\leq t$, 
then 
$J\simeq J_{G/H}\oplus\bZ[G/H]^{\oplus t-1}$ and hence 
$\T_A=R^{(1)}_{A/k}(\bG_m)$ is stably $($resp. retract$)$ $k$-rational 
if and only if $\T_K=R^{(1)}_{K/k}(\bG_m)$ 
is stably $($resp. retract$)$ $k$-rational. 
\end{theorem}
\begin{proposition}[{Endo \cite[Proposition 1.3]{End11}}]\label{prop4.3}
Let 
$\T_A=R^{(1)}_{A/k}(\bG_m)$ be the norm one torus 
of an \'etale algebra $A/k$ 
with ${\rm dim}_k\, A=n$ where $A=\prod_{i=1}^rK_i$ with $[K_i:k]=n_i$. 
Let $X=X(A)=\Hom_k(A,\overline{k})$ be the corresponding 
$G$-set with $\widehat{T}_A\simeq J_X$ and $|X|=n$ where 
$X=\coprod_{i=1}^r X_i$ is the disjoint union of $G$-orbits 
with $|X_i|=n_i$ and $n=\sum_{i=1}^r n_i$. 
If there exists $1\leq i\leq r$ such that $n_i=1$, 
i.e. $K_i=k$, that is, $X_i=\{x_i\}$ where $x_i$ is a fixed point, 
then $J_X\simeq \bZ[X\setminus X_i]$ is permutation. 
In particular, $k(\T_A)\simeq L(J_X)^G$ 
is $k$-rational where $\T_A=R^{(1)}_{A/k}(\bG_m)$. 
\end{proposition}
\begin{proof}
This is a special case of Theorem \ref{th4.2} 
(Endo \cite[Proposition 1.3]{End11}), 
see also Endo and Miyata \cite[Proposition 1.1]{EM73}, Hajja and Kang \cite[Theorem 1]{HK95}, Miyata \cite[Lemma]{Miy71}. 
We give an alternative proof here for the reader's convenience.  
We have $0\to I_X\xrightarrow{\iota}\bZ[X]\xrightarrow{\varepsilon}\bZ\to 0$ 
where $\varepsilon$ is the augmentation map and 
$X=X_i\cup Y$ where $Y=X\setminus X_i$ and $X_i=\{x_i\}$. 
We see that the restriction map of 
$\varepsilon$ to $X_i$, $\varepsilon |_{X_i}: \bZ[X_i]\xrightarrow{\sim} \bZ$, 
$a_i[x_i]\mapsto a_i$ becomes isomorphism. 
For $a_y[y]\in \bZ[Y]$, there exists 
$-a_y[x_i]\in \bZ[X_i]$ such that $a_y[y]-a_y[x_i]\in I_X$. 
Hence we can get the isomorphism $\varphi: \bZ[Y]\xrightarrow{\sim} I_X$, 
$a_y[y]\mapsto a_y[y]-a_y[x_i]$ with $\varphi\circ\iota=\Id_{I_X}$: 
\begin{align*}
\xymatrix@C=18pt@R=18pt{
 & & 0\ar[d] & & \\
 & \bZ[Y]\ar@{=}[r]\ar[d]^{\simeq} & \bZ[Y]\ar[d] & & \\
0 \ar[r] & I_X\ar[r]^-{\iota} & \bZ[X]\simeq \bZ[X_i]\oplus\bZ[Y]\ar[r]^-{\varepsilon} \ar[d]& \bZ\ar[r]\ar@{=}[d] & 0.\\
  &  & \bZ[X_i]\ar[r]^-{\varepsilon|_{X_i}}_\simeq\ar[d] &  \bZ & \\
 & & 0 & &
}\end{align*}
This implies that $J_X=(I_X)^\circ\simeq \bZ[Y]$ 
and hence it follows from Endo and Miyata \cite[Proposition 1.1]{EM73} that 
$k(\T_A)\simeq L(J_X)^G$ is $k$-rational 
(see also Hajja and Kang \cite[Theorem 1]{HK95}, Miyata \cite[Lemma]{Miy71}). 
\end{proof}

\section{Hasse norm principle for field extensions $K/k$ and \'etale algebras $A/k$}\label{S4}

Let $k$ be a global field, 
i.e. a number field (a finite extension of $\bQ$) 
or a function field of an algebraic curve over 
$\bF_q$ (a finite extension of $\bF_q(t))$. 
Let $v$ be a place of $k$ and $k_v$ be the completion of $k$ at $v$. 

Let $X$ be a $k$-variety. 
We say that 
{\it Hasse principle holds for $X$} 
if $X$ has a $k_v$-rational point for any place $v$ of $k$, 
then $X$ has a $k$-rational point. 
Hasse-Minkowski theorem says that 
Hasse principle holds for $X$ where 
$X$ is defined as $Q(x_1,\ldots,x_n)=0$ 
for a quadratic form $Q(x_1,\ldots,x_n)$ in $\bP^{n-1}$ 
(see Hasse \cite{Has23a}, \cite{Has23b}, \cite{Has24a}, \cite{Has24b} 
over any number field $k$). 
Lind \cite{Lin40} and Reichardt \cite{Rei42} 
gave a counterexample to Hasse principle for $X$ 
where $X$ is a genus $1$ curve defined as $2y^2=1-17x^4$ 
(see also Colliot-Th\'{e}l\`{e}ne and Poonen \cite{CTP00}, 
Poonen \cite{Poo01}). 
Selmer \cite{Sel51} also gave a counterexample to Hasse principle 
for $X$ where $X$ is a plane cubic curve of genus $1$ in $\bP^2$ 
defined as $3x^3+4y^3+5z^3=0$. 

\begin{definition}
Let $T$ be an algebraic $k$-torus 
and $T(k)$ be the group of $k$-rational points of $T$. 
Then $T(k)$ 
embeds into $\prod_{v\in V_k} T(k_v)$ by the diagonal map 
where 
$V_k$ is the set of all places of $k$ and 
$k_v$ is the completion of $k$ at $v$. 
Let $\overline{T(k)}$ be the closure of $T(k)$  
in the product $\prod_{v\in V_k} T(k_v)$. 
The group 
\begin{align*}
A(T)=\left(\prod_{v\in V_k} T(k_v)\right)/\overline{T(k)}
\end{align*}
is called {\it the kernel of the weak approximation} of $T$. 
We say that {\it $T$ has the weak approximation property} if $A(T)=0$. 
\end{definition}

\begin{definition}
Let $E$ be a principal homogeneous space (= torsor) under $T$.  
{\it Hasse principle holds for $E$} means that 
if $E$ has a $k_v$-rational point for all $k_v$, 
then $E$ has a $k$-rational point. 
The set $H^1(k,T)$ classifies all such torsors $E$ up 
to (non-unique) isomorphism. 
We define {\it the Shafarevich-Tate group} of $T$: 
\begin{align*}
\Sha(T)={\rm Ker}\left\{H^1(k,T)\xrightarrow{\rm res} \bigoplus_{v\in V_k} 
H^1(k_v,T)\right\}.
\end{align*} 
Then 
Hasse principle holds for all torsors $E$ under $T$ 
if and only if $\Sha(T)=0$. 
\end{definition}

%
\begin{theorem}[{Voskresenskii \cite[Theorem 5, page 1213]{Vos69}, 
\cite[Theorem 6, page 9]{Vos70}, see also \cite[Section 11.5, page 120, Section 11.6, Theorem, page 120]{Vos98}}]\label{thV}
Let $k$ be a global field, 
$T$ be an algebraic $k$-torus and $X$ be a smooth $k$-compactification of $T$. 
Then there exists an exact sequence
\begin{align*}
0\to A(T)\to H^1(k,{\rm Pic}\,\overline{X})^{\vee}\to \Sha(T)\to 0
\end{align*}
where $M^{\vee}={\rm Hom}(M,\bQ/\bZ)$ is the Pontryagin dual of $M$. 
In particular, if $T$ is retract $k$-rational, then $ H^1(k,{\rm Pic}\,\overline{X})=0$ and hence 
$A(T)=0$ and $\Sha(T)=0$. 
Moreover, if $L$ is the splitting field of $T$ and $L/k$ 
is an unramified extension, then $A(T)=0$ and 
$H^1(k,{\rm Pic}\,\overline{X})^{\vee}\simeq \Sha(T)$. 
\end{theorem}
It follows that 
$H^1(k,{\rm Pic}\,\overline{X})=0$ if and only if $A(T)=0$ and $\Sha(T)=0$, 
i.e. $T$ has the weak approximation property and 
Hasse principle holds for all torsors $E$ under $T$. 

\begin{definition}
Let $k$ be a global field, 
$K/k$ be a finite extension and $\bA_K^\times$ be the idele group of $K$.
Let $A=\prod_{i=1}^r K_i$ be an 
\'etale algebra where $K_i/k$ $(1\leq i\leq r)$ 
is a finite separable field extension and 
$\bA_A^\times$ be the idele group of $A$.\\ 
(1) We say that {\it the Hasse norm principle holds for $K/k$} 
if $(N_{K/k}(\bA_K^\times)\cap k^\times)/N_{K/k}(K^\times)=1$ 
where $N_{K/k}$ is the norm map.\\ 
(2) We say that {\it the Hasse norm principle holds for $A/k$}, 
also {\it the Hasse multinorm principle holds for $A/k$} if $r\geq 2$, 
if $(\prod_{i=1}^rN_{K_i/k}(\bA_{K_i}^\times)\cap k^\times)/\prod_{i=1}^rN_{K_i/k}(K_i^\times)=1$ 
where $N_{K_i/k}$ $(1\leq i\leq r)$ is the norm map. 
\end{definition}

Hasse \cite[Satz, page 64]{Has31} proved that 
the Hasse norm principle holds for any cyclic extension $K/k$ 
but does not hold for bicyclic extension $\bQ(\sqrt{-39},\sqrt{-3})/\bQ$. 
For Galois extensions $K/k$, Tate \cite{Tat67} gave the following theorem: 

\begin{theorem}[{Tate \cite[page 198]{Tat67}}]\label{thTate}
Let $k$ be a global field, $K/k$ be a finite Galois extension 
with Galois group $G={\rm Gal}(K/k)$. 
Let $V_k$ be the set of all places of $k$ 
and $G_v$ be the decomposition group of $G$ at $v\in V_k$. 
Then 
\begin{align*}
(N_{K/k}(\bA_K^\times)\cap k^\times)/N_{K/k}(K^\times)\simeq 
{\rm Coker}\left\{\bigoplus_{v\in V_k}\widehat H^{-3}(G_v,\bZ)\xrightarrow{\rm cores}\widehat H^{-3}(G,\bZ)\right\}
\end{align*}
where $\widehat H$ is the Tate cohomology. 
In particular, the Hasse norm principle holds for $K/k$ 
if and only if the restriction map 
$H^3(G,\bZ)\xrightarrow{\rm res}\bigoplus_{v\in V_k}H^3(G_v,\bZ)$ 
is injective. 
In particular, if $H^3(G,\bZ)\simeq M(G)=0$, 
then the Hasse norm principle holds for $K/k$ 
where 
$M(G)=H^2(G,\bC^\times)$ 
is the Schur multiplier of $G$. 
\end{theorem}
If $G\simeq C_n$ is cyclic, then 
$\widehat H^{-3}(G,\bZ)\simeq H^3(G,\bZ)\simeq H^1(G,\bZ)=0$ 
and hence the Hasse's original theorem follows. 

For algebraic $k$-tori $T$, 
we also obtain the group $T(k)/R$ of $R$-equivalence classes 
over a local field $k$ via 
$T(k)/R\simeq H^1(k,{\rm Pic}\,\overline{X})\simeq 
H^1(G,[\widehat{T}]^{fl})$ 
(see Colliot-Th\'{e}l\`{e}ne and Sansuc \cite[Corollary 5, page 201]{CTS77}, 
Voskresenskii \cite[Section 17.2]{Vos98} and Hoshi, Kanai and Yamasaki \cite[Section 7, Application 1]{HKY22}). 

Ono \cite{Ono63} established the relationship 
between the Hasse norm principle for $K/k$ 
and the Hasse principle for all torsors $E$ under
the norm one torus $R^{(1)}_{K/k}(\bG_m)$ 
(see Platonov and Rapinchuk \cite[Section 6.3]{PR94} for $r\geq 2$): 
\begin{theorem}[{Ono \cite[page 70]{Ono63}, see also Platonov \cite[page 44]{Pla82}, 
Kunyavskii \cite[Remark 3]{Kun84}, Platonov and Rapinchuk \cite[Section 6.3, page 307, page 313]{PR94}}]\label{thOno}
Let $k$ be a global field 
and 
$A=\prod_{i=1}^r K_i$ be an 
\'etale algebra where $K_i/k$ $(1\leq i\leq r)$ 
is a finite separable field extension. 
Then 
\begin{align*}
\Sha(R^{(1)}_{A/k}(\bG_m))\simeq 
(\prod_{i=1}^rN_{K_i/k}(\bA_{K_i}^\times)\cap k^\times)/\prod_{i=1}^rN_{K_i/k}(K_i^\times).
\end{align*}
In particular, $\Sha(R^{(1)}_{A/k}(\bG_m))=0$ if and only if 
the Hasse norm principle holds for $A/k$. 
\end{theorem}

For norm one tori $\T_A=R^{(1)}_{A/k}(\bG_m)$, 
recall that 
the function field $k(\T_A)$ may be regarded as $L(M)^G$ 
for the character lattice $M=\widehat{\T}_A\simeq J_X$ and hence we have: 
\begin{align*}
&\hspace*{2.5mm}[J_X]^{fl}=0\hspace*{10mm}
\ \ \Rightarrow\ \ \hspace*{0mm}[J_X]^{fl}\ \textrm{is\ invertible}
\ \ \Rightarrow\ \ H^1(G,[J_X]^{fl})=0
\ \ \Rightarrow\ \ A(\T_A)=0\ \textrm{and}\ \Sha(\T_A)=0\\
&\hspace*{10mm}\Updownarrow\hspace*{37mm} \Updownarrow\\
&\hspace*{-4mm}\T_A\ \textrm{is\ stably}\ k\textrm{-rational}\ \ \hspace*{0mm}
\Rightarrow\ \ \hspace*{0mm}\T_A\ \textrm{is\ retract}\ k\textrm{-rational}\ 
\end{align*}
where the last implication holds over a global field $k$ 
(see Theorem \ref{thV}, 
see also Colliot-Th\'{e}l\`{e}ne and Sansuc \cite[page 29]{CTS77}). 
The last conditions mean that 
$\T_A$ has the weak approximation property, 
Hasse principle holds for all torsors $E$ under $\T_A$ and 
the Hasse norm principle holds for $A/k$ as above.  


By Poitou-Tate duality (see Milne \cite[Theorem 4.20]{Mil86}, 
Platonov and Rapinchuk \cite[Theorem 6.10]{PR94}, 
Neukirch, Schmidt and Wingberg \cite[Theorem 8.6.8]{NSW00}, 
Harari \cite[Theorem 17.13]{Har20}), 
we also have 
\begin{align*}
\Sha(T)^\vee\simeq\Sha^2(G,\widehat{T})
\end{align*}
where $\Sha(T)^\vee={\rm Hom}(\Sha(T),\bQ/\bZ)$, 
\begin{align*}
\Sha^i(G,\widehat{T})={\rm Ker}\left\{H^i(G,\widehat{T})\xrightarrow{\rm res} \bigoplus_{v\in V_k} 
H^i(G_v,\widehat{T})\right\}\quad (i\geq 1)
\end{align*}
is {\it the $i$-th Shafarevich-Tate group} 
of $\widehat{T}\simeq {\rm Hom}(T\times_k L,\bG_{m,L})$, 
$G={\rm Gal}(L/k)$ and $L$ is the minimal splitting field of 
algebraic $k$-torus $T$. 
Note that $\Sha(T)\simeq \Sha^1(G,T)$. 
In the special case where 
$\T_K=R^{(1)}_{K/k}(\bG_m)$ and $K/k$ is Galois with $G={\rm Gal}(K/k)$, 
we have $\widehat{\T}_K\simeq J_{G}$ and 
$H^2(G,J_{G})\simeq H^3(G,\bZ)$ and hence 
we get Tate's theorem (Theorem \ref{thTate}). 

By taking the duals of Voskresenskii's exact sequence as in Theorem \ref{thV}, 
we get the following exact sequence
\begin{align*}
0\to \Sha^2(G,\widehat{T})\to \Sha^2_\omega(G,\widehat{T})\to A(T)^\vee\to 0
\end{align*}
where the map $\Sha^2(G,\widehat{T})\to \Sha^2_\omega(G,\widehat{T})$ 
is the natural inclusion arising from the Chebotarev density theorem 
(see also Macedo and Newton \cite[Proposition 2.4]{MN22}). 

\subsection{Case of field extensions $K/k$} 
The Hasse norm principle for Galois extensions $K/k$ 
was investigated by Tate \cite{Tat67}, 
Gerth \cite{Ger77}, \cite{Ger78} and 
Gurak \cite{Gur78a}, \cite{Gur78b}, \cite{Gur80} 
(see also \cite[pages 308--309]{PR94}). 
Gurak \cite{Gur78a} showed that 
the Hasse norm principle holds for Galois extension $K/k$ 
if all the Sylow subgroups of $G={\rm Gal}(K/k)$ are cyclic 
(see also Remark \ref{rem3.6}). 
Note that this also follows from Theorem \ref{thOno} 
and the retract $k$-rationality of 
$\T_K=R^{(1)}_{K/k}(\bG_m)$ due to Endo and Miyata \cite[Theorem 2.3]{EM75}. 

For non-Galois extension $K/k$ of degree $n$, 
the Hasse norm principle was investigated by 
Bartels \cite{Bar81a} (holds for $n=p$; prime), 
\cite{Bar81b} (holds for $G\simeq D_n$), 
Voskresenskii and Kunyavskii \cite{VK84} (holds for $G\simeq S_n$), 
Kunyavskii \cite{Kun84} $(n=4)$, 
Drakokhrust and Platonov \cite{DP87} $(n=6)$, 
Macedo \cite{Mac20} (holds for $G\simeq A_n$ ($n\neq 4)$), 
Macedo and Newton \cite{MN22} 
($G\simeq A_4$, $S_4$, $A_5$, $S_5$, $A_6$, $A_7$ $($general $n))$, 
Hoshi, Kanai and Yamasaki \cite{HKY22} $(n\leq 15$ $(n\neq 12))$ 
(holds for $G\simeq M_n$ ($n=11,12,22,23,24$; $5$ Mathieu groups)), 
\cite{HKY23} $(n=12)$, 
\cite{HKY25} $(n=16)$, 
\cite{HKY} $(G\simeq M_{11}$, $J_1$ $($general $n))$, 
Hoshi and Yamasaki \cite{HY2} 
(holds for $G\simeq {\rm PSL}_2(\bF_7)$ $(n=21)$, 
${\rm PSL}_2(\bF_8)$ $(n=63)$), 
\cite{HY3} (holds for metacyclic groups $G$ 
with trivial Schur multiplier $M(G)=0$ $($general $n))$, 
\cite{HY4} ($G\simeq (C_p)^2\rtimes C_p$ ($n=p^2$, $p\geq 3$; prime)) 
where $G={\rm Gal}(L/k)$ and $L/k$ is the Galois closure of $K/k$. 
Recall that the 
case where $n=p$ 
also follows from Theorem \ref{thOno} and 
the retract $k$-rationality of 
$\T_K=R^{(1)}_{K/k}(\bG_m)$ due to 
Colliot-Th\'{e}l\`{e}ne and Sansuc \cite[Proposition 9.1]{CTS87}. 

\subsection{Case of \'etale algebras $A/k$}

We collect known results about 
Hasse norm principle for \'etale algebras $A/k$: 
\begin{theorem}[{H\"urlimann \cite[Proposition 3.3]{Hur84}, Prasad and Rapinchuk \cite[Proposition 4.2]{PR10}, Pollio and Rapinchuk \cite[Proposition 15, Section 5]{PR13}, Colliot-Th\'{e}l\`{e}ne \cite[Th\'eor\`eme 4.1]{CT14}, Wei \cite[Corollary 3.3]{Wei14}, Demarche and Wei \cite[Theorem 1, Theorem 6]{DW14}, Bayer-Fluckiger, Lee and Parimala \cite{BLP19}}]\label{th4.3}
%
Let $A=\prod_{i=1}^rK_i$ be an \'etale $k$-algebra 
with $X={\rm Hom}_k(A,\overline{k})$ 
and $L_i/k$ $(1\leq i\leq r)$ 
be the Galois closure of $K_i/k$ with Galois groups 
$G_i={\rm Gal}(L_i/k)$, $H_i={\rm Gal}(L_i/K_i)\leq G_i$  
and $X_i={\rm Hom}_k(K_i,\overline{k})$ $(1\leq i\leq r)$ 
with $X=\bigcup_{i=1}^r X_i$. 
Let $L=L_1\cdots L_r\subset\overline{k}$ be the composite field of 
$L_1,\ldots,L_r$, i.e. 
the smallest field which contains all $L_i$, 
with Galois group $G={\rm Gal}(L/k)$ 
and $\T_A=R^{(1)}_{A/k}(\bG_m)$ be the norm one torus of $A/k$ 
with $\widehat{\T}_A\simeq J_X$. 
Let $G_v$ be the decomposition group of $G$ at a place $v$ of $k$. 
Let $p$ be a prime number.\\
{\rm (1)} $(${\rm H\"urlimann} \cite{Hur84}$;$ $r=2)$ 
If $K_1=L_1$, $K_2=L_2$; Galois, $G_1\simeq C_n$; cyclic and 
$K_1\cap K_2=k$, then $G\simeq C_n\times G_2$ and 
$\Sha^2_\omega(G,J_X)=0$;\\
{\rm (2)} $(${\rm Platonov and Rapinchuk} \cite{PR94}$;$ $r=2)$ 
If $k$ is a global field, 
the Hasse norm principle holds for $K_1/k$ and $L_1\cap L_2=k$, 
then $G\simeq G_1\times G_2$ 
and $\Sha(\T_A)=0$, i.e. the Hasse norm principle holds for $A/k$;\\
{\rm (3)} $(${\rm Prasad and Rapinchuk} \cite{PR10}$;$ $r=2)$ 
If $k$ is a global field, 
$K_1=L_1$; Galois, $G_1={\rm Gal}(L_1/k)$; abel, 
the Hasse norm principle holds for $L_1/k$ and $K_1\cap K_2=k$, 
then $G\leq G_1\times G_2$ is a subdirect product of $G_1$, $G_2$ 
and $\Sha(\T_A)=0$, i.e. the Hasse norm principle holds for $A/k$;\\
{\rm (4)} $(${\rm Pollio\ and\ Rapinchuk} \cite{PR13}$)$ 
If $L_1\cap L_2=k$, then $G\simeq G_1\times G_2$ and 
$\Sha^2_\omega(G,J_X)=0$. 
In particular, if $(L_1\cdots L_{i-1})\cap L_i=k$ for any $2\leq i\leq r$, 
then $G\simeq G_1\times \cdots\times G_r$ and $\Sha^2_\omega(G,J_X)=0$;\\
{\rm (5)} $(${\rm Colliot-Th\'{e}l\`{e}ne} \cite{CT14}$;$ $r=3)$ 
If $K_1=L_1, K_2=L_2, K_3=L_3$, $G_1\simeq G_2\simeq G_3\simeq C_2$, 
$L_1\cap L_2=L_1\cap L_3=L_2\cap L_3=k$ and $L_1L_2\supset L_3$, 
then $G\simeq C_2\times C_2$ and 
$\Sha^2_\omega(G,J_X)\simeq\bZ/2\bZ$;\\
{\rm (6)} $(${\rm Wei} \cite{Wei14}; $r=2)$ 
If $K_1\cap L_2=k$, then $G\leq G_1\times G_2$ is a subdirect product of $G_1$, $G_2$ and 
$\Sha^2_\omega(G,J_X)=0$;\\
{\rm (7)} $(${\rm Demarche and Wei} \cite{DW14}$)$ 
Assume that $\{1,\ldots,r\}=I\cup J$ 
$(I\cap J=\emptyset, I,J\neq\emptyset)$ 
and $F:=\bigcap_{i=1}^r K_i/k$ is Galois. 
Let $F_i/K_i$ be a separable field extension with the natural map 
${\rm Aut}_k(F_i)\to {\rm Aut}_k(F)$ surjective. 
Let $F_I$ $($resp. $F_J)$ be the composite field of $L_i$ $(i\in I)$ 
(resp. $j\in J$), i.e. 
the smallest field which contains all $F_i$ $(i\in I)$ 
$($resp. $(i\in J))$ and 
$L_I/F$ $($resp. $L_J/F)$ be the Galois closure of 
$F_I/F$ $($resp. $F_J/F)$ with 
$G_I={\rm Gal}(L_I/F)$ $($resp. $G_J={\rm Gal}(L_J/F))$. 
If $F_I\cap L_J=F$, then $H={\rm Gal}(L/F)\leq G_I\times G_J$ is a subdirect product 
of $G_I$, $G_J$ and $\Sha^2_\omega(G,J_X)\simeq 
\Sha^2_\omega(\overline{G},\widehat{\T}_F)$ 
where $\overline{G}=G/H$ and 
$\T_F=R^{(1)}_{F/k}(\bG_m)$ is the norm one torus of $F/k$. 
As the special case, if $F=k$ $(G=H)$, then $\Sha^2_\omega(G,J_X)=0$;\\
{\rm (8)} $(${\rm Bayer-Fluckiger, Lee and Parimala} \cite{BLP19}$)$ 
If $K_i=L_i$; Galois, $G_i\simeq C_p$ $(1\leq i\leq r)$, 
$L_i\cap L_j=k$ $(1\leq i<j\leq r)$, then 
\begin{align*}
\Sha(\T_A)=\begin{cases}
(\bZ/p\bZ)^{\oplus r-2}& {\rm if}\ G\simeq C_p\times C_p, 
G_v\simeq C_p\ {\rm for\ any\ place}\ v\ {\rm of}\ k\ (2\leq r\leq p+1)\\
0& {\rm otherwise}.
\end{cases}
\end{align*}

In particular, by Theorem \ref{thV} 
(see also Section \ref{S2} and Section \ref{S4}), 
if $k$ is a global field 
and 
$\Sha^2_\omega(G,J_X)\simeq H^1(k,{\rm Pic}\, \overline{X})\simeq 
H^1(G,{\rm Pic}\, X_K)\simeq H^1(G,[J_X]^{fl})\simeq 
{\rm Br}(X)/{\rm Br}(k)\simeq 
{\rm Br}_{\rm nr}(k(X))/{\rm Br}(k)=0$ 
where $X$ is a smooth $k$-compactification of $\T_A$, 
then 
$A(\T_A)=0$, i.e. $\T_A$ has the weak approximation property, 
and $\Sha(\T_A)=0$, i.e. the Hasse norm principle holds for $A/k$ 
$($that is, Hasse principle holds for all torsors $E$ under $\T_A$$)$ 
where $\Sha(\T_A)^\vee\leq \Sha^2_\omega(G,J_X)$ 
and $\Sha(\T_A)^\vee={\rm Hom}(\Sha(\T_A),\bQ/\bZ)$. 
The condition $\Sha(\T_A)=0$ means that 
for the corresponding norm hypersurface 
\begin{align*}
\prod_{i=1}^r f_i(x_{i,1},\ldots,x_{i,n_i})=b, 
\end{align*}
it has a $k$-rational point 
if and only if it has a $k_v$-rational point 
for any place $v$ of $k$ where 
$f_i(x_{i,1},\ldots,x_{i,n_i})\in k[x_{i,1},\ldots,x_{i,n_i}]$ 
$(1\leq i\leq r)$ 
is the polynomial of total degree $n_i$ 
defined as 
\begin{align*}
f_i(x_{i,1},\ldots,x_{i,n_i})=N_{K_i/k}(x_{i,1}w_{i,1}+\cdots+x_{i,n_i}w_{i,n_i})=\prod_{\overline{g}\in G_i/H_i}\overline{g}(x_{i,1}w_{i,1}+\cdots+x_{i,n_i}w_{i,n_i})
\end{align*}
where 
$\{w_{i,1},\ldots,w_{i,n_i}\}$ is a basis of $K_i/k$, 
$N_{K_i/k}:K^\times\to k^\times$ is the norm map 
and $b\in k^\times$ $($see Voskresenskii \cite[Example 4, page 122]{Vos98}$)$. 
\end{theorem}
\begin{remark}
(1) 
Theorem \ref{th4.3} (6) implies that if 
$\Sha^2_\omega(\overline{G},\widehat{S})=0$ with $\overline{G}=G/H$, 
then $\Sha^2_\omega(G,J_X)=0$. 
For example, it follows from the retract $k$-rationality of 
$\T_F=R^{(1)}_{F/k}(\bG_m)$ 
due to Endo and Miyata \cite[Theorem 2.3]{EM75} 
(also follows from Gurak \cite{Gur78a}) that the Hasse norm principle holds 
for Galois extension $F/k$ 
if all the Sylow subgroups of ${\rm Gal}(F/k)$ are cyclic 
(see also Remark \ref{rem3.6}). 
Moreover, by Tate's theorem (Theorem 
\ref{thTate}), 
if $M(\overline{G})=0$, then $\Sha(\widehat{S})=0$, 
i.e. Hasse norm principle holds for $F/k$, 
where $M(\overline{G})$ is the Schur multiplier of $\overline{G}$. 
It also follows from Volskresenskii's theorem (Theorem \ref{thV}) that 
if $M(\overline{G})=0$, then $\Sha^2_\omega(\overline{G},\widehat{S})
=H^2(\overline{G},J_{\overline{G}})\simeq M(\overline{G})=0$. 
See Hoshi and Yamasaki \cite{HY3} for further examples.\\
(2) By considering an unramified extension $L/k$, 
it follows from Theorem \ref{th4.3} (8) that 
if $K_i=L_i$; Galois, $G_i\simeq C_p$ $(1\leq i\leq r)$, 
$L_i\cap L_j=k$ $(1\leq i<j\leq r)$, then 
\begin{align*}
\Sha^2_\omega(G,J_X)=\begin{cases}
(\bZ/p\bZ)^{\oplus r-2}& {\rm if}\ G\simeq C_p\times C_p\ (2\leq r\leq p+1)\\
0& {\rm otherwise} 
\end{cases}
\end{align*}
(see also Macedo \cite[Theorem 4.9]{Mac25}). 
\end{remark}

\section{Group cohomology and tensor products of extensions of $G$-lattices}\label{S5} 

\subsection{Recollections from group cohomology} 
%
Let $M$ be a $G$-lattice and $N$ be a $G$-module. 
For any $i\geq 0$, it is classical  that there exists an isomorphism 
\begin{align*}
\Ext^i_{\bZ[G]}(M,N)\xrightarrow{\sim} H^i(G,{\rm Hom}_{\bZ}(M,N))\simeq H^i(G,M^\circ\otimes N)
\end{align*}
where $M^\circ:={\rm Hom}(M,\bZ)$ and $G$ acts on ${\rm Hom}_{\bZ}(M,N)$ by 
\begin{align*}
(gu)(m)=g\left(u(g^{-1}m)\right)\ (g\in G, u\in {\rm Hom}(M,N), m\in M)
\end{align*}
(see Brown \cite[Proposition 8.3 (b), page 28, page 56, Proposition 2.2, page 61]{Bro82}, see also Colliot-Th\'{e}l\`{e}ne and Sansuc \cite[Section 0.5, page 155]{CTS87}, Voskresenskii \cite[page 72]{Vos98}). Let us describe it explicitly.
\begin{lemma}\label{lem5.1}
Let $A$ and $C$ be $G$-lattices. 
For $1$-cocycle $f\in Z^1(G,{\rm Hom}_\bZ(C,A))$, 
we define the $G$-lattice $B_f=A\times C$ by the $G$-action 
\begin{align*}
g(a,c)=(ga+f(g)(gc),gc)\ (g\in G, a\in A, c\in C).
\end{align*}
Then we have an exact sequence of $G$-lattices
\begin{align*}
(E_f): 0\to A\xrightarrow{\iota}B_f\xrightarrow{\pi}C\to 0
\end{align*}
where 
\begin{align*}
\iota(a)&=(a,0),\\
\pi(a,c)&=c.
\end{align*}
Then, for $1$-cocycles $f_1,f_2\in Z^1(G,{\rm Hom}_\bZ(C,A))$, 
$(E_{f_1})$ is equivalent to $(E_{f_2})$ 
$($see Brown \cite[Chapter IV, page 86]{Bro82}$)$ 
if and only if $f_1-f_2\in B^1(G,{\rm Hom}_\bZ(C,A))$. 
In particular, the association
\begin{align*}
H^1(G,{\rm Hom}_\bZ(C,A))\to {\rm Ext}^1_{\bZ[G]}(C,A),\ [f]\mapsto E_f
\end{align*}
is bijective. 
\end{lemma}
\begin{proof}
See Brown \cite[IV.3, pages 91--94]{Bro82} for the similar case 
where $A$, $B_f$, $C$ are groups and $A$ is $C$-module 
(we can make a proof by using the similar arguments). 
\end{proof}
\begin{definition}\label{def5.2}
Let $(E): 0\to A\to B\to C\to 0$ be an exact sequence of $G$-lattices. 
{\it The order ${\rm ord}(E)$ of $(E)$} is defined to be 
the order of $(E)$ as an element in the group 
$\Ext^1_{\bZ[G]}(C,A)\simeq H^1(G,{\rm Hom}_{\bZ}(C,A))$. 
\end{definition}
\begin{lemma}\label{lem5.3}
Let $(E): 0\to A\xrightarrow{\iota} B\xrightarrow{\pi} C\to 0$ 
be an exact sequence of $G$-lattices 
with $e:={\rm ord}(E)$.\\ 
{\rm (1)} Let $s \in \mathrm{Hom}_\bZ(C,B)$ be such that 
$\pi\circ s\in {\rm Hom}_{\bZ[G]}(C,C)$. 
Define $f_s\in {\rm Map}(G,{\rm Hom}_\bZ(C,A))$ by 
\begin{align*}
f_s(g)(c)=(gs)(c)-s(c)=g\left(s(g^{-1}c)-s(c)\right)\ 
(g\in G, c\in C).
\end{align*}
Then {\rm (i)} $f_s=0$ if and only if $s\in {\rm Hom}_{\bZ[G]}(C,B)$, 
and {\rm (ii)} we have $f_s\in Z^1(G,{\rm Hom}_\bZ(C,A))$.\\
{\rm (2)} Let $s$, $s^\prime\in \mathrm{Hom}_\mathbb{Z}(C,B)$,
be such that $\pi\circ s$, $\pi\circ s^\prime \in {\rm Hom}_{\bZ[G]}(C,C)$. 
Define $f_s$, $f_{s^\prime}\in Z^1(G,{\rm Hom}_\bZ(C,A))$ as in $(1)$. 
Then $\pi\circ s=\pi\circ s^\prime$ if and only if 
$(f_s-f_{s^\prime})\in B^1(G,{\rm Hom}_\bZ(C,A))$. 
Conversely, if $f\in B^1(G,{\rm Hom}_\bZ(C,A))$, 
then there exists $s\in {\rm Hom}_\bZ(C,B)$ such that 
 $\pi\circ s=0$ and $f_s=f$.\\
{\rm (3)} For $s_1\in {\rm Hom}_{\bZ}(C,B)$, 
we assume that $\pi\circ s_1={\rm Id}_C\in {\rm Hom}_{\bZ[G]}(C,C)$ 
and hence we can take the unique element 
$[f_{s_1}]\in H^1(G,{\rm Hom}_\bZ(C,A))
=Z^1(G,{\rm Hom}_\bZ(C,A))/B^1(G,{\rm Hom}_\bZ(C,A))$ 
with ${\rm ord}([f_{s_1}])=e$ 
as in $(1)$ and $(2)$. 
Then there exists $\widetilde{s}\in {\rm Hom}_{\bZ[G]}(C,B)$ such that 
$[f_{\widetilde{s}}]=[f_{es_1}]=e[f_{s_1}]=0$ and $\pi\circ\widetilde{s}=e\,\Id_C$. 
Moreover, there exists 
$\overline{s}\in {\rm Hom}_{\bZ[G]}(C,B)$ such that 
$\pi\circ\overline{s}=m\,\Id_C$ if and only if $e\mid m$, 
i.e. $e$ divides $m$. 
\end{lemma}
\begin{proof}
(1) By $\pi\left(f_s(g)(c)\right)=(\pi\circ s)(c)-gg^{-1}(\pi\circ s)(c)=0$, 
we may regard 
$f_s(g)(c)\in A\simeq \iota(A)$ by the exactness of $(E)$ and hence 
$f_s\in {\rm Map}(G,{\rm Hom}_\bZ(C,A))$. 
For any $g_1,g_2\in G$, $c\in C$, we have 
\begin{align*}
f_s(g_1g_2)(c)&=g_1g_2\left(s(g_2^{-1}g_1^{-1}c)\right)-s(c)\\
&=g_1g_2\left(s(g_2^{-1}g_1^{-1}c)\right)
-g_1\left(s(g_1^{-1}c)\right)+g_1\left(s(g_1^{-1}c)\right)
-s(c)\\
&=(g_1f_s)(g_2)(c)+f_s(g_1)(c). 
\end{align*}
Hence $f_s$ becomes a $1$-cocyle of $G$ in ${\rm Hom}_\bZ(C,A)$, i.e. 
$f_s\in Z^1(G,{\rm Hom}_\bZ(C,A))$.\\
(2) We see that $\pi\circ s=\pi\circ s^\prime$ if and only if 
$(s-s^\prime)(c)\in \iota(A)$ for any $c\in C$. 
It follows from 
\begin{align*}
(f_s-f_{s^\prime})(g)(c)=g\left((s-s^\prime)(g^{-1}c)\right)-(s-s^\prime)(c)
=\left(g(s-s^\prime)\right)(c)-(s-s^\prime)(c)
\end{align*}
that $f_s-f_{s^\prime}$ is a $1$-coboundary of $G$ in ${\rm Hom}_\bZ(C,A)$, 
i.e. 
$f_s-f_{s^\prime}\in B^1(G,{\rm Hom}_\bZ(C,A))$.
Conversely, if $f\in B^1(G,{\rm Hom}_\bZ(C,A))$, 
then thre exists $r\in {\rm Hom}_\bZ(C,A)$ such that 
\begin{align*}
f(g)(c)=(gr)(c)-r(c)\ (g\in G, c\in C). 
\end{align*}
We can take $s=i\circ r\in {\rm Hom}_{\bZ}(C,B)$ with 
$\pi\circ s=(\pi\circ i)\circ r=0\circ r=0$ and 
$f_s=f_{i\circ r}=f$ because $\iota(A)\simeq A$.\\
(3) Because $C$ is free as a $\bZ$-module, 
we can take $s_1\in {\rm Hom}_\bZ(C,B)$ 
with $\pi\circ s_1=\Id_C$. 
It also follows from $\pi\circ s_1=\Id_C\in {\rm Hom}_{\bZ[G]}(C,C)$ 
and (1) that we can get the unique element 
$[f_{s_1}]\in H^1(G, {\rm Hom}_\bZ(C,A))$. 

For $m\in \bZ$, we define $ms_1\in {\rm Hom}_\bZ(C,B)$ 
by $(ms_1)(c)=m\left(s_1(c)\right)$. 
Then we have $\pi\circ(ms_1)=m(\pi\circ s_1)=m\,\Id_C$. 
We also see that 
\begin{align*}
f_{ms_1}(g)(c)=g\left((ms_1)(g^{-1}c)\right)-(ms_1)(c)
=mf_{s_1}(g)(c).
\end{align*}
In particular, for $m=e$, we have $[f_{es_1}]=e[f_{s_1}]=0$ 
because ${\rm ord}(f_{s_1})=e$. 
Hence it follows from (2) that there exists 
$s_0\in {\rm Hom}_{\bZ[G]}(C,A)$ such that 
$f_{s_0}=f_{es_1}$ 
and 
$\pi\circ s_0=0$ with $[f_{s_0}]=0$. 
Define $\widetilde{s}:=es_1-s_0$. 
Then we get $f_{\widetilde{s}}=f_{es_1}-f_{s_0}=0$, 
$\pi\circ\widetilde{s}=\pi(es_1)-\pi(s_0)=\pi(es_1)
=e(\pi\circ s_1)=e\,\Id_C$. 
For the last statement, we see that 
$\pi\circ\overline{s}=m\,\Id_C$ if and only if 
$[f_{\overline{s}}]=[f_{ms_1}]=m[f_{s_1}]=0$ 
if and only if $e\mid m$. 
\end{proof}
\begin{lemma}\label{lem5.4}
Let $(E): 0\to A\xrightarrow{\iota} B\xrightarrow{\pi} C\to 0$ 
be an exact sequence of $G$-lattices.\\ 
\noindent The following conditions are equivalent:\\
{\rm (1)} $e={\rm ord}(E)$, i.e. 
the order of an element 
$(E) \in \Ext^1_{\bZ[G]}(C,A)\simeq H^1(G,{\rm Hom}_{\bZ}(C,A))$ 
as in Definition \ref{def5.2};\\ 
{\rm (2)} The smallest positive integer $e\in \bZ$ satisfies that 
there exists 
$s\in {\rm Hom}_{\bZ[G]}(C,B)$ 
such that $\pi \circ s= e\,\Id_C$;\\
{\rm (3)} The smallest positive integer $e\in \bZ$ satisfies 
that there exists 
$t\in {\rm Hom}_{\bZ[G]}(B,A)$ such that $t\circ \iota= e\,\Id_A$. 
\end{lemma}
\begin{proof}
$(1)\Leftrightarrow (2)$: This follows from Lemma \ref{lem5.3} (3).\\
$(1)\Leftrightarrow (3)$: 
We can take the dual 
$(F):0\to C^\circ\xrightarrow{\pi^\circ} B^\circ 
\xrightarrow{\iota^\circ} A^\circ\to 0$ 
with ${\rm ord}(F)={\rm ord}(E)=e$. 
Then by Lemma \ref{lem5.3} there exists 
$s^\circ\in {\rm Hom}_{\bZ[G]}(A^\circ,B^\circ)$ 
such that $\iota^\circ\circ s^\circ=e\,\Id_{A^\circ}$. 
By taking the dual again, 
we have $t=(s^\circ)^{\circ}\in {\rm Hom}_{\bZ[G]}(B,A)$ 
such that $t\circ \iota=t \circ ((\iota^\circ)^\circ)
=e\,\Id_{(A^\circ)^\circ}=e\,\Id_A$. 
\end{proof}
%
\subsection{Tensor products of extensions of $G$-lattices}
\begin{proposition}\label{prop5.6}
Consider two exact sequences of $G$-lattices
\begin{align*}
&(E_1): 0 \to A_1 \xrightarrow{\iota_1}B_1 \xrightarrow{\pi_1}C_1\to 0,\\
&(E_2): 0 \to A_2 \xrightarrow{\iota_2}B_2 \xrightarrow{\pi_2}C_2\to 0
\end{align*}
with 
$e_1={\rm ord}(E_1)$, $e_2={\rm ord}(E_2)$. 
Then there exists an exact sequence of $G$-lattices
\begin{align*}
(E_{1,2}^{(1)}): 0 \to A_1 \otimes A_2 \xrightarrow{\iota} 
B_1 \otimes B_2 
 \xrightarrow{f} (C_1 \otimes B_2) \oplus (B_1 \otimes C_2)
 \xrightarrow{\pi}
 C_1 \otimes C_2 \to 0
\end{align*}
where 
\begin{align*}
\iota(a_1\otimes a_2)&=\iota_1(a_1)\otimes \iota_2(a_2),\\
f(b_1\otimes b_2)&=(\pi_1(b_1)\otimes b_2,b_1\otimes\pi_2(b_2)),\\
\pi(c_1\otimes b_2,b_1\otimes c_2)&=c_1\otimes\pi_2(b_2)-\pi_1(b_1)\otimes c_2. %
\end{align*}
{\rm (1)} If $\gcd(e_1,e_2)=1$, 
then the following exact sequence of $G$-lattices splits: 
\begin{align*}
(E_{1,2}^{(1R)}): 0\to {\rm Image}(f)\to (C_1 \otimes B_2) \oplus (B_1 \otimes C_2)
\mathrel{\mathop{\rightleftarrows}^{\mathrm{\pi}}_{s}}
 C_1 \otimes C_2 \to 0, 
\end{align*}
i.e. ${\rm Image}(f)\oplus (C_1 \otimes C_2)\simeq (C_1 \otimes B_2) \oplus (B_1 \otimes C_2)$, that is, there exists 
$s\in {\rm Hom}_{\bZ[G]}(C_1\otimes C_2, 
(C_1 \otimes B_2)\oplus (B_1 \otimes C_2))$ 
such that $\pi\circ s=\Id_{C_1\otimes C_2}$.\\ 
{\rm (2)} If $\gcd(e_1,e_2)=1$, 
then there exists an exact sequence of $G$-lattices
\begin{align*}
(E_{1,2}^{(1RS)}): 0 \to A_1 \otimes A_2 \xrightarrow{(\iota,0)} 
(B_1 \otimes B_2)\oplus(C_1 \otimes C_2)
 \xrightarrow{(f,s)} (C_1 \otimes B_2) \oplus (B_1 \otimes C_2)
 \to 0
\end{align*}
with ${\rm ord}(E_{1,2}^{(1RS)})\mid e_1e_2$. 
\end{proposition}
\begin{proof}
We first show that $(E_{1,2}^{(1)})$ is exact. 
We see that $f \circ \iota=0$, $\pi \circ f=0$ because 
\begin{align*}
&f \circ \iota(a_1 \otimes a_2)
=f\big(\iota_1(a_1) \otimes \iota_2(a_2)\big)
=\big(\pi_1(\iota_1(a_1))\otimes \iota_2(a_2),\iota_1(a_1)\otimes \pi_2(\iota_2(a_2))\big)
=\big(0 \otimes \iota_2(a_2),\iota_1(a_1) \otimes 0\big)=0,\\
&\pi \circ f(b_1 \otimes b_2)
=\pi\big(\pi_1(b_1) \otimes b_2, b_1 \otimes \pi_2(b_2) \big)
=\pi_1(b_1) \otimes \pi_2(b_2)
-\pi_1(b_1) \otimes \pi_2(b_2)=0. 
\end{align*}
(i) ${\rm Ker}(\iota)=0$. 
Because $G$-lattices are flat (since $\bZ$-free), 
by tensoring $A_2$ and $B_1$, 
we see that 
$\iota\in {\rm Hom}_{\bZ[G]}(A_1 \otimes A_2,B_1 \otimes B_2)$ 
is injective because 
$\iota$ is the composition of two injections:
\begin{align*}
\iota: A_1\otimes A_2\xrightarrow{\iota_1\otimes\Id_{A_2}} B_1\otimes A_2\xrightarrow{\Id_{B_1}\otimes \iota_2} B_1\otimes B_2. 
\end{align*}
(ii) ${\rm Image}(\iota)={\rm Ker}(f)$. 
We can take a section $s_i\in {\rm Hom}_\bZ(C_i,B_i)$ $(i=1,2)$ 
with $\pi_i\circ s_i={\rm Id}_{C_i}$ 
because $C_i$ is free as a $\bZ$-module 
and hence we have $B_i\simeq \iota_i(A_i)\oplus s_i(C_i)\simeq A_i\oplus C_i$ 
$(i=1,2)$ as a $\bZ$-module. 
Then 
\begin{align*}
f: B_1 \otimes B_2\to (C_1 \otimes B_2) \oplus (B_1 \otimes C_2)
\end{align*} 
can be written as 
\begin{align*}
f:\ &(\iota_1(A_1)\oplus s_1(C_1)) \otimes (\iota_2(A_2)\oplus s_2(C_2))\\
&=(\iota_1(A_1)\otimes \iota_2(A_2))\oplus(\iota_1(A_1)\otimes s_2(C_2))
\oplus(s_1(C_1)\otimes \iota_2(A_2))\oplus(s_1(C_1)\otimes s_2(C_2))\\
\to\ &(C_1 \otimes (\iota_2(A_2)\oplus s_2(C_2))) \oplus ((\iota_1(A_1)\oplus s_1(C_1)) \otimes C_2)\\
&=(C_1 \otimes\iota_2(A_2))\oplus (C_1 \otimes s_2(C_2))\oplus 
(\iota_1(A_1) \otimes C_2)\oplus (s_1(C_1) \otimes C_2).
\end{align*}
We find that 
\begin{align*}
&f|_{\iota_1(A_1)\otimes \iota_2(A_2)}=0,\\ 
&f|_{\iota_1(A_1)\otimes s_2(C_2)}: \iota_1(A_1)\otimes s_2(C_2)\xrightarrow{\sim} \iota_1(A_1) \otimes C_2,\\ 
&f|_{s_1(C_1)\otimes \iota_2(A_2)}: s_1(C_1)\otimes \iota_2(A_2)\xrightarrow{\sim} C_1 \otimes\iota_2(A_2),\\ 
&f|_{s_1(C_1)\otimes s_2(C_2)}: s_1(C_1)\otimes s_2(C_2)\to 
(C_1 \otimes s_2(C_2))\oplus (s_1(C_1) \otimes C_2),\\ 
&\hspace*{36mm}x\mapsto \big((\pi_1\otimes\Id_{s_2(C_2)})(x), (\Id_{s_1(C_1)}\otimes\pi_2)(x)\big). 
\end{align*}
This implies that ${\rm Ker}(f)=\iota_1(A_1)\otimes \iota_2(A_2)={\rm Image}(\iota)$.\\
(iii) ${\rm Image}(f)={\rm Ker}(\pi)$ and ${\rm Image}(\pi)=C_1\otimes C_2$. 
As in (ii), we can take a section $s_i\in {\rm Hom}_\bZ(C_i,B_i)$ $(i=1,2)$ 
with $\pi_i\circ s_i={\rm Id}_{C_i}$ 
because $C_i$ is free as a $\bZ$-module 
and hence we have $B_i\simeq \iota_i(A_i)\oplus s_i(C_i)\simeq A_i\oplus C_i$ 
$(i=1,2)$ as a $\bZ$-module. 
Then 
\begin{align*}
\pi: (C_1 \otimes B_2) \oplus (B_1 \otimes C_2)\to C_1 \otimes C_2
\end{align*}
can be written as 
\begin{align*}
\pi:\ &(C_1 \otimes (\iota_2(A_2)\oplus s_2(C_2))) \oplus ((\iota_1(A_1)\oplus s_1(C_1)) \otimes C_2)\\
&=(C_1 \otimes\iota_2(A_2))\oplus (C_1 \otimes s_2(C_2))\oplus 
(\iota_1(A_1) \otimes C_2)\oplus (s_1(C_1) \otimes C_2)\to C_1 \otimes C_2. 
\end{align*}
It follows from (ii) and 
\begin{align}
\Id_{C_1\otimes C_2}: C_1\otimes C_2\xrightarrow[\sim]{s_1\otimes s_2}
s_1(C_1)\otimes s_2(C_2)\xrightarrow[\sim]{\pi_1\otimes\pi_2} C_1\otimes C_2\tag{$*$}
\end{align}
that 
\begin{align*}
{\rm Image}(f)\overset{\rm (ii)}{=}&\ \bigl\{
(x_1,x_2,x_3,x_4)\in (C_1 \otimes\iota_2(A_2))\oplus (C_1 \otimes s_2(C_2))\oplus 
(\iota_1(A_1) \otimes C_2)\oplus (s_1(C_1) \otimes C_2)\\
&\ \ \mid 
(s_1\otimes\Id_{s_2(C_2)})(x_2)=(\Id_{s_1(C_1)}\otimes s_2)(x_4)
\bigr\}\\
\overset{(*)}{=}&\ \bigl\{
(x_1,x_2,x_3,x_4)\in (C_1 \otimes\iota_2(A_2))\oplus (C_1 \otimes s_2(C_2))\oplus 
(\iota_1(A_1) \otimes C_2)\oplus (s_1(C_1) \otimes C_2)\\
&\ \ \mid 
(\Id_{C_1}\otimes \pi_2)(x_2)=(\pi_1\otimes\Id_{C_2})(x_4)
\bigr\}\\
\overset{(**)}{=}&\ {\rm Ker}(\pi).
\end{align*}
The last equality $(**)$ and ${\rm Image}(\pi)=C_1\otimes C_2$ follow from 
\begin{align*}
&\pi|_{C_1 \otimes\iota_2(A_2)}=0,\\ 
&\pi|_{C_1 \otimes s_2(C_2)}: C_1 \otimes s_2(C_2)
\xrightarrow{\sim} C_1\otimes C_2,\ x\mapsto (\Id_{C_1}\otimes \pi_2)(x),\\
&\pi|_{\iota_1(A_1) \otimes C_2}=0,\\ 
&\pi|_{s_1(C_1) \otimes C_2}: s_1(C_1) \otimes C_2
\xrightarrow{\sim} C_1\otimes C_2,\ x\mapsto -(\pi_1\otimes\Id_{C_2})(x).
\end{align*}
(1) By Lemma \ref{lem5.4}, we can take 
$s_i\in {\rm Hom}_{\bZ[G]}(C_i,B_i)$ 
such that $\pi_i \circ s_i= e_i\,\Id_{C_i}$ $(i=1,2)$. 
By the assumption $\gcd(e_1,e_2)=1$, there exist 
integers $u,v$ such that $ve_2-ue_1=1$. 
Then we can define the map 
\begin{align*}
s : &\ C_1\otimes C_2\rightarrow (C_1 \otimes B_2)\oplus (B_1 \otimes C_2)\\
&\ c_1\otimes c_2\mapsto (vc_1\otimes s_2(c_2), us_1(c_1)\otimes c_2)
\end{align*}
which becomes a splitting of $\pi$, 
i.e. $\pi\circ s=\Id_{C_1\otimes C_2}$. 
By Lemma \ref{lem5.4}, this implies that $(E_{1,2}^{(1R)})$ splits.\\ 
(2) It follows from ${\rm Coker}(f)=C_1\otimes C_2$ and (1) above 
that $(E_{1,2}^{(1R)})$ is exact. 
For the last assertion, 
by Lemma \ref{lem5.4}, we can take 
$t_i\in {\rm Hom}_{\bZ[G]}(B_i,A_i)$ 
such that $t_i \circ \iota_i= e_i\,\Id_{A_i}$ $(i=1,2)$. 
Then 
we see that 
\begin{align*}
A_1\otimes A_2\mathrel{\mathop{\rightleftarrows}^{\mathrm{(\iota,0)}}_{(t,0)}} (B_1 \otimes B_2)\oplus(C_1 \otimes C_2)
\end{align*}
where $(t,0):=(t_1\otimes t_2,0)$ 
satisfies that $(t,0)\circ (\iota,0)=e_1e_2\,\Id_{A_1\otimes A_2}$. 
By Lemma \ref{lem5.4}, we get ${\rm ord}(E_{1,2}^{(1RS)})\mid e_1e_2$. 
\end{proof}

We also get the dual version of Proposition \ref{prop5.6}: 

\begin{proposition}\label{prop5.5}
Consider two exact sequences of $G$-lattices
\begin{align*}
&(E_1): 0 \to A_1 \xrightarrow{\iota_1}B_1 \xrightarrow{\pi_1}C_1\to 0,\\
&(E_2): 0 \to A_2 \xrightarrow{\iota_2}B_2 \xrightarrow{\pi_2}C_2\to 0
\end{align*}
with 
$e_1={\rm ord}(E_1)$, $e_2={\rm ord}(E_2)$. 
Then there exists an exact sequence of $G$-lattices
\begin{align*}
(E_{1,2}^{(2)}): 0 \to A_1 \otimes A_2 \xrightarrow{\iota} (A_1 \otimes B_2) \oplus (B_1 \otimes A_2) \xrightarrow{f} 
B_1 \otimes B_2 \xrightarrow{\pi}
 C_1 \otimes C_2 \to 0
\end{align*}
where 
\begin{align*}
\iota(a_1 \otimes a_2) &= (a_1 \otimes \iota_2(a_2), -\iota_1(a_1) \otimes a_2),\\
f (a_1 \otimes b_2, b_1 \otimes a_2) &= \iota_1(a_1) \otimes b_2+ b_1 \otimes \iota_2(a_2),\\
\pi(b_1\otimes b_2) & =\pi_1(b_1)\otimes\pi_2(b_2).
\end{align*}
{\rm (1)} If $\gcd(e_1,e_2)=1$, 
then the following exact sequence of $G$-lattices splits: 
\begin{align*}
(E_{1,2}^{(2L)}): 0 \to A_1 \otimes A_2 
\mathrel{\mathop{\rightleftarrows}^{\mathrm{\iota}}_{t}}
(A_1 \otimes B_2)\oplus(B_1 \otimes A_2)\xrightarrow{f}{\rm Image}(f)\to 0, 
\end{align*}
i.e. $(A_1 \otimes A_2)\oplus {\rm Image}(f)\simeq (A_1 \otimes B_2)\oplus(B_1 \otimes A_2)$, that is, there exists 
$t\in {\rm Hom}_{\bZ[G]}((A_1 \otimes B_2)\oplus(B_1 \otimes A_2), 
A_1 \otimes A_2)$ 
such that $t\circ\iota=\Id_{A_1\otimes A_2}$.\\
{\rm (2)} If $\gcd(e_1,e_2)=1$, 
then there exists an exact sequence of $G$-lattices
\begin{align*}
(E_{1,2}^{(2LS)}): 0 \to (A_1 \otimes B_2) \oplus (B_1 \otimes A_2) \xrightarrow{(t,f)}  (A_1 \otimes A_2)\oplus (B_1 \otimes B_2) \xrightarrow{(0,\pi)}
 C_1 \otimes C_2 \to 0
\end{align*}
with ${\rm ord}(E_{1,2}^{(2LS)})\mid e_1e_2$. 
\end{proposition}
\begin{proof}
That $(E_{1,2}^{(2)})$ is exact is a straightforward verification. 
Alternatively, we can get that $(E_{1,2}^{(2)})$ is exact 
by taking the dual of $(E_{1,2}^{(1)})$ as in Proposition \ref{prop5.6} 
(see also the proof of Proposition \ref{prop5.6}).\\
{\rm (1)} By Lemma \ref{lem5.4}, we can take 
$t_i\in {\rm Hom}_{\bZ[G]}(B_i,A_i)$ 
such that $t_i \circ \iota_i= e_i\,\Id_{A_i}$ $(i=1,2)$. 
By the assumption $\gcd(e_1,e_2)=1$, there exist 
integers $u,v$ such that $ve_2-ue_1=1$. 
Then we can define the map
\begin{align*}
t : &\ (A_1 \otimes B_2)\oplus (B_1 \otimes A_2) \rightarrow A_1 \otimes A_2,\\
&\ (a_1 \otimes b_2, b_1\otimes a_2)\mapsto v a_1\otimes t_2(b_2)+ut_1(b_1)\otimes a_2
\end{align*}
which becomes a retraction of $\iota$, 
i.e. $t\circ \iota=\Id_{A_1\otimes A_2}$. 
By Lemma \ref{lem5.4}, this implies that $(E_{1,2}^{(2L)})$ splits.\\
(2) It follows from ${\rm Ker}(f)={\rm Image}(\iota)\simeq A_1\otimes A_2$ 
and (1) above that $(E_{1,2}^{(2LS)})$ is exact. 
For the last assertion, 
by Lemma \ref{lem5.4}, we can take 
$s_i\in {\rm Hom}_{\bZ[G]}(C_i,B_i)$ 
such that $\pi_i\circ s_i= e_i\,\Id_{C_i}$ $(i=1,2)$. 
Then 
we see that 
\begin{align*}
B_1\otimes B_2\mathrel{\mathop{\rightleftarrows}^{\mathrm{\pi}}_{s_1\otimes s2}} (C_1 \otimes C_2)
\end{align*}
satisfies that $\pi\circ(s_1\otimes s_2)=e_1e_2\,\Id_{C_1\otimes C_2}$. 
By Lemma \ref{lem5.4}, we get ${\rm ord}(E_{1,2}^{(2LS)})\mid e_1e_2$. 
\end{proof}
%
Let $T=T_1\otimes \cdots \otimes T_r$ be 
the tensor product of algebraic $k$-tori $T_1,\ldots,T_r$ 
as in Definiiton \ref{deftensorT}. 

As an application of the preceding constructions as in 
Proposition \ref{prop5.6}, we derive the following result.
\begin{theorem}[{see also Voskresenskii \cite[Section 6.3, pages 69--71, Example 6, Example 7, pages 100--102]{Vos98}}]\label{thKly}
Let $k$ be a field and $G$ be a finite group. 
Let $X_i$ $(1\leq i\leq r)$ be a finite $G$-set with $|X_i|=n_i$. 
Recall the 
exact sequences of $G$-lattices 
\begin{align*}
&(E_{X_i}): 0\to I_{X_i}\to \bZ[X_i]\xrightarrow{\varepsilon_i} \bZ\to 0
\end{align*}
as in Definition \ref{def8.1}. 
Define the permutation $G$-lattices 
\begin{align*}
P^+:=\bigoplus_{A \subset \{1,\ldots,r\}, \;  \vert A \vert \equiv r (\mathrm{mod} \; 2)} \bZ [\Pi_{i\in A} X_i]\quad 
(e.g.\ P^+=\bZ[X_1 \times X_2] \oplus \bZ\ {\rm if}\ r=2)
\end{align*}
and 
\begin{align*}
P^-:=\bigoplus_{A \subset \{1,\ldots,r\}, \; \vert A \vert \equiv r+1 (\mathrm{mod} \; 2)} \bZ [\Pi_{i\in A} X_i]\quad 
(e.g.\ P^-=\bZ[X_1] \oplus \bZ[X_2]\ {\rm if}\ r=2).
\end{align*}
If $\gcd(n_i,n_j)=1$ for any $1\leq i<j\leq r$, 
then there exists a natural exact sequence of $G$-lattices
\begin{align*}
(E^{r}): 0\to I_{X_1}\otimes \cdots \otimes I_{X_r}\to P^+\to P^-\to 0
\end{align*}
with ${\rm ord}(E^r)\mid n_1 \cdots n_r$.
In particular, we have $[I_{X_1}\otimes\cdots\otimes I_{X_r}]^{fl}=[P^-]=0$. 

In other wards, 
let $h: \Gal(\overline k/k) \to G$ be a continuous homomorphism, 
$A_i$ $(1\leq i\leq r)$ be the \'etale $k$-algebra 
corresponding to $X_i$ via $h$ and 
$T_i:=R_{A_i/k}(\bG_m)/\bG_{m,k}$ be the algebraic $k$-torus 
with character lattice $\widehat{T}_i\simeq I_{X_i}$. 
Then the algebraic $k$-torus $T:=T_1 \otimes \ldots \otimes T_r$ 
with character lattice 
$\widehat{T}\simeq I_{X_1}\otimes\cdots\otimes I_{X_r}$ 
is stably $k$-rational. 
\end{theorem}

\begin{proof}
If $r=2$, we have $P^+=\bZ[X_1 \times X_2] \oplus \bZ$, 
$P^-=\bZ[X_1] \oplus \bZ[X_2]$ and 
\begin{align*}
(E^2): 0\to I_{X_1}\otimes I_{X_2}\to\bZ[X_1\times X_2]\oplus\bZ\to\bZ[X_1]\oplus\bZ[X_2]\to 0.
\end{align*}
This follos from the exact sequence $(E_{1,2}^{(1RS)})$ as in 
Proposition \ref{prop5.6} (2) 
by applying $(E_i)=(E_{X_i})$ $(i=1,2)$. 
The general case is by induction on $r \geq 2$. 
For the induction step, observe that 
$\gcd({\rm ord}(E^r)$,${\rm ord}(E^{r+1}))=1$, 
so that one may apply Proposition \ref{prop5.6} (2) 
to the extensions $(E^r)$ and $(E^{r+1})$. 
It is a straightforward verification, that the resulting extension 
is the desired $(E^{r+1})$. 
In particular, we have $[\widehat{T}]^{fl}
=[I_{X_1}\otimes\cdots\otimes I_{X_r}]^{fl}=[P^-]=0$ 
and hence it follows from Theorem \ref{th2.7} that 
$T$ is stably $k$-rational. 
\end{proof}

\begin{remark}\label{r5.8}
When $r=2$, if $\gcd(n_1,n_2)=1$ where $n_i={\rm dim}\,(A_i)$ $(i=1,2)$, 
then the algebraic $k$-torus $T=T_1\otimes T_2$ as in Theorem \ref{thKly} 
is $k$-rational 
by a result of Klyachko \cite{Kly88} 
(see also Florence and van Garrel \cite[Theorem 2.1]{FvG17}). 
However, it is an open problem whether 
$T=T_1 \otimes \ldots \otimes T_r$ is $k$-rational or not when $r\geq 3$ 
(see Voskresenskii \cite[Section 6.3, page 71]{Vos98}). 
\end{remark}

The results given in Section \ref{S6} and Section \ref{S7} will improve 
on Theorem \ref{thKly}, using a similar method based on 
Proposition \ref{prop5.6}. Applications shall be given to the case of 
a tensor product of norm one tori $\T_{A_i}=R_{A_i}^{(1)}(\bG_m)$, 
which is arguably more interesting than the (dual) case of 
$R_{A_i/k}(\bG_m)/\bG_{m,k}$, dealt as in Theorem \ref{thKly}. 

\subsection{Tensor products of Chevalley modules $J_{G/H}$ and an application of Theorem \ref{thKly} and Remark \ref{r5.8}}

We here collect some 
results which are related to the tensor products of 
Chevalley modules $J_{G/H}$ 
and by applying Theorem \ref{thKly} and Remark \ref{r5.8}, 
we get the $k$-rationality of norm one tori 
$\T_K=R_{K/k}^{(1)}(\bG_m)$ with ${\rm Gal}(L/k)\simeq D_n$ 
for odd integers $n\geq 3$.

\begin{theorem}[{Endo and Miyata \cite[Lemma 2.3]{EM82}}]
Let $G$ be a finite group. 
We have an exact sequence of $G$-lattices 
\begin{align*}
0\to J_G\to \bZ[G]^{\oplus |G|-1}\to (J_G)^{\otimes 2}\to 0
\end{align*}
and $(J_G)^{\otimes 2}$ is flabby and coflabby. 
In particular, $[J_G]^{fl}=[(J_G)^{\otimes 2}]$. 
\end{theorem}

\begin{theorem}[{Lemire \cite[Proposition 5.11]{Lem}, see also Colliot-Th\'{e}l\`{e}ne and Sansuc \cite[Proposition 9.1]{CTS87}, Bessenrodt and Le Bruyn \cite[page 188]{BLB91}}]
Let $p$ be a prime number 
and $S_p$ be the symmetric group of degree $p$. 
Then we have an exact sequence of $S_p$-lattices 
\begin{align*}
0\to J_{S_p/S_{p-1}}\to \bZ[S_p/S_{p-2}]\to 
(J_{S_p/S_{p-1}})^{\otimes 2}\to 0
\end{align*}
and $(J_{S_p/S_{p-1}})^{\otimes 2}$ is invertible. 
In particular, 
$[J_{S_p/S_{p-1}}]^{fl}=[(J_{S_p/S_{p-1}})^{\otimes 2}]$ is invertible 
and hence the norm one torus 
$\T_K$ with $\widehat{\T}_K\simeq J_{S_p/S_{p-1}}$ is retract $k$-rational. 
\end{theorem}

\begin{theorem}[{Beneish \cite[page 3572]{Ben98}, Lemire \cite[Lemma 5.1]{Lem}, see also Bessenrodt and Le Bruyn \cite[page 180]{BLB91}, Lorenz \cite[page 46, page 146]{Lor05}}]
Let $S_n$ be the symmetric group of degree $n$. 
Let $u_i$ $(1\leq i\leq n)$ be a basis of $\bZ[S_n/S_{n-1}]$ 
with $\sigma(u_i)=u_{\sigma(i)}$ for any $\sigma\in S_n$. 
Let $y_{i,j}$ $(1\leq i,j\leq n, i\neq j)$ be a basis of $\bZ[S_n/S_{n-2}]$ 
with $\sigma(y_{i,j})=y_{\sigma(i),\sigma(j)}$ for any $\sigma\in S_n$. 
Then we have an isomorphism 
\begin{align*}
f: I_{S_n/S_{n-1}}\otimes\bZ[S_n/S_{n-1}]\xrightarrow{\sim}\bZ[S_n/S_{n-2}],\ 
(u_i-u_j)\otimes u_i\mapsto y_{i,j}.
\end{align*}
\end{theorem}
\begin{theorem}[{Beneish \cite[Proposition 1.1]{Ben98}, see also Lorenz \cite[Proposition 2.12.2]{Lor05}, Lemire \cite[Lemma 5.2]{Lem}}]
Let $p$ be a prime number and $S_p$ be the symmetric group of degree $p$.  
Then 
\begin{align*}
(J_{S_p/S_{p-1}}\otimes I_{S_p/S_{p-1}})\oplus\bZ[S_p/S_{p-1}]\simeq \bZ[S_p/S_{p-2}]\oplus\bZ.
\end{align*}
\end{theorem}

\begin{proposition}[{Lemire \cite[Proposition 4.5]{Lem}}]\label{prop5.14}
Let $n\geq 3$ be an odd integer and $D_n$ be the dihedral group of order $2n$. 
Then 
$J_{D_n/C_2}\simeq I_{D_n/C_2}\otimes I_{D_n/C_n}$. 
\end{proposition}

By combining Theorem \ref{thKly}, Remark \ref{r5.8} and 
Propositon \ref{prop5.14}, 
we get the $k$-rationality of $\T_K=R_{K/k}^{(1)}(\bG_m)$ 
with ${\rm Gal}(L/k)\simeq D_n$ for odd integers $n\geq 3$ 
(see also Theorem \ref{th1.14} for the stable $k$-rationality): 
\begin{proposition}
Let $n\geq 3$ be an odd integer and $D_n$ be the dihedral group of order $2n$. 
Let $K/k$ be a separable 
field extension of degree $n$ and $L/k$ be the Galois closure 
of $K/k$ with $G={\rm Gal}(L/k)\simeq D_n$.  
Let $\T_K=R_{K/k}^{(1)}(\bG_m)$ be the norm one torus of $K/k$ with 
$\widehat{\T}_K\simeq J_{D_n/C_2}$ and 
$T^\prime$ $($resp. $T^{\prime\prime}$$)$ be an algebraic $k$-torus with 
$\widehat{T}^\prime\simeq I_{D_n/C_2}$ 
$($resp. $\widehat{T}^{\prime\prime}\simeq I_{D_n/C_n}$$)$. 
Then $\T_K=R_{K/k}^{(1)}(\bG_m)\simeq T^\prime\otimes T^{\prime\prime}$ 
with $\widehat{\T}_K\simeq J_{D_n/C_2}\simeq I_{D_n/C_2}\otimes I_{D_n/C_n}$ 
is $k$-rational. 
\end{proposition}

\section{Rationality of tensor products $T_1\otimes T_2$ of algebraic $k$-tori $T_1$ and $T_2$}\label{S6} 
In this section, $G$ is a finite group.

\subsection{Permutation order ${\rm p}$-${\rm ord}(M)$ and invertibility of $G$-lattices $M$}

The following characterisation of invertible $G$-lattices is well-known. 

\begin{lemma}[{Lenstra \cite[Proposition 1.2]{Len74}, see also 
Endo and Miyata \cite[Lemma 1.2]{EM75}, 
Colliot-Th\'{e}l\`{e}ne and Sansuc \cite[Lemme 9, page 182]{CTS77}, 
Swan \cite[Corollary 2.5]{Swa10}}]\label{lem6.3} 
Let $E$ be a $G$-lattice. 
Then $E$ is invertible if and only if  
any short exact sequence $0\to C\to N\to E\to 0$ of $G$-lattices 
with $C$ coflabby splits, 
i.e. ${\rm Ext}^1_{\bZ[G]}(E,C)=0$ for any coflabby $G$-lattice $C$. 
\end{lemma}
The following definition appears to be new.
\begin{definition}\label{DefiPermutationOrder}
Let $M$ be a $G$-lattice. 
{\it The permutation order} ${\rm p}$-${\rm ord}(M)$ of $M$ 
is the smallest integer $a \geq 1$ 
such that 
$a\,\Id_M: M\to M$ 
factors through a permutation $G$-lattice 
as a $\bZ[G]$-homomorphism 
where $\Id_M: M\to M$ is the identity map of $M$, 
i.e. there exists a permutation $G$-lattice $P$ such that 
$a\,\Id_M: M\xrightarrow{f} P\xrightarrow{g} M$ with $g\circ f=a\,\Id_M$. 
\end{definition}
\begin{lemma}\label{lem6.2}
Let $M$ be a $G$-lattice and $M^\circ={\rm Hom}(M,\bZ)$ be the dual of $M$. 
Then ${\rm p}$-${\rm ord}(M)={\rm p}$-${\rm ord}(M^\circ)$.
\end{lemma}
\begin{proof}
By considering the dual 
$\Id_{M^\circ}: M^\circ\rightarrow P^{\circ}\rightarrow M^\circ$ 
of $\Id_M: M\to P\to M$ where $P$ is permutation, 
we see the assertion 
because $P^\circ$ is permutation. 
\end{proof}
\begin{lemma}\label{lemRes}
Let $G$ be a finite group, 
$p$ be a prime number 
and $G_p={\rm Syl}_p(G)$ be a $p$-Sylow subgroup of $G$. 
Let $M$ be a $G$-lattice and 
$M|_{G_p}$ be a $G_p$-lattice 
obtained by restricting the action of $G$ on $M$ to $G_p$. 
Then $H^1(G,M)\xrightarrow{\oplus {\rm res}_p} \oplus_{p\mid |G|} H^1(G_p,M|_{G_p})$ becomes injective. 
In particular, if $H^1(G_p,M|_{G_p})=0$ for any prime divisor $p\mid |G|$, then $H^1(G,M)=0$.  
\end{lemma}
\begin{proof}
We see that the composite map 
${\rm cores}_p\circ {\rm res}_p: H^1(G,M)\xrightarrow{{\rm res}_p} H^1(G_p,M|_{G_p})\xrightarrow{{\rm cores}_p} H^1(G,M)$ 
is the multiplication by $[G:G_p]$ which is coprime to $p$ 
(see e.g. Serre \cite[Proposition 4, page 130]{Ser79}, Neukirch, Schmidt and Wingberg \cite[Corollary 1.5.7]{NSW00}). 
If $\alpha \in {\rm Ker}\{H^1(G,M)\xrightarrow{\oplus {\rm res}_p} \oplus_{p\mid |G|} H^1(G_p,M|_{G_p})\}$, 
then $[G:G_p]\alpha=0$ for all $p\mid |G|$. 
This implies that $\alpha=0\in H^1(G,M)$. 
We get the injection $H^1(G,M)\hookrightarrow \oplus_{p\mid |G|} H^1(G_p,M|_{G_p})$. 
\end{proof}

\begin{proposition}\label{prop6.5}
Let $G$ be a finite group 
and $M$, $N$ be $G$-lattices. 
Let $(E): 0 \to C \xrightarrow{\iota} N \xrightarrow{\pi} M \to 0$ 
be an exact sequence of $G$-lattices with $C$ coflabby. 
Let $p$ be a prime number, 
$G_p={\rm Syl}_p(G)$ be a $p$-Sylow subgroup of $G$ 
and 
$M|_{G_p}$ be a $G_p$-lattice 
obtained by restricting the action of $G$ on $M$ to $G_p$. 
Then\\
{\rm (1)} ${\rm ord}(E)\mid |G|$, i.e. ${\rm ord}(E)$ divides $|G|$.\\ 
{\rm (2)} ${\rm ord}(E)\mid {\rm p}$-${\rm ord}(M)$.\\
{\rm (3)} ${\rm p}$-${\rm ord}(M)={\rm ord}(E)$ if 
$N$ is permutation, i.e. $(E)$ is a coflabby resolution of $M$. 
In particular, if $a\,\Id_M: M\to M$ 
factors through a permutation $G$-lattice 
as a $\bZ[G]$-homomorphism, 
i.e. there exists a permutation $G$-lattice $P$ such that 
$a\,\Id_M: M\xrightarrow{f} P\xrightarrow{g} M$ with $g\circ f=a\,\Id_M$, 
then ${\rm p}$-${\rm ord}(M)\mid a$.\\
{\rm (4)} {${\rm p}$-${\rm ord}(M)$} $\mid$ $|G|$.\\
{\rm (5)} ${\rm p}$-${\rm ord}(M)=\prod_{p\mid |G|}{\rm p}$-${\rm ord}(M|_{G_p})$ where ${\rm p}$-${\rm ord}(M|_{G_p})\mid |G_p|$.\\
{\rm (6)} The following conditions are equivalent:\\
{\rm (i)} ${\rm p}$-${\rm ord}(M)=1$;\\
{\rm (ii)} ${\rm p}$-${\rm ord}(M|_{G_p})=1$ for any prime number $p$;\\
{\rm (iii)} ${\rm p}$-${\rm ord}(M|_{G_p})=1$ for any prime divisor $p\mid |G|$;\\
{\rm (iv)} $(E)$ with $N$ permutation splits;\\
{\rm (v)} $M$ is invertible. 

In particular, $M$ is invertible if and only if 
$M|_{G_p}$ is invertible for any prime  number $p$ if and only if 
$M|_{G_p}$ is invertible for any prime divisor $p\mid |G|$ 
$($due to Endo and Miyata \cite[Lemma 1.4]{EM75}$)$. 
\end{proposition}
\begin{proof}
Define 
$d:={\rm p}$-${\rm ord}(M)$ and 
$e:={\rm ord}(E)$, i.e. 
the order of $(E)$ as an element 
$(E) \in \Ext^1_{\bZ[G]}(M,C)\simeq H^1(G,{\rm Hom}_{\bZ}(M,C))$.\\ 
(1) We see $e\mid |G|$ because 
$\Ext^1_{\bZ[G]} (M,C)\simeq H^1(G,{\rm Hom}_{\bZ}(M,C))$  
is killed by $|G|$ 
(see Neukirch, Schmidt and Wingberg \cite[Proposition 1.6.1]{NSW00}).\\ 
(2) We will show $e\mid d$. 
It also follows from the definition of $d$ that 
there exist a permutation $G$-lattice $Q$, 
$\varphi\in {\rm Hom}_{\bZ[G]}(M,Q)$, 
$\psi\in {\rm Hom}_{\bZ[G]}(Q,M)$ 
such that 
$d\,\Id_M=\psi\circ\varphi: M \xrightarrow{\varphi} Q \xrightarrow{\psi} M$. 
Form the pull-back diagram 
(see e.g. Hilton and Stammbach \cite[II.6, page 59]{HS71}): 
\begin{align*}
\xymatrix@C=45pt@R=20pt{
(\psi\circ\varphi)^\ast(E): 
0 \ar[r] & C \ar[r]^-{(\psi\circ\varphi)^\ast(\iota)} \ar@{=}[d] & (\psi\circ\varphi)^\ast(N) \ar[r]^-{(\psi\circ\varphi)^\ast(\pi)} \ar[d]^{\varphi_N} & M\ar[r] \ar[d]^\varphi & 0 \\
\psi^\ast(E): 
0 \ar[r] & C \ar[r]^-{\psi^\ast(\iota)} \ar@{=}[d] & \psi^\ast(N) \ar[r]^-{\psi^\ast(\pi)} \ar[d]^{\psi_N} & Q\ar[r] \ar[d]^\psi & 0 \\
(E):  0 \ar[r] & C \ar[r]^-\iota & N \ar[r]^-\pi & M \ar[r]  & 0
}
\end{align*}
where 
\begin{align*}
\psi^\ast(N)&=\{(n,q)\in N\times Q\mid \pi(n)=\psi(q) \},\\
\psi_N(n,q)&=n,\ 
\psi^\ast(\pi)(n,q)=q,\ 
\psi^\ast(\iota)(c)=\big(\iota(c),0\big),\\
(\psi\circ\varphi)^\ast(N)&=\{(n,q,m)\in N\times Q\times M \mid 
\pi(n)=\psi(q),\ q=\varphi(m) \},\\
\varphi_N(n,q,m)&=(n,q),\ 
(\psi\circ\varphi)^\ast(\pi)(n,q,m)=m,\ 
(\psi\circ\varphi)^\ast(\iota)(c)=\big(\iota(c),0,0\big).
\end{align*}
Because $C$ is coflabby and $Q$ is permutation, by Lemma \ref{lem6.3}, 
$\psi^\ast(E)$ splits, i.e. $\Ext^1_{\bZ[G]}(Q,C)=0$. 
Hence there exists $t\in {\rm Hom}_{\bZ[G]}(Q,\psi^\ast(N))$ 
such that $\psi^\ast(\pi)\circ t=\Id_Q$. 
We can take 
$\varphi^\ast(t)\in {\rm Hom}_{\bZ[G]}(M,(\psi\circ\varphi)^\ast(N))$ 
as $\varphi^\ast(t)(m)=(\psi_N\circ t\circ\varphi(m),\varphi(m),m)$. 
Then we have $(\psi\circ\varphi)^\ast(\pi)\circ\varphi^\ast(t)=\Id_M$. 
Now we can take $\psi_\ast(t)\in {\rm Hom}_{\bZ[G]}(M,N)$ 
defined by $\psi_\ast(t)(m)=\psi_N\circ t\circ \varphi(m)$ 
which satisfies $\pi\circ\psi_\ast(t)=\pi\circ(\psi_N\circ t\circ \varphi)
=(\psi\circ\psi^\ast(\pi))\circ(t\circ \varphi)=(\psi\circ \varphi)=d\,\Id_M$. 
This implies that $(\psi\circ\varphi)^\ast(E)=d(E)$. 
Because $d(E)$ splits, we conclude that $e\mid d$.\\
(3) By (2), we should prove $d\leq e$. 
If $N=P$ is permutation, then it follows from Lemma \ref{lem5.4} that 
there exists $s\in {\rm Hom}_{\bZ[G]}(M,P)$ 
such that $\pi \circ s= e\,\Id_M$. 
This implies $d \leq e$ from the definition of $d$.
The last statement follows from Lemma \ref{lem5.3} (3). 
\\
(4) it follows from the existence of a coflabby resolution 
(see Endo and Miyata \cite[Lemma 1.2]{EM75}), (1) and (3) that 
{${\rm p}$-${\rm ord}(M)$} $\mid$ $|G|$.\\
(5) By (3), we have 
 ${\rm p}$-${\rm ord}(M)={\rm ord}(E)$ and 
 ${\rm p}$-${\rm ord}(M|_{G_p})={\rm ord}(E|_{G_p})$. 
Hence the assertion follows from Lemma \ref{lemRes}.\\
(6) By (5), we have ${\rm (i)}\Leftrightarrow {\rm (ii)}
\Leftrightarrow {\rm (iii)}$. 
By (3), we have ${\rm (i)}\Leftrightarrow {\rm (iv)}$.\\
${\rm (iv)}\Rightarrow {\rm (v)}$: 
By (3), if 
${\rm ord}(E)=1$, i.e. $(E)$ splits, then $C\oplus M\simeq N$ permutation 
and hence $M$ is invertible.\\
${\rm (v)}\Rightarrow {\rm (iv)}$: 
If $M$ is invertible, then it follows from 
Lemma \ref{lem6.3} that $(E)$ splits. 

The last assertion follows from (6). 
\end{proof}

Let $p$ be a prime number. 
Let $\bZ_{(p)}$ be the localization of $\bZ$ at a prime ideal $(p)$ and 
$\bZ_p$ be the completion of $\bZ$ at a prime $p$. 
It is known that 
Krull-Schmidt-Azumaya theorem holds for $\bZ_p[G]$-lattices, 
i.e. $\bZ_p[G]$-lattices have the unique decomposition into 
indecomposable ones up to isomorphism and numbering 
(see Curtis and Reiner \cite[Theorem 6.12]{CR81}, 
see also Azumaya \cite[Theorem 1]{Azu50}, 
Hoshi and Yamasaki \cite[Chapter 4]{HY17}, \cite[Section 7]{HY1}) 
although it does not hold for $G$-lattices in general 
(see Swan \cite[Section 10]{Swa60}, Curtis and Reiner \cite[Section 36]{CR81}, 
Hoshi and Yamasaki \cite[page 9 and Chapter 4]{HY17}). 
If $G$ is a $p$-group with $p\neq 2$, then 
Krull-Schmidt-Azumaya theorem holds for $\bZ_{(p)}[G]$-lattices 
(see Jones \cite[{Theorem 2 with $q=1$}]{Jon65}, 
Curtis and Reiner \cite[Theorem 36.1]{CR81}). 

\begin{proposition}\label{prop6.6}
Let $G$ be a finite group, $M$ be a $G$-lattice and 
$(E): 0 \to C \xrightarrow{\iota} P \xrightarrow{\pi} M \to 0$ 
be a coflabby resolution of $M$ as $G$-lattices 
with $P$ permutation and $C$ coflabby. 
Let $p$ be a prime number and 
$G_p={\rm Syl}_p(G)$ be a $p$-Sylow subgroup of $G$. 
Let $M|_{G_p}$ be a $G_p$-lattice 
obtained by restricting the action of $G$ on $M$ to $G_p$. 
Let 
\begin{align*}
(E|_{G_p}): 0 \to C|_{G_p} \xrightarrow{\iota|_{G_p}} P|_{G_p} \xrightarrow{\pi|_{G_p}} M|_{G_p} \to 0
\end{align*}
be a coflabby resolution of $M|_{G_p}$ as $G_p$-lattices, 
\begin{align*}
(E_{(p)}): 0 \to C_{(p)}:=C\otimes \bZ_{(p)} \xrightarrow{\iota_{(p)}} P_{(p)}:=P\otimes \bZ_{(p)}\xrightarrow{\pi_{(p)}} M_{(p)}:=M\otimes \bZ_{(p)}\to 0
\end{align*}
be a coflabby resolution of $M_{(p)}=M\otimes \bZ_{(p)}$ 
as $\bZ_{(p)}[G]$-lattices 
and 
\begin{align*}
(E_p): 0 \to C_p:=C\otimes \bZ_p \xrightarrow{\iota_p} P_p:=P\otimes \bZ_p 
\xrightarrow{\pi_p} M_p:=M\otimes \bZ_p \to 0
\end{align*}
be a coflabby resolution of $M_p=M\otimes \bZ_p$ as $\bZ_p[G]$-lattices. 
Then 
${\rm p}$-${\rm ord}(M|_{G_p})={\rm p}$-${\rm ord}(M_{(p)})
={\rm p}$-${\rm ord}(M_p)={\rm ord}(E|_{G_p})
={\rm ord}(E_{(p)})={\rm ord}(E_p)$ 
which is the $p$-part of ${\rm p}$-${\rm ord}(M)={\rm ord}(E)$. 
\end{proposition}
\begin{proof}
For ${\rm p}$-${\rm ord}(M|_{G_p})={\rm p}$-${\rm ord}(M_{(p)})
={\rm ord}(E|_{G_p})={\rm ord}(E_{(p)})$, 
by a similar argument of the proof of Proposition \ref{prop6.5} (3), 
we have 
${\rm p}$-${\rm ord}(M|_{G_p})={\rm ord}(E|_{G_p})$ and 
${\rm p}$-${\rm ord}(M_{(p)})={\rm ord}(E_{(p)})$. 
Because ${\rm ord}(E_{(p)})$ is the order as an element 
$(E_{(p)}) \in \Ext^1_{\bZ_{(p)}[G]}(M|_{(p)},C|_{(p)})\simeq 
H^1(G,{\rm Hom}_{\bZ_{(p)}}(M_{(p)},C_{(p)}))$ and 
${\rm ord}(E|_{G_p})$ is the order as an element 
$(E|_{G_p}) \in \Ext^1_{\bZ[{G_p}]}(M|_{G_p},C|_{G_p})\simeq 
H^1(G_p,{\rm Hom}_{\bZ}(M|_{G_p},C|_{G_p}))$, 
${\rm ord}(E_{(p)})={\rm ord}(E|_{G_p})$ follows from 
$(E_{(p)})=(E|_{G_p})\in H^1(G,{\rm Hom}_{\bZ_{(p)}}(M_{(p)},C_{(p)}))\simeq 
H^1(G,{\rm Hom}_{\bZ_{(p)}}(M_{(p)},C_{(p)}))|_{G_p}\hookrightarrow
H^1(G_p,{\rm Hom}_{\bZ_{(p)}}(M_{(p)}|_{G_p},C_{(p)}|_{G_p}))\simeq 
H^1(G_p,{\rm Hom}_{\bZ}(M|_{G_p},C|_{G_p}))$ 
which (the last one) is obtained as the $p$-part of the Smith normal form of 
\begin{align*}
\left[
\begin{array}{cccc}
[a_{i,j}(g_1)]-I_n\ \mid\ \cdots\ \mid\ [a_{i,j}(g_r)]-I_n
\end{array}
\right]
\end{align*}
where $G_p=\langle g_1,\ldots,g_r\rangle$, 
$u_i^\sigma=\sum_{j=1}^n a_{i,j}(\sigma) u_j$ $(1\leq i\leq n, \sigma\in G_p)$ 
with $\bZ$-basis $\{u_1,\ldots,u_n\}$ of $M$ 
and $I_n$ is the $n\times n$ identity matrix 
via invariant factor theory (see Hoshi and Yamasaki 
\cite[Section 5.0, Algorithm F0]{HY17}). 

For ${\rm p}$-${\rm ord}(M|_{G_p})={\rm p}$-${\rm ord}(M_p)
={\rm ord}(E|_{G_p})={\rm ord}(E_p)$,  
we get the assertion by the similar manner. 
\end{proof}

\begin{definition}[{Scavia \cite[Definition 2.1]{Sca20}}]
Let $G$ be a finite group and $M$ be a $G$-lattice. 
Let $p$ be a prime number and 
$M_{(p)}=M\otimes \bZ_{(p)}$ be a $\bZ_{(p)}[G]$-lattice. 
A $G$-lattice $M$ is called $p$-invertible if there exists permutation 
$G$-lattice $P$ such that $M_{(p)}$ is a direct summand of 
$P_{(p)}=P\otimes \bZ_{(p)}$, i.e. 
there exists $\bZ_{(p)}[G]$-lattice $M^\prime$ such that 
$M_{(p)}\oplus M^\prime\simeq P_{(p)}$. 
\end{definition}
Scavia \cite[Lemma 2.5]{Sca20} showed that 
$M$ is $p$-invertible if and only if $M|_{G_p}$ is invertible. 
The concept of $p$-invertiblity of a $G$-lattice $M$ can be well-understood 
via the permutation order {${\rm p}$-${\rm ord}(M)$} of $M$ 
(note that $M$ is invertible if and only if ${\rm p}$-${\rm ord}(M)=1$ 
by Proposition \ref{prop6.5} (6)):
\begin{proposition}[{see also Scavia \cite[Definition 2.1]{Sca20}}]
Let $G$ be a finite group and $M$ be a $G$-lattice. 
Let $p$ be a prime number and 
$G_p={\rm Syl}_p(G)$ be a $p$-Sylow subgroup of $G$. 
Let $M|_{G_p}$ be a $G_p$-lattice 
obtained by restricting the action of $G$ on $M$ to $G_p$. 
The following conditions are equivalent:\\
{\rm (i)} $M$ is $p$-invertible;\\
{\rm (ii)} $\gcd(p,{\rm p}$-${\rm ord}(M))=1$;\\
{\rm (iii)} $p$ ${\not |}$ ${\rm p}$-${\rm ord}(M)$;\\
{\rm (iv)} ${\rm p}$-${\rm ord}(M|_{G_p})=1$. 
\end{proposition}
\begin{proof}
The assertion follows from 
Proposition \ref{prop6.5} (3), (6) and Proposition \ref{prop6.6}.
\end{proof}
\begin{definition}[{Merkurjev \cite[Section 2]{Mer20}}]
Let $k$ be a field. 
An integral $k$-variety $X$ is called {\it $p$-retract $k$-rational} if 
there exists a rational dominant morphism 
$f:\mathbb{P}^n \dashrightarrow X$ 
for some $n$ such that for every nonempty open subset 
$U \subset \mathbb{P}^n$ in the domain of definition of $f$, 
there is a morphism 
$g:Y \to \mathbb{P}^n$ such that $\operatorname{Im}(g) \subset U$ and 
the composition $f\circ g:Y \to X$ is dominant of finite degree prime to $p$: 
\begin{align*}
\xymatrix{
{ } 
&{Y} \ar[dl]_g \ar[d]^{\mathrm{of \ degree \ prime \ to}\ p } \\
{\mathbb{P}^n} \ar@{-->}[r]^f
&{X}.
}
\end{align*}
\end{definition}
The following result is due to Scavia \cite{Sca20} 
(see also Proposition \ref{prop6.5} (5), (6)): 
\begin{theorem}[{Scavia \cite[Proposition 3.1, Theorem 1.1]{Sca20}}]
Let $T$ be an algebraic $k$-torus with the minimal splitting field $L$ 
and the $G$-lattice $\widehat{T}\simeq {\rm Hom}(T\times_k L, \bG_{m,L})$ 
where $G={\rm Gal}(L/k)$. 
Then 
$T$ is $p$-retract $k$-rational if and only if 
$[\widehat{T}]^{fl}$ is $p$-invertible. 
In particular, 
$T$ is retract $k$-rational if and only if $T$ is $p$-retract $k$-rational 
for any prime number $p$ if and only if $T$ is $p$-retract $k$-rational 
for any prime divisor $p\mid |G|$. 
\end{theorem}

\subsection{Tensor products $T_1\otimes T_2$ of algebraic $k$-tori $T_1$ and $T_2$}
\begin{lemma}\label{lem6.5}
Let $P$ be a permutation $G$-lattice.\\
{\rm (1)} If $P^\prime$ is a permutation $G$-lattice, 
then $P\otimes P^\prime$ is a permutation $G$-lattice;\\
{\rm (2)} If $S$ is a stably permutation $G$-lattice, 
then $S\otimes P$ is a stably permutation $G$-lattice;\\
{\rm (3)} If $S$, $S^\prime$ are stably permutation $G$-lattices, 
then $S\otimes S^\prime$ is a stably permutation $G$-lattice;\\
{\rm (4)} If $I$ is an invertible $G$-lattice, 
then $I\otimes P$ is an invertible $G$-lattice;\\
{\rm (5)} If $I$, $J$ are invertible $G$-lattices, 
then $I\otimes J$ is an invertible $G$-lattice;\\
{\rm (6)} If $C$ is a coflabby $G$-lattice, 
then $C\otimes P$ is a coflabby $G$-lattice;\\
{\rm (7)} If $F$ is a flabby $G$-lattice, 
then $F\otimes P$ is a flabby $G$-lattice. 
\end{lemma}
\begin{proof}
(1) If $P\simeq\bZ[X]$, $P^\prime\simeq\bZ[Y]$ are permutation, 
then there exists an isomorphism 
$\varphi: \bZ[X] \otimes \bZ[Y] \xrightarrow{\sim} \bZ[X\times Y]$, 
$(\sum_{x\in X}a_x[x])\otimes(\sum_{y\in Y}b_y[y])\mapsto 
\sum_{x\in X}\sum_{y\in Y}a_xb_y[(x,y)]$. 
Hence $P\otimes P^\prime\simeq\bZ[X\times Y]$ is permutation.\\
(2) If $S$ is stably permutation, then 
there exists a permutation $G$-lattice $Q$ 
such that $S\oplus Q$ is permutation. 
By (1), $(S\oplus Q)\otimes P\simeq (S\otimes P)\oplus(Q\otimes P)$ 
and $Q\otimes P$ are permutation. 
Hence $S\otimes P$ is stably permutation.\\
(3) If $S$, $S^\prime$ are stably permutation, then 
there exist permutation $G$-lattices $Q$, $Q^\prime$ 
such that $S\oplus Q$, $S^\prime\oplus Q^\prime$ are permutation. 
By (1), $(S\oplus Q)\otimes P$, $(S^\prime\oplus Q^\prime)\otimes P$ 
are permutation. 
We have 
$(S\oplus Q)\otimes (S^\prime\oplus Q^\prime)\simeq 
(S\otimes S^\prime)\oplus(S\otimes Q^\prime)\oplus
(Q\otimes S^\prime)\oplus(Q\otimes Q^\prime)$. 
By (1), (2), we find that 
$(S\oplus Q)\otimes (S^\prime\oplus Q^\prime)$, 
$S\otimes Q^\prime$, $Q\otimes S^\prime$, $Q\otimes Q^\prime$ 
are permutation. 
Hence $S\otimes S^\prime$ is stably permutation.\\
(4) If $I$ is invertible, then there exists $G$-lattice $I^\prime$ 
such that $I\oplus I^\prime$ is permutation. 
By (1), $(I\oplus I^\prime)\otimes P$ is permutation. 
Hence $I\otimes P$ is invertible because 
$(I\oplus I^\prime)\otimes P\simeq (I\otimes P)\oplus(I^\prime\otimes P)$.\\ 
(5) 
If $I$ is invertible, then there exists $G$-lattice $I^\prime$ 
such that $I\oplus I^\prime$ is permutation. 
By (4), $(I\oplus I^\prime)\otimes J$ is invertible. 
Hence $I\otimes J$ is invertible because 
$(I\oplus I^\prime)\otimes J\simeq (I\otimes J)\oplus(I^\prime\otimes J)$. 
Alternatively, 
if $I$, $J$ are invertible, then there exist $G$-lattice $I^\prime$, 
$J^\prime$ such that $I\oplus I^\prime\simeq P$, 
$J\oplus J^\prime\simeq Q$ are permutation. 
Then $P\otimes Q\simeq (I\otimes J)\oplus (I\otimes J^\prime)\oplus 
(I^\prime\otimes J)\oplus (I^\prime\otimes J^\prime)$. 
By (1), $P\otimes Q$ is permutation. 
Hence $I\otimes J$ is invertible.\\
(6) 
Let $H\leq G$ be a subgroup. 
We need to show that $H^1(H,C\otimes P)=0$. 
Restricted to an $H$-lattice, 
$P=\oplus_{i=1}^r\bZ[H/H_i]$ for subgroups $H_i\leq H$. 
By Shapiro's lemma (see Brown \cite[Proposition 6.2, page 73]{Bro82}, 
Neukirch, Schmidt and Wingberg \cite[Proposition 1.6.3]{NSW00}), 
we have $H^1(H,C\otimes P)=H^1(H,C\otimes (\oplus_{i=1}^r\bZ[H/H_i]))
=H^1(H,\oplus_{i=1}^r(C\otimes\bZ[H/H_i]))
=\oplus_{i=1}^r H^1(H,C\otimes\bZ[H/H_i])
=\oplus_{i=1}^r H^1(H,C[H/H_i])=\oplus_{i=1}^rH^1(H_i,C)=0$ 
because $C$ is coflabby.\\
(7) The proof is the same as for (6), 
using $H^{-1}$ instead of $H^1$.
\end{proof}
%
\begin{theorem}\label{th6.7}
Let $M_1$ and $M_2$ be $G$-lattices with 
$d_1={\rm p}$-${\rm ord}(M_1)$ and $d_2={\rm p}$-${\rm ord}(M_2)$. 
We take flabby resolutions of $M_1$ and $M_2$ respectively:  
\begin{align*}
&(E_1): 0 \to M_1\xrightarrow{\iota_1}  P_1 \xrightarrow{\pi_1} F_1 \to 0,\\
&(E_2): 0 \to M_2\xrightarrow{\iota_2}  P_2 \xrightarrow{\pi_2} F_2 \to 0
\end{align*}
with $F_i$ flabby and $P_i$ permutation $(i=1,2)$. 
Then there exists an exact sequence of $G$-lattices 
\begin{align*}
(E_{1,2}^{(1)}): 0 \to M_1 \otimes M_2 \xrightarrow{\iota} P_1 \otimes P_2 \xrightarrow{f} (F_1 \otimes P_2) \oplus (P_1 \otimes F_2) \xrightarrow{\pi} F_1 \otimes F_2 \to 0
\end{align*}
where 
\begin{align*}
\iota(m_1\otimes m_2)&=\iota_1(m_1)\otimes \iota_2(m_2),\\
f(p_1\otimes p_2)&=(\pi_1(p_1)\otimes p_2,p_1\otimes\pi_2(p_2)),\\
\pi(f_1\otimes p_2,p_1\otimes f_2)&=f_1\otimes\pi_2(p_2)-\pi_1(p_1)\otimes f_2. 
\end{align*}
Define
\[
F:={\rm Image}(f)={\rm Ker}(\pi)\leq (F_1 \otimes P_2) \oplus (P_1 \otimes F_2).
\]
{\rm (1)} 
If $\gcd(d_1,d_2)=1$, then 
the exact sequence 
\begin{align*}
(E_{1,2}^{(1R)}): 0\to F\xrightarrow{f|_F} (F_1 \otimes P_2) \oplus (P_1 \otimes F_2) \xrightarrow{\pi} F_1 \otimes F_2 \to 0
\end{align*}
splits, i.e. $F\oplus(F_1 \otimes F_2)\simeq (F_1 \otimes P_2) \oplus (P_1 \otimes F_2)$. 
Hence the $G$-lattices $F$ and $F_1 \otimes F_2$ are flabby 
and the exact sequence 
\begin{align*}
(E_{1,2}^{(1L)}): 0 \to M_1 \otimes M_2 \xrightarrow{\iota} 
P_1\otimes P_2  \xrightarrow{f} F \to 0
\end{align*} 
gives a flabby resolution of $M_1 \otimes M_2$ 
with $F$ flabby and $P_1\otimes P_2$ permutation, i.e. 
$[M_1\otimes M_2]^{fl}=F$. 
In particular, ${\rm p}$-${\rm ord}(M_1\otimes M_2)
={\rm ord}(E_{1,2}^{(1L)})\mid d_1d_2$.\\
{\rm (2)} If $\gcd(d_1,d_2)=1$ and 
$F_1$ and $F_2$ are stably permutation, 
then $F$ is stably permutation, i.e. $[F]=0$.\\
{\rm (3)} If $\gcd(d_1,d_2)=1$ and 
$F_1$ and $F_2$ are invertible, 
then $F$ is invertible.
\end{theorem}
\begin{proof}
(1) It follows from Lemma \ref{lem6.2} an Proposition \ref{prop6.5} that 
$d_1={\rm ord}(E_1)$ and $d_2={\rm ord}(E_2)$. 
Then the assertion follows from Proposition \ref{prop5.6} 
with the help of Lemma \ref{lem6.5} because for $G$-lattices $M$, $N$, 
we have $\widehat{H}^{-1}(H,M\oplus N)\simeq \widehat{H}^{-1}(H,M)\oplus \widehat{H}^{-1}(H,N)$ for any $H\leq G$. 
In particular, by (the dual version of) 
Proposition \ref{prop6.5} (3), we see that 
{${\rm p}$-${\rm ord}$}$(M_1\otimes M_2)
={\rm ord}(E_{1,2}^{(1L)})\mid d_1d_2$.\\ 
(2) By Lemma \ref{lem6.5} (3), 
$F_1 \otimes F_2$ is stably permutation. 
By (1), we have  
$F\oplus(F_1 \otimes F_2)\simeq (F_1 \otimes P_2) \oplus (P_1 \otimes F_2)$. 
Use Lemma \ref{lem6.5} again. 
Then we find that $F$ is stably permutation.\\
(3) By (1) and Lemma \ref{lem6.5}, we see 
$F\oplus(F_1 \otimes F_2)\simeq (F_1 \otimes P_2) \oplus (P_1 \otimes F_2)$ 
is invertible.
\end{proof}
\begin{remark}
(1) 
In general, for $G$-lattice $M_i$ $(i=1,2)$ with 
$d_i={\rm p}$-${\rm ord}(M_i)$, we see that 
{${\rm p}$-${\rm ord}$}$(M_1\otimes M_2)\mid d_1d_2$ 
without the assumption $\gcd(d_1,d_2)=1$ because 
if $a_i\,{\rm id}_{M_i}: M_i\xrightarrow{f_i} P_i\xrightarrow{g_i} M_i$ 
$(i=1,2)$, then 
$a_1a_2{\rm id}_{M_1\otimes M_2}
=(a_1{\rm id}_{M_1})\otimes (a_2{\rm id}_{M_2}): 
M_1\otimes M_2\xrightarrow{f_1\otimes f_2} P_1\otimes P_2\xrightarrow{g_1\otimes g_2} M_1\otimes M_2$ and 
$(g_1\otimes g_2)\circ (f_1\otimes f_2)=(g_1\circ f_1)\otimes (g_2\circ f_2)$.\\
(2) 
We will see that 
there exist stably $k$-rational algebraic $k$-tori $T_1$, $T_2$ 
such that $T_1\otimes T_2$ is not retract $k$-rational when 
$\gcd(d_1,d_2)\neq 1$ (see Example \ref{ex7.7}). 
This gives an example which 
$(E_{1,2}^{(1R)})$ does not split and 
$[F]$ is not the flabby class of $M_1\otimes M_2$, 
i.e. the assumption $\gcd(d_1,d_2)=1$ 
of Theorem \ref{th6.7} is necessary. 
\end{remark}

We also get the dual version of Theorem \ref{th6.7} by the same manner. 
\begin{theorem}\label{th6.6}
Let $M_1$ and $M_2$ be $G$-lattices with 
$d_1={\rm p}$-${\rm ord}(M_1)$ and $d_2={\rm p}$-${\rm ord}(M_2)$. 
We take coflabby resolutions of $M_1$ and $M_2$:  
\begin{align*}
(E_1): 0 \to C_1\xrightarrow{\iota_1}  P_1 \xrightarrow{\pi_1} M_1 \to 0,\\
(E_2): 0 \to C_2\xrightarrow{\iota_2}  P_2 \xrightarrow{\pi_2} M_2 \to 0
\end{align*}
with $C_i$ coflabby and $P_i$ permutation $(i=1,2)$. 
Then there exists an 
exact sequence of $G$-lattices 
\begin{align*}
(E_{1,2}^{(2)}): 0 \to C_1 \otimes C_2 \xrightarrow{\iota} (C_1 \otimes P_2) \oplus (P_1 \otimes C_2) \xrightarrow{f} P_1 \otimes P_2 \xrightarrow{\pi} M_1 \otimes M_2 \to 0
\end{align*}
where 
\begin{align*}
\iota(c_1 \otimes c_2) &= (c_1 \otimes \iota_2(c_2), -\iota_1(c_1) \otimes c_2),\\
f (c_1 \otimes p_2, p_1 \otimes c_2) &= \iota_1(c_1) \otimes p_2 + p_1 \otimes \iota_2(c_2),\\
\pi(p_1,p_2) & =\pi_1(p_1) \otimes \pi_2(p_2).
\end{align*}
Define
\[
C:={\rm Image}(f)={\rm Ker}(\pi)\leq P_1 \otimes P_2.
\]
{\rm (1)} 
If $\gcd(d_1,d_2)=1$, then 
the exact sequence 
\begin{align*}
(E_{1,2}^{(2L)}): 0 \to C_1 \otimes C_2 \xrightarrow{\iota} 
(C_1 \otimes P_2) \oplus (P_1 \otimes C_2)  \xrightarrow{f} C \to 0 
\end{align*}
splits, i.e. $(C_1 \otimes C_2)\oplus C\simeq (C_1 \otimes P_2) \oplus (P_1 \otimes C_2)$. 
Hence the $G$-lattices $C_1 \otimes C_2$ and $C$ are coflabby 
and the exact sequence 
\begin{align*}
(E_{1,2}^{(2R)}): 0 \to C \xrightarrow{f|_C} P_1\otimes P_2 \xrightarrow{\pi} M_1 \otimes M_2 \to 0
\end{align*} 
gives a coflabby resolution of $M_1 \otimes M_2$ 
with $C$ coflabby and $P_1\otimes P_2$ permutation. 
In particular, ${\rm p}$-${\rm ord}(M_1\otimes M_2)
={\rm ord}(E_{1,2}^{(2R)})\mid d_1d_2$.\\
{\rm (2)} If $\gcd(d_1,d_2)=1$ and 
$C_1$ and $C_2$ are stably permutation, 
then $C$ is stably permutation, i.e. $[C]=0$.\\ 
{\rm (3)} If $\gcd(d_1,d_2)=1$ and 
$C_1$ and $C_2$ are invertible, 
then $C$ is invertible. 
\end{theorem}

By Theorem \ref{th6.7}, we get one of the main  contribution  of this paper 
(see Section \ref{S3} for examples of stably (resp. retract) $k$-rational algebraic $k$-tori $T$): 
\begin{theorem}\label{th6.15}
Let $T_1$ and $T_2$ be algebraic $k$-tori with 
permutation order 
$d_1=\rm{p}$-${\rm ord}(\widehat{T}_1)$, 
$d_2=\rm{p}$-${\rm ord}(\widehat{T}_2)$. 
If ${\rm gcd}(d_1,d_2)=1$ and 
$T_1$ and $T_2$ are stably $($resp. retract$)$ $k$-rational, 
then $T_1 \otimes T_2$ is stably $($resp. retract$)$ $k$-rational. 
\end{theorem}
\begin{proof}
Define $M_i:=M_{G_i}=\widehat{T}_i$ $(i=1,2)$ 
where $G_i\leq {\rm GL}(n_i,\bZ)$ as in Definition \ref{defMG}, 
and introduce a flabby resolution of $M_i$ with $[M_i]^{fl}=F_i$ 
as in  Theorem \ref{th6.7}. 
Applying Theorem \ref{th6.7} yields a flabby resolution 
\begin{align*}
(E_{1,2}^{(1L)}): 0 \to M_1 \otimes M_2 \xrightarrow{\iota} 
P_1\otimes P_2  \xrightarrow{f} F \to 0.
\end{align*}
Again, Theorem \ref{th6.7} (2) (resp. (3)) shows that, 
if $F_1$ and $F_2$ are stably permutation (resp. invertible), then so is $F$. 
By Theorem \ref{th2.7}, the stable (resp. retract) $k$-rationality 
of $T_1$ and $T_2$, indeed implies that of $T_1 \otimes T_2$.
\end{proof} 

There is a straightforward partial converse to Theorem \ref{th6.15} 
(without the assumption  ${\rm gcd}(d_1,d_2)=1$): 

\begin{proposition}\label{convth6.15}
Let $G_i\leq \GL(n_i,\bZ)$ $(i=1,2)$ be a finite subgroup 
and 
$M_{G_i}$ be the $G_i$-lattice as in Definition \ref{defMG}. 
Define $G:= G_1 \times G_2$. 
Let $h: \Gal(\overline k/k) \to G$ be a continuous surjective homomorphism 
and $T_i$ be the algebraic $k$-torus 
with $\widehat{T}_i=M_{G_i}$ $($via $h$$)$. 
If 
$T_1\otimes T_2$ is retract $k$-rational, 
then $T_1$ and $T_2$ are retract $k$-rational. 
\end{proposition}
\begin{proof}
Consider a flabby resolution of $G$-lattices
\begin{align*}
 0 \to M_{G_1}\otimes M_{G_2} \to 
P \to F \to 0.
\end{align*}

Assume that $T$ is retract $k$-rational, 
i.e. (by  Theorem \ref{th2.7}) that $[F]$ is an invertible $G$-lattice. 
After restricting the action of $G$ to $G_1$ $(=G_1 \times \{ \Id\})$, 
we get the exact sequence of $G_1$-lattices 
\begin{align*}
0 \to (M_{G_1})^{\oplus n_2} \to P|_{G_1} \to F|_{G_1} \to 0 
\end{align*} 
where $P|_{G_1}$ is permutation and 
$F|_{G_1}$ is invertible (as $G_1$-lattices). 
Thus, $[(M_{G_1})^{\oplus n_2}]^{fl}=n_2[M_{G_1}]^{fl}$ is invertible. Equivalently, $[M_{G_1}]^{fl}$ is invertible. 
Thus, by Theorem \ref{th2.7}, 
the algebraic $k$-torus $T_1$ is retract $k$-rational.
\end{proof} 
%

\section{Rationality of norm one tori $\T_{A\otimes B}=R^{(1)}_{A\otimes B/k}(\bG_m)$ of tensor products of \'etale algebras $A/k$ and $B/k$}\label{S7}

Let $A/k$ be an \'etale $k$-algebra. 
Write  $A=\prod_{i=1}^rK_i$ 
where $K_i/k$ $(1\leq i\leq r)$ 
are finite separable field extensions with $[K_i:k]=n_i$. 
Let $L \subset \overline k$ be the splitting field of $A/k$. 
It is the Galois closure of the compositum of the $K_i$'s, 
inside $\overline k$ (see Section \ref{ss23} and Section \ref{ss272}). 
Then the finite $G$-set $X:=X(A)=\Hom_k(A,\overline{k})$ 
is a disjoint union of $G$-orbits $X_i$, 
with $|X_i|=n_i$ and $n=\sum_{i=1}^r n_i$. 

Let $M$ be a $G$-lattice. 
Recall that the permutation order ${\rm p}$-${\rm ord}(M)$ of $M$ 
is the smallest integer $a \geq 1$ 
such that 
$a\,\Id_M: M\to M$ 
factors through a permutation $G$-lattice 
as a $\bZ[G]$-homomorphism (see Definition \ref{DefiPermutationOrder}). 

\begin{theorem}\label{th8.3}
Let 
$\T_A=R^{(1)}_{A/k}(\bG_m)$ be the norm one torus 
of an \'etale algebra $A/k$ 
with ${\rm dim}_k\, A=n$ where $A=\prod_{i=1}^rK_i$ with $[K_i:k]=n_i$. 
Let $X=X(A)=\Hom_k(A,\overline{k})$ be the corresponding 
$G$-set with $\widehat{\T}_A\simeq J_X$ and $|X|=n$ where 
$X=\coprod_{i=1}^r X_i$ is the disjoint union of $G$-orbits 
with $|X_i|=n_i$ and $n=\sum_{i=1}^r n_i$. 
Let 
\begin{align*}
(E_X): 0\to I_X\xrightarrow{\iota}\bZ[X]\xrightarrow{\varepsilon}\bZ\to 0
\end{align*} 
be an exact sequence of $G$-lattices and its dual 
\begin{align*}
(F_X): 0 \to\bZ\xrightarrow{\varepsilon^\circ}
\bZ[X]\xrightarrow{\iota^\circ} J_X\to 0
\end{align*}
given as in Definition \ref{def8.1}. 
Then\\ 
{\rm (1)} 
${\rm p}$-${\rm ord}(I_X)={\rm p}$-${\rm ord}(J_X)={\rm ord}(E_X)={\rm ord}(F_X)$. 
In particular, ${\rm p}$-${\rm ord}(I_X)\mid |G|$.\\ 
{\rm (2)} If $r=1$, then ${\rm p}$-${\rm ord}(J_X)=n$. 
In particular, we have ${\rm p}$-${\rm ord}(J_{G/H})=n$ with 
$J_X\simeq J_{G/H}$ when $G\leq S_n$ transitive and $[G:H]=n$.\\
{\rm (3)} If $r\geq 2$, then ${\rm p}$-${\rm ord}(J_X)=\gcd(n_1,\ldots,n_r) \mid n$.\\
{\rm (4)} If $\gcd(n_1,\ldots,n_r)=1$, e.g. $n=p$; prime, 
then $\bZ\oplus J_X\simeq \bZ[X]$ and hence 
$\T_A\times \bG_{m,k}$ is $k$-rational. 
In particular, 
$\T_A$ 
is stably $k$-rational. 
\end{theorem}
\begin{proof}
Define 
$d:={\rm p}$-${\rm ord}(J_X)$ and 
$e:={\rm ord}(F_X)$, i.e. 
the order of $(F_X)$ as an element 
$(F_X) \in \Ext^1_{\bZ[G]}(J_X,\bZ)\simeq H^1(G,{\rm Hom}_{\bZ}(J_X,\bZ))$.\\ 
(1) It follows from Proposition \ref{prop6.5} that $d=e$. 
Hence we get the assertion by Lemma \ref{lem6.2} 
and ${\rm ord}(E_X)={\rm ord}(F_X)$ because $(F_X)$ is the dual of $(E_X)$. 
The last assertion follows from Proposition \ref{prop6.5}.\\
(2), (3) 
Define 
\begin{align*}
g:=\begin{cases}
n & {\rm if}\ \ r=1\\
\gcd(n_1,\ldots,n_r) & {\rm if}\ \ r\geq 2. 
\end{cases}
\end{align*}
We will show that (i) $e\mid g$ and (ii) $g\mid d$. 
Then the assertion follows from (1) because $d=e$. 

(i) If $r\geq 2$, then 
pick a B\'ezout relation $g=b_1n_1 + \cdots + b_rn_r$ $(b_i \in \bZ)$. 
By the exact sequence $(E_X)$, we get a long exact sequence 
\begin{align*}
H^0(G,\bZ[X])=\bZ[X]^G=\bigoplus_{i=1}^r
\left\langle\sum_{x\in X_i}[x]\right\rangle_{\!\!\!\bZ}
\simeq \bZ^r\xrightarrow{\varepsilon}
H^0(G,\bZ)=\bZ^G\simeq \bZ
\xrightarrow{\delta}H^1(G,I_X) \to H^1(G,\bZ[X])=0
\end{align*}
where the map $\varepsilon: \bZ^r \to \bZ$ is given by 
\begin{align*}
(\lambda_1, \ldots, \lambda_r) \mapsto \lambda_1 n_1+ \ldots+ \lambda_r n_r.
\end{align*}
This implies that 
$H^1(G,I_X)\simeq \bZ/{\rm Image}(\varepsilon)\simeq \bZ/g\bZ$. 
By (1) and the computation
$$(F_X)\in {\rm Ext}^1_{\bZ[G]}(J_X,\bZ)\simeq H^1(G,{\rm Hom}_\bZ(J_X,\bZ))\simeq H^1(G,(J_X)^\circ)\simeq H^1(G,I_X)\simeq\bZ/g\bZ,$$ 
we find that $e\mid g$. 

(ii) From the proof of (i) as above, 
we have ${\rm Ext}^1_{\bZ[G]}(J_X,\bZ)\simeq\bZ/g\bZ$ 
and hence 
there exists an exact sequence of $G$-lattices 
\begin{align*}
(G_X): 0 \to\bZ\to N\to J_X\to 0
\end{align*}
with $\langle (G_X)\rangle={\rm Ext}^1_{\bZ[G]}(J_X,\bZ)$ and 
${\rm ord}(G_X)=g$. 
By applying Proposition \ref{prop6.5}, we get $g\mid d$.\\
(4) If $\gcd(n_1,\ldots,n_r)=1$, then 
${\rm p}$-${\rm ord}(J_X)=1$ and hence it follows from 
Proposition \ref{prop6.5} that $J_X$ is invertible. 
By Lemma \ref{lem6.3}, $(F_X)$ splits and hence $\bZ\oplus J_X\simeq \bZ[X]$. 
This implies that $J_X\oplus\bZ$ is permutation and 
$k(\T_A\times \bG_{m,k})\simeq L(J_X\oplus\bZ)^G\simeq L(J_X)^G(t)$ 
is $k$-rational. 
\end{proof}

\begin{example}\label{ex8.4}
Let $p$ be a prime number and $X$ be a $G$-set with $|X|=p$. 
It follows from Theorem \ref{th8.3} (3) that 
${\rm p}$-${\rm ord}(J_X)={\rm p}$-${\rm ord}(I_X)=p$ or $1$. 
We see that ${\rm p}$-${\rm ord}(J_X)=p$ 
if and only if $X$ is transitive, i.e. 
$J_X\simeq J_{G/H}$ with $G\leq S_p$ transitive and $[G:H]=p$. 
When ${\rm p}$-${\rm ord}(J_X)=p$, by Theorem \ref{thS}, 
the norm one torus $\T_A$ with 
$\widehat{\T}_A\simeq J_X$ 
is retract $k$-rational, i.e. $[J_X]^{fl}$ is invertible 
(that is, ${\rm p}$-${\rm ord}([J_X]^{fl})=1$, 
see Proposition \ref{prop6.5} (6). 
When ${\rm p}$-${\rm ord}(J_X)=1$, by Theorem \ref{th8.3}, 
$\T_A$ is stably $k$-rational. 
\end{example}

By Proposition \ref{prop5.5} (Proposition \ref{prop5.6}) 
and Theorem \ref{th8.3}, we get a relation between $J_{X\times Y}$ 
and $J_X\otimes J_Y$: 
\begin{theorem}\label{th7.4}
Let $X$ and $Y$ be 
$G$-sets with $|X|=m$, $|Y|=n$ and 
$X=\coprod_{i=1}^r X_i$, 
$Y=\coprod_{j=1}^s Y_j$ 
be the disjoint the unions of $G$-orbits 
with $|X_i|=m_i$, $|Y_j|=n_j$ and 
$m=\sum_{i=1}^r m_i$, $n=\sum_{j=1}^s n_j$.
Let 
\begin{align*}
(E_X):&\ 0\to I_X\xrightarrow{\iota_X}\bZ[X]\xrightarrow{\varepsilon_X}\bZ\to 0,\\
(E_Y):&\ 0\to I_Y\xrightarrow{\iota_Y}\bZ[Y]\xrightarrow{\varepsilon_Y}\bZ\to 0,\\
(E_{X\times Y}):&\ 0\to I_{X\times Y}\xrightarrow{\iota_{X\times Y}}\bZ[X\times Y]\xrightarrow{\varepsilon_{X\times Y}}\bZ\to 0
\end{align*}
be exact sequences of $G$-lattices as in Definition \ref{def8.1}.  
There exist natural inclusions of $G$-lattices
\begin{align*}
I_X \otimes I_Y 
&\subset\bZ[X]\otimes I_Y\subset I_{X\times Y}\subset\bZ[X\times Y]\simeq \bZ[X]\otimes\bZ[Y],\\
I_X \otimes I_Y 
&\subset I_X\otimes \bZ[Y]\subset I_{X\times Y}\subset\bZ[X\times Y]\simeq\bZ[X]\otimes\bZ[Y]
\end{align*}
which we use without further notice. 
There exists an exact sequence of $G$-lattices 
\begin{align*}
(E_{X,Y}^{(2)}): 0 \to I_X \otimes I_Y \xrightarrow{\iota} (I_X \otimes \bZ[Y]) \oplus (\bZ[X] \otimes I_Y) \xrightarrow{f} \bZ[X] \otimes \bZ[Y]\simeq\bZ[X\times Y] \xrightarrow{\pi} \bZ \otimes \bZ\simeq \bZ \to 0
\end{align*}
where 
\begin{align*}
\iota(\iota_x \otimes \iota_y) &= (\iota_x \otimes \iota_Y(\iota_y), -\iota_X(\iota_x) \otimes \iota_y)
=(\iota_x\otimes \iota_y, -\iota_x\otimes \iota_y),\\
f (\iota_x \otimes y, x \otimes \iota_y)
&= \iota_X(\iota_x) \otimes y+ x \otimes \iota_Y(\iota_y)= \iota_x\otimes y+x\otimes \iota_y,\\
\pi(x,y) & =\varepsilon_X(x) \otimes \varepsilon_Y(y).
\end{align*}
Then we have 
\begin{align*}
{\rm Image}(f)={\rm Ker}(\pi)=I_{X\times Y}\leq \bZ[X] \otimes \bZ[Y].
\end{align*}
{\rm (1)} 
If $\gcd(m_i,n_j\mid 1\leq i\leq r, 1\leq j\leq s)=1$, 
e.g. $\gcd(m,n)=1$, 
then the exact sequence of $G$-lattices 
\begin{align*}
(E_{X,Y}^{(2L)}): 0 \to I_X \otimes I_Y 
\mathrel{\mathop{\rightleftarrows}^{\mathrm{\iota}}_{t}}
(I_X \otimes \bZ[Y]) \oplus (\bZ[X] \otimes I_Y)  \xrightarrow{f} I_{X\times Y} \to 0
\end{align*}
splits, i.e. 
$(I_X \otimes I_Y)\oplus I_{X\times Y}\simeq 
(I_X \otimes \bZ[Y]) \oplus (\bZ[X] \otimes  I_Y)$, 
that is, there exists 
$t\in {\rm Hom}_{\bZ[G]}((I_X \otimes \bZ[Y]) \oplus (\bZ[X] \otimes I_Y) \to I_X \otimes I_Y)$ 
such that $t\circ\iota=\Id_{I_X \otimes I_Y}$. 
In particular, we also get the dual exact sequence of $G$-lattices 
\begin{align*}
(F_{X,Y}^{(1R)}): 0 \to J_{X\times Y}\xrightarrow{f^\circ}  
(J_X \otimes \bZ[Y]) \oplus (\bZ[X] \otimes J_Y)
\mathrel{\mathop{\rightleftarrows}^{\mathrm{\iota^\circ}}_{t^\circ}}
J_X \otimes J_Y  \to 0
\end{align*}
splits, i.e. 
$J_{X\times Y}\oplus (J_X \otimes J_Y)\simeq 
(J_X \otimes \bZ[Y]) \oplus (\bZ[X] \otimes  J_Y)$, 
that is, there exists 
$t^\circ\in {\rm Hom}_{\bZ[G]}(J_X \otimes J_Y, 
(J_X \otimes \bZ[Y]) \oplus (\bZ[X] \otimes J_Y))$ 
such that $\iota^\circ\circ t^\circ=\Id_{J_X \otimes J_Y}$.\\
{\rm (2)} 
If $\gcd(m_i,n_j\mid 1\leq i\leq r, 1\leq j\leq s)=1$, 
e.g. $\gcd(m,n)=1$, 
then there exists an exact sequence of $G$-lattices
\begin{align*}
(E_{X,Y}^{(2LS)}): 0 \to (I_X \otimes \bZ[Y]) \oplus (\bZ[X] \otimes I_Y) \xrightarrow{(t,f)}  (I_X \otimes I_Y)\oplus (\bZ[X] \otimes \bZ[Y]) \xrightarrow{(0,\pi)}
 \bZ\otimes\bZ \to 0. 
\end{align*}
In particular, we also get the dual exact sequence of $G$-lattices 
\begin{align*}
(E_{X,Y}^{(1RS)}): 0 \to \bZ\otimes\bZ \xrightarrow{(\pi^\circ,0)} 
(\bZ[X] \otimes \bZ[Y])\oplus(J_X \otimes J_Y)
 \xrightarrow{(f^\circ,t^\circ)} (J_X \otimes \bZ[Y]) \oplus (\bZ[X] \otimes J_Y)
 \to 0. 
\end{align*}
\end{theorem}
\begin{proof}
This is a particular case of Proposition \ref{prop5.5}, 
applied to $(E_1)=(E_X)$ and $(E_2)=(E_Y)$. 
Indeed, if $A_1=I_X$, $A_2=I_Y$, $B_1=\bZ[X]$, $B_2=\bZ[Y]$, 
then we get that $(E_{1,2}^{(2)})=(E_{X,Y}^{(2)})$ and 
\begin{align*}
{\rm Image}(f)={\rm Ker}(f)=I_{X\times Y}. 
\end{align*} 
The assertion follows from Theorem \ref{th8.3} because we have 
${\rm ord}(E_X)={\rm p}$-${\rm ord}(I_X)=\gcd(m_1,\ldots,m_r)$, 
${\rm ord}(E_Y)={\rm p}$-${\rm ord}(I_Y)=\gcd(n_1,\ldots,n_s)$ 
and hence $\gcd({\rm ord}(E_X),{\rm ord}(E_Y))
=\gcd(m_i,n_j\mid 1\leq i\leq r, 1\leq j\leq s)$. 
The dual exact sequence follows from 
$(M\otimes N)^\circ\simeq M^\circ\otimes N^\circ$ 
and 
$(I_X)^\circ\simeq J_X$, 
$(I_Y)^\circ\simeq J_Y$, 
$(I_{X\times Y})^\circ\simeq J_{X\times Y}$, 
$\bZ[X]^\circ\simeq\bZ[X]$, 
$\bZ[Y]^\circ\simeq \bZ[Y]$ 
where $M^\circ={\rm Hom}(M,\bZ)$ is the dual of a $G$-lattice $M$ 
(see also Proposition \ref{prop5.6}). 
\end{proof}
%
The next Theorem is the main application of our results 
to the algebraic $k$-torus $\T_A\otimes \T_B$ and 
the norm one tori $\T_{A\otimes B}$ where 
$\T_A=R^{(1)}_{A/k}(\bG_m)$ is the norm one torus 
of the \'etale algebra $A/k$. 
\begin{theorem}\label{th7.5}
Let $k$ be a field. 
Let $A=\prod_{i=1}^r K_i$ and $B=\prod_{j=1}^s E_j$ be \'etale $k$-algebras 
with ${\rm dim}_k\, A=m$, ${\rm dim}_k\, B=n$  
where $K_i$ and $E_j$ are finite separable field extensions of $k$ with 
$[K_i:k]=m_i$, $[E_j:k]=n_j$ and 
$m=\sum_{i=1}^r m_i$, 
$n=\sum_{j=1}^s n_j$. 
Let $\T_A=R^{(1)}_{A/k}(\bG_m)$ be 
the norm one torus of the \'etale algebra $A/k$. 
If $\gcd(m_i,n_j\mid 1\leq i\leq r, 1\leq j\leq s)=1$, 
e.g. $\gcd(m,n)=1$, and 
$\T_A$ and $\T_B$ are stably $($resp. retract$)$ $k$-rational, then 
the algebraic $k$-torus $\T_A\otimes \T_B$ and 
the norm one torus $\T_{A \otimes B}$ are stably 
$($resp. retract$)$ $k$-rational.

In particular, if $r=s=1$, 
i.e. $A$, $B$ are finite separable field extensions of $k$, 
then 
$A\otimes B$ is a finite separable 
field extension of $k$ with $[A\otimes B:k]=mn$ and 
the norm one torus $\T_{A\otimes B}$ 
of a field extension $(A\otimes B)/k$ of dimension $mn-1$ 
is stably $($resp. retract$)$ $k$-rational. 
\end{theorem}
\begin{proof}
Set $X:=X(A)$ and $Y:=X(B)$. 
Recall that the character lattice of $\T_A$ (resp. $\T_B$) is $J_X$ 
(resp. $J_Y$)
and the permutation order $\rm{p}$-${\rm ord}(J_X)$ 
(resp. $\rm{p}$-${\rm ord}(J_Y)$) 
has been computed in Theorem \ref{th8.3} (3). 
Define $d_X:={\rm p}$-${\rm ord}(J_X)=\gcd(m_i\mid 1\leq i\leq r)$ 
and 
$d_Y:={\rm p}$-${\rm ord}(J_Y)=\gcd(n_j\mid 1\leq j\leq s)$. 
Then we see $\gcd(d_X,d_Y)=\gcd(m_i,n_j\mid 1\leq i\leq r, 1\leq j\leq s)$. 

By the assumption $\gcd(d_X,d_Y)=1$, 
it follows from Theorem \ref{th6.15} that 
if $\T_A$ and $\T_B$ are stably $($resp. retract$)$ $k$-rational, 
then $\T_A\otimes \T_B$ is stably $($resp. retract$)$ $k$-rational, 
i.e. $[J_X\otimes J_Y]^{fl}=0$ (resp.  $[J_X\otimes J_Y]^{fl}$ is invertible). 

Next, for the norm one torus $\T_{A\otimes B}$, 
we consider the extensions of $G$-lattices 
\begin{align*}
&(F_X): 0 \to \bZ \to \bZ[X]  \to J_{X} \to 0\\
&(F_Y): 0 \to \bZ \to \bZ[Y] \to J_{Y} \to 0
\end{align*}
with $\gcd({\rm ord}(F_X), {\rm ord}(F_Y))=1$ as in Definition \ref{def8.1}.   
%
We apply Theorem \ref{th7.4}. 
Then the isomorphism of $G$-lattices thereof, can be rewritten here as: 
\begin{align*}
J_{X\times Y} \oplus (J_X \otimes J_Y)\simeq
(J_X \otimes \bZ[Y])\oplus (\bZ[X] \otimes J_Y).
\end{align*}
In particular, we have 
\begin{align*}
[J_{X\times Y}]^{fl}+[J_X\otimes J_Y]^{fl}=
[J_X \otimes \bZ[Y]]^{fl}+[\bZ[X] \otimes J_{Y}]^{fl}.
\end{align*}
In geometric wording, this corresponds to an isomorphism of $k$-tori 
\begin{align*}
\T_{A \otimes B} \times (\T_A\otimes \T_B)\simeq 
(\T_A\otimes R_{B/k}(\bG_m))\times (R_{A/k}(\bG_m)\otimes \T_B).
\end{align*}

By Theorem \ref{th6.15} again, 
because $R_{A/k}(\bG_{m,A})$ 
with character lattice $\bZ[X]$ and ${\rm p}$-${\rm ord}(\bZ[X])=1$ 
and $R_{B/k}(\bG_{m,B})$ 
with character lattice $\bZ[Y]$ and ${\rm p}$-${\rm ord}(\bZ[Y])=1$ 
are $k$-rational, 
if $\T_A$ and $\T_B$ are stably (resp. retract) $k$-rational, 
then $\T_A\otimes R_{B/k}(\bG_{m,B})$ and 
$R_{A/k}(\bG_{m,A})\otimes \T_B$ are stably (resp. retract) $k$-rational, 
i.e. $[J_X\otimes \bZ[Y]]^{fl}=[\bZ[X]\otimes J_Y]^{fl}=0$ 
(resp. $[J_X\otimes \bZ[Y]]^{fl}$ and $[\bZ[X]\otimes J_Y]^{fl}$ 
are invertible). 

This implies that if 
$\gcd(d_X,d_Y)=1$ 
and 
$\T_A$ and $\T_B$ are stably (resp. retract) $k$-rational, then 
$\T_{A\otimes B}$ is stably (resp. retract) $k$-rational, 
i.e. $[J_{X\times Y}]^{fl}=0$ 
(resp. $[J_{X\times Y}]^{fl}$ is invertible). 

For the last assertion, if $r=s=1$, then $A$, $B$ 
are finite separable field extensions of $k$. 
If $\gcd(m,n)=1$, then 
the surjective map $A\otimes B\to AB$ becomes isomorphic 
because ${\rm dim}_k\, A\otimes B=mn$ and 
it follows from $[A:k]=m$, $[B:k]=n$, $A\cap B=k$ that 
${\rm dim}_k\,AB=[AB:k]=mn$
where $AB$ is the composite field of $A, B$, 
i.e. the smallest field which contains $A$, $B$. 
\end{proof}

There is a partial converse to Theorem \ref{th7.5}: 
\begin{proposition}\label{propconv}
Let $A$ and $B$ be \'etale $k$-algebras. 
Denote by $k \subset L_1 \subset \overline k$ 
$($resp. $k \subset L_2 \subset \overline k$$)$ 
the minimal Galois splitting field of $A$ $($resp. $B$$)$. 
Define $G_i:=\Gal(L_i/k)$ $(i=1,2)$. 
Let $\T_A=R^{(1)}_{A/k}(\bG_m)$ be 
the norm one torus of the \'etale algebra $A/k$. 
If $L_1$ and $L_2$ are $k$-linearly disjoint, 
i.e. $L_1 \otimes L_2=L_1L_2$ is a Galois field extension of $k$ 
with $\Gal(L_1L_2/k)\simeq G_1 \times G_2$ 
and $\T_{A\otimes B}$ is stably $($resp. retract$)$ $k$-rational, then 
$\T_A$ and $\T_B$ are stably $($resp. retract$)$ $k$-rational.
\end{proposition}
\begin{proof}
We can apply 
Theorem \ref{th4.2} and use a technique of 
\cite[The proof of Theorem 2.1, page 89]{End11} by Shizuo Endo. 
Denote by $X:=X(A)$ the $G_1$-set associated to $A$ and 
$Y:=X(B)$ the $G_2$-set associated to $B$ with $n_1:=|X|$, $n_2:=|Y|$. 
Define $G:=G_1 \times G_2$.  
The character lattice of the algebraic $k$-torus $\T_A \otimes \T_B$ is 
the $G$-lattice $J_{X \times Y}$, fitting into the extension
\begin{align*}
(F_{X \times Y}): 0 \to\bZ\to
\bZ[X \times Y]\to J_{X \times Y}\to 0.
\end{align*}
Restricting the action of $G$ to $G_1$, we get the exact sequence of 
$G_1$-lattices: 
\begin{align*}
 0 \to\bZ\to
\bZ[X]^{\oplus n_2}\to J_{X \times Y}|_{G_1}\to 0.
\end{align*}
By Theorem \ref{th4.2}, we get an isomorphism of $G_1$-lattices
\begin{align*}
J_{X \times Y}|_{G_1} \simeq J_{X} \bigoplus \bZ[X]^{\oplus n_2-1}
\end{align*}
and $L(J_{X\times Y})^G|_{G_1}=L(J_{X})^{G_1}(u_1,\ldots,u_t)$, 
i.e. $[J_{X\times Y}]^{fl}|_{G_1}=[J_{X}]^{fl}$ as $G_1$-lattices. 
Hence if $[J_{X\times Y}]^{fl}=0$ (resp. $[J_{X\times Y}]^{fl}$ 
is invertible) as a $G$-lattice, 
then $[J_{X}]^{fl}=0$ (resp. $[J_{X}]^{fl}$ is invertible) 
as a $G_1$-lattice. 
This implies that $\T_A$ is stably (resp. retract) $k$-rational 
(via Theorem \ref{th2.7}). 
Similarly, we obtain the same result for $\T_B$ by restricting 
the action of $G$ to $G_2$. 
\end{proof}

\begin{example}\label{ex7.7}
The condition  $\gcd(m,n)=1$ 
in Theorem \ref{th7.5} is necessary.\\ 
(1) $m={\rm p}$-${\rm ord}(J_X)=|X|=3$, $n={\rm p}$-${\rm ord}(J_Y)=|Y|=3$, 
$\gcd(m,n)=3$ and $k=\bQ$, we see that 
if $A=\bQ(\sqrt[3]{2})$, $B=\bQ(\theta)$ 
with $\theta^3-\theta-1=0$ 
(resp. $A=\bQ(\sqrt[3]{2})$, $B=\bQ(\sqrt[3]{5})$), 
then ${\rm Gal}(L/\bQ)\simeq S_3\times S_3$ 
(resp. ${\rm Gal}(L/\bQ)={\rm Gal}(\bQ(\sqrt[3]{2},\sqrt[3]{5},\sqrt{-3})/\bQ)
\simeq (C_3\times C_3)\rtimes C_2$) and 
$\T_A\otimes\T_B$ of 
dimension $2\cdot 2=4$ is not retract $k$-rational 
although $\T_A=R^{(1)}_{A/k}(\bG_m)$ and $\T_B=R^{(1)}_{B/k}(\bG_m)$ are 
stably $k$-rational.\\ 
(2) 
$m={\rm p}$-${\rm ord}(J_X)=|X|=2$, $n={\rm p}$-${\rm ord}(J_Y)=|Y|=2$, 
$\gcd(m,n)=2$ and $k=\bQ$, we see that 
if $A=\bQ(\sqrt{2})$, $B=\bQ(\sqrt{3})$, then 
$\T_{A\otimes B}=R^{(1)}_{A\otimes B/k}(\bG_m)
=R^{(1)}_{\bQ(\sqrt{2},\sqrt{3})/\bQ}(\bG_m)$ 
of dimension $2\cdot 2-1=3$ is not retract $k$-rational 
(see Kunyavskii \cite[Theorem 1]{Kun90}, 
Hoshi and Yamasaki \cite[Theorem 1.2]{HY17} for the case $U_1$) 
although $\T_A=R^{(1)}_{A/k}(\bG_m)$ and $\T_B=R^{(1)}_{B/k}(\bG_m)$ are 
stably $k$-rational.\\
(3) 
Let 
$A=K_1=\bQ(\theta)$ (resp. $B=F_1=\bQ(\eta)$) 
be the cyclic cubic field with discriminant $7^2$ (resp. $9^2$) 
where $\theta^3+\theta^2-2\theta-1=0$ 
(resp. $\eta^3-3\eta-1=0$) with 
$[A:k]=[B:k]=3$, 
${\rm Gal}(A/k)\simeq {\rm Gal}(B/k)\simeq C_3$ and $A\cap B=k$. 
Then we find that $G={\rm Gal}(L/k)\simeq C_3\times C_3$ and 
$m={\rm p}$-${\rm ord}(J_X)=3$, $n={\rm p}$-${\rm ord}(J_Y)=3$, $\gcd(m,n)=3$. 
We see that 
$\T_A\otimes\T_B$ of 
dimension $2\cdot 2=4$ 
and $\T_{A\otimes B}=R^{(1)}_{A\otimes B/k}(\bG_m)$ of 
dimension $3\cdot 3-1=8$ are not retract $k$-rational 
although $\T_A=R^{(1)}_{A/k}(\bG_m)$ and $\T_B=R^{(1)}_{B/k}(\bG_m)$ are 
stably $k$-rational. 

We can confirm the claims above by using the GAP (\cite{GAP}) function 
{\tt IsInvertibleF} as in Hoshi and Yamasaki 
\cite[Algorithm F2, page 84]{HY17}. 
The related GAP functions are also available in \cite{RatProbAlgTori}. 
\end{example}

\begin{example}\label{ex7.5}
If $F_j=k$ $(1\leq j\leq s)$, then $B=\prod_{j=1}^s k\simeq k\times \cdots\times k$ and  
$A\otimes B\simeq \prod_{j=1}^s A=A\times\cdots\times A$. 
It follows from $n_j=1$ $(1\leq j\leq s)$ that 
$\gcd(m_i,n_j\mid 1\leq i\leq r, 1\leq j\leq s)=1$. 
By Theorem \ref{th7.5}, we see that 
if $\T_A=R^{(1)}_{A/k}(\bG_m)$ is stably $($resp. retract$)$ $k$-rational, 
then $\T_{A\otimes B}=R^{(1)}_{(A \otimes B)/k}(\bG_m)\simeq R^{(1)}_{(\prod_{j=1}^s A)/k}(\bG_m)
\simeq R^{(1)}_{(A\times\dots\times A)/k}(\bG_m)$ 
is stably $($resp. retract$)$ $k$-rational. 
This is a special case of a theorem  of Endo \cite[Proposition 1.3, Corollary 1.4]{End11}. 
We note that the inverse direction of the statement also holds by 
Endo's theorem.

For example, if $A=\bQ(\sqrt{2})$ and $B=\bQ\times\cdots\times \bQ$, then 
$A\otimes B\simeq \bQ(\sqrt{2})\otimes(\bQ\times\cdots\times \bQ)
\simeq \bQ(\sqrt{2})\times\cdots\times \bQ(\sqrt{2})$ and 
$\T_{A\otimes B}=R^{(1)}_{(\bQ(\sqrt{2})\times\dots\times \bQ(\sqrt{2}))/\bQ}(\bG_m)$ 
is stably $\bQ$-rational.
\end{example}

\section{Applications to Hasse norm principle for field extensions $K/k$ where $K=K_1\otimes\cdots\otimes K_t$}

We treat the case where $r=s=1$, i.e. $A=K_1$, $B=K_2$, 
of Theorem \ref{th7.5} but for general 
$K=K_1\otimes \cdots\otimes K_t$. 

\begin{theorem}\label{th8.1}
Let $k$ be a field, 
$\overline{k}$ be a fixed separable closure of $k$,  
$K_i/k$ $(1\leq i\leq t)$ be a finite separable field extension 
with $[K_i:k]=n_i$ 
and $L_i/k$ 
be the Galois closure of $K_i/k$ in $\overline{k}$ 
with Galois groups $G_i={\rm Gal}(L_i/k)$ and $H_i={\rm Gal}(L_i/K_i)\leq G_i$ 
with $[G_i:H_i]=n_i$.  
Let $L=L_1\cdots L_t\subset \overline{k}$ be the composite field of 
$L_1,\ldots,L_t$, 
i.e. the smallest field which contains all $L_i$, 
with Galois group $G={\rm Gal}(L/k)$ which is 
a subdirect product of $G_1,\ldots,G_t$, i.e. 
a subgroup $G\leq G_1\times\cdots\times G_t$ 
with surjections $\varphi_i: G\rightarrow G_i$ $(1\leq i\leq t)$. 

Assume that 
$\gcd(n_i,n_j)=1$ for any $1\leq i<j\leq t$. 
Then $K:=K_1\otimes\cdots\otimes K_t=K_1\cdots K_t$ 
is a finite separable field extension of $k$ with 
$[K:k]=n:=n_1\cdots n_t$, 
$G={\rm Gal}(L/k)=nTm\leq S_n$ transitive and 
$H={\rm Gal}(L/K)\leq G$ with 
$H=\cap_{i=1}^t\varphi_i^{-1}(H_i)$ and $[G:H]=n$.
Let $\T_{K_i}=R^{(1)}_{K_i/k}(\bG_m)$ $(1\leq i\leq t)$ be 
the norm one torus of $K_i/k$. 
If $\T_{K_i}$ 
is stably $($resp. retract$)$ $k$-rational 
for any $1\leq i\leq t$, 
then the algebraic $k$-torus $\T_{K_1}\otimes\cdots\otimes \T_{K_t}$ 
with character lattice 
$J_{G_1/H_1}\otimes\cdots\otimes J_{G_t/H_t}$
and 
the norm one torus $\T_K$ with $\widehat{\T}_K\simeq J_{G/H}$ 
are stably $($resp. retract$)$ $k$-rational. 

In particular, by Theorem \ref{thV} 
$($see also Section \ref{S2} and Section \ref{S4}$)$, 
if $k$ is a global field and $\T_K=R^{(1)}_{K/k}(\bG_m)$ 
is stably $($resp. retract$)$ $k$-rational, then 
$\Sha^2_\omega(G,J_{G/H})\simeq H^1(k,{\rm Pic}\, \overline{X})\simeq 
H^1(G,{\rm Pic}\, X_K)\simeq H^1(G,[J_{G/H}]^{fl})\simeq 
{\rm Br}(X)/{\rm Br}(k)\simeq 
{\rm Br}_{\rm nr}(k(X))/{\rm Br}(k)=0$ 
where $X$ is a smooth $k$-compactification of $\T_K$ 
and hence 
$A(\T_K)=0$, i.e. $\T_K$ has the weak approximation property, 
and $\Sha(\T_K)=0$, i.e. the Hasse norm principle holds for $A/k$ 
$($that is, Hasse principle holds for all torsors $E$ under $\T_K$$)$ 
where $\Sha(\T_K)^\vee\leq \Sha^2_\omega(G,J_{G/H})$ 
and $\Sha(\T_K)^\vee={\rm Hom}(\Sha(\T_K),\bQ/\bZ)$. 
The condition $\Sha(\T_K)=0$ means that 
for the corresponding norm hypersurface 
\begin{align*}
f(x_1,\ldots,x_n)=b, 
\end{align*}
it has a $k$-rational point 
if and only if it has a $k_v$-rational point 
for any place $v$ of $k$ where 
$f(x_1,\ldots,x_n)\in k[x_1,\ldots,x_n]$ 
is the polynomial of total degree $n$ 
defined as 
\begin{align*}
f(x_1,\ldots,x_n)=N_{K/k}(x_1w_1+\cdots+x_nw_n)=\prod_{\overline{g}\in G/H}\overline{g}(x_1w_1+\cdots+x_nw_n)
\end{align*}
where 
$\{w_1,\ldots,w_n\}$ is a basis of $K/k$, 
$N_{K/k}:K^\times\to k^\times$ is the norm map 
and $b\in k^\times$ $($see Voskresenskii \cite[Example 4, page 122]{Vos98}$)$. 
\end{theorem}
\begin{proof}
By the assumption $\gcd(n_i,n_j)=1$ for any $1\leq i<j\leq t$, 
this follows from Theorem \ref{th7.5} repeatedly. 
\end{proof}

We wiil use the following lemma which gives 
the set of all subdirect products of $G_1$, $G_2$. 
\begin{lemma}[{Hoshi and Yamasaki \cite[Proof of Theorem 8.1 (1)]{HY1}}]\label{lem2.2} 
Let $G_1$ and $G_2$ be finite groups. 
For a subdirect product $G\leq G_1\times G_2$ of $G_1$, $G_2$ 
with surjections $\varphi_1: G\to G_1$, $\varphi_2: G\to G_2$, define 
\begin{align*}
N_1:=\varphi_1({\rm Ker}(\varphi_2))\lhd G_1, 
N_2:=\varphi_2({\rm Ker}(\varphi_1))\lhd G_2, 
\pi_1:G_1\rightarrow G_1/N_1, 
\pi_2:G_2\rightarrow G_2/N_2.
\end{align*}
Then we have 
$\overline{\varphi}_1=\pi_1\circ\varphi_1: G\rightarrow G_1/N_1$, 
$\overline{\varphi}_2=\pi_2\circ\varphi_2: G\rightarrow G_2/N_2$, 
and $\overline{\varphi}=(\overline{\varphi}_2)(\overline{\varphi}_1)^{-1}:
G_1/N_1\xrightarrow{\sim}G_2/N_2$. 
We find that a subdirect product 
$G\leq G_1\times G_2$ of $G_1$, $G_2$ 
with surjections $\varphi_1: G\to G_1$, $\varphi_2: G\to G_2$ 
is given by 
\begin{align*}
G=\{(g_1,g_2)\in G_1\times G_2\mid \overline{\varphi}(\pi_1(g_1))=\pi_2(g_2), 
\overline{\varphi}=(\overline{\varphi}_2)(\overline{\varphi}_1)^{-1}:
G_1/N_1\xrightarrow{\sim}G_2/N_2\}
\end{align*}
and hence there exists a one-to-one correspondence 
between the set of all subdirect products $G$ of $G_1$, $G_2$ 
with surjections $\varphi_1: G\to G_1$, $\varphi_2: G\to G_2$ and 
\begin{align*}
\{(N_1,N_2,\overline{\varphi})\mid N_1\lhd G_1, N_2\lhd G_2, 
\overline{\varphi}=(\overline{\varphi}_2)(\overline{\varphi}_1)^{-1}:
G_1/N_1\xrightarrow{\sim}G_2/N_2\}. 
\end{align*}
\end{lemma}

\subsection{Stably $k$-rational norm one tori $\T_K=R^{(1)}_{K/k}(\bG_m)$ of field extensions $K/k$ where $K=K_1\otimes\cdots\otimes K_t$}
\begin{theorem}[Application of Theorem \ref{th8.1}: Stably $k$-rational norm one tori $\T_K=R^{(1)}_{K/k}(\bG_m)$ of field extensions $K/k$ where $K=K_1\otimes\cdots\otimes K_t$]\label{th8.2}
Let the notation be given as in Theorem \ref{th8.1}.
%
Assume that 
$\gcd(n_i,n_j)=1$ for any $1\leq i<j\leq t$. 
Then $K:=K_1\otimes\cdots\otimes K_t=K_1\cdots K_t$ is a finite separable 
field extension of $k$ with 
$[K:k]=n:=n_1\cdots n_t$, 
$G={\rm Gal}(L/k)\leq G_1\times\cdots\times G_{n_t}$, 
$G=nTm\leq S_n$ transitive and 
$H={\rm Gal}(L/K)\leq G$ with $H=\cap_{i=1}^t\varphi_i^{-1}(H_i)$ 
and $[G:H]=n$.
If we choose $G_i\leq S_{n_i}$ transitive 
and $H_i$ with $[G_i:H_i]=n_i$ $(1\leq i\leq t)$ and 
$\gcd(n_i,n_j)=1$ for any $1\leq i<j\leq t$ from the following table 
\begin{align*}
\begin{tabular}{c|cccccccc} 
$G_i$ & $C_{m_i}$ & $D_{m_i}$ & $D_{m_i}$ & $MC_{u,2^d}$ & $A_5$ & $A_5$ & $\PSL_2(\bF_8)$ & $\PSL_2(\bF_8)$ \\\hline
$H_i$ & $\{1\}$ & $\{1\}$ & $C_2$ & $\{1\}$ & $C_2^2$ & $A_4$ & $C_2^3$ & $C_2^3\rtimes C_7$\\\hline
$n_i$ & $m_i$ & $2m_i$ & $m_i$ & $2^du$ & $15$ & $5$ & $63$ & $9$
\end{tabular}
\end{align*}
where 
\begin{align*}
MC_{u,2^d}=\langle\sigma,\tau\mid\sigma^u=\tau^{2^d}=1,
\tau\sigma\tau^{-1}=\sigma^{-1}\rangle\simeq C_u\rtimes C_{2^d}\ 
(d\geq 1, u\geq 3: {\rm odd}).
\end{align*} 
Then the norm one torus 
$\T_{K_i}=R^{(1)}_{K_i/k}(\bG_m)$ of $K_i/k$  
is stably $k$-rational 
for any $1\leq i\leq t$ 
and hence 
the norm one torus 
$\T_K=R^{(1)}_{K/k}(\bG_m)$ of $K/k$ 
with $\widehat{\T}_K\simeq J_{G/H}$ 
is stably $k$-rational. 
In particular, if $k$ is a global field, then 
$\Sha^2_\omega(J_{G/H})=0$ and hence 
$\T_K$ has the weak approximation property and 
the Hasse norm principle holds for $K/k$.  
\end{theorem}
\begin{proof}
This follows from Theorem \ref{th8.1} with the aid of 
Theorem \ref{th1.11}, 
Theorem \ref{th1.12}, 
Theorem \ref{th1.14}, 
Theorem \ref{th1.17}, 
Theorem \ref{th1.18}, 
Theorem \ref{th1.23} 
and Theorem \ref{th1.24}. 
\end{proof}
\begin{example}[Example of Theorem \ref{th8.2}: stably $k$-rational tori $\T_K=R^{(1)}_{K/k}(\bG_m)$ for $G={\rm Gal}(L/k)=nTm\leq S_n$ $(2\leq n\leq 15)$ as in Theorem \ref{th1.23} (Table $2$)]\label{ex8.3}~\\
Assume that $n=n_1\cdots n_t$ with $\gcd(n_i,n_j)=1$ for any $1\leq i<j\leq t$, and that  
each $\T_{K_i}=R^{(1)}_{K_i/k}(\bG_m)$ $(1\leq i\leq t)$ is stably $k$-rational. 
By applying Theorem \ref{th8.1}, 
we see that $R^{(1)}_{K/k}(\bG_m)$ with $K=K_1 \otimes \cdots \otimes  K_t$ is stably $k$-rational. 
Then the Galois group $G={\rm Gal}(L/K)=nTm\leq S_n$ is a subdirect product of 
$G_i={\rm Gal}(L_i/k)=n_iTm_i\leq S_{n_i}$ $(1\leq i\leq t)$. 
By Lemma \ref{lem2.2}, 
all of the subdirect product 
$G=nTm\leq S_n$ $(n=6$, $10$, $12$, $14$, $15)$ of $G_1=n_1Tm_1\leq S_{n_1}$, 
$G_2=n_2Tm_2\leq S_{n_2}$ with $G\leq G_1\times G_2$, 
$n=n_1n_2$ and $\gcd(n_1,n_2)=1$ 
when $\T_{K_i}=R^{(1)}_{K_i/k}(\bG_m)$ $(i=1,2)$ is 
stably $k$-rational 
as in Theorem \ref{th1.23} (Table $2$) 
are given as follows:\\
(1) ($n=6=n_1n_2=2\cdot 3$) $G=6Tm\leq S_2\times S_3$;\\ 
$6T1\simeq C_2\times C_3$, $6T3\simeq C_2\times S_3$, 
$6T2\simeq S_3\leq C_2\times S_3$ with $C_2\simeq S_3/C_3$.\\
(2) $n=10=n_1n_2=2\cdot 5$: $G=10Tm\leq S_2\times S_5$;\\ 
$10T1\simeq C_2\times C_5$, $10T3\simeq C_2\times D_5\simeq D_{10}$, 
$10T11\simeq C_2\times A_5$,\\ 
$10T2\simeq D_5\leq C_2\times D_5$ with $C_2\simeq D_5/C_5$.\\
(3) $n=12=n_1n_2=3\cdot 4$: $G=12Tm\leq S_3\times S_4$;\\ 
$12T1\simeq C_3\times C_4\simeq C_{12}$, $12T11\simeq S_3\times C_4$,\\ 
$12T5\simeq Q_{12}\simeq C_3\rtimes C_4\leq S_3\times C_4$ with $S_3/C_3\simeq C_4/C_2$.\\ 
(4) $n=14=n_1n_2=2\cdot 7$: $G=14Tm\leq S_2\times S_7$;\\ 
$14T1\simeq C_2\times C_7\simeq C_{14}$, 
$14T3\simeq C_2\times D_7\simeq D_{14}$,\\ 
$14T2\simeq D_7\leq C_2\times D_7$ with $C_2\simeq D_7/C_7$.\\
(5) $n=15=n_1n_2=3\cdot 5$: $G=15Tm\leq S_3\times S_5$;\\ 
$15T1\simeq C_3\times C_5\simeq C_{15}$, 
$15T3\simeq C_3\times D_5$, 
$15T4\simeq S_3\times C_5$, 
$15T7\simeq S_3\times D_5$, $15T16\simeq C_3\times A_5$, 
$15T23\simeq S_3\times A_5$,\\ 
$15T2\simeq D_{15}\leq S_3\times D_5$ with $S_3/C_3\simeq D_5/C_5$. 
\end{example}
\begin{remark}\label{rem8.4}
In Theorem \ref{th1.23}, 
the case where 
$15T5\simeq A_5$ as in Table $2$ (stably $k$-rational cases) 
can not be obtained as a subdirect product 
$G=nTm\leq S_n$ 
of $G_1=n_1Tm_1\leq S_{n_1}$, 
$G_2=n_2Tm_2\leq S_{n_2}$ with $n=n_1n_2$ and $\gcd(n_1,n_2)=1$ 
as in Example \ref{ex8.3}. 
\end{remark}
\subsection{Retract $k$-rational norm one tori $\T_K=R^{(1)}_{K/k}(\bG_m)$ of field extensions $K/k$ where $K=K_1\otimes\cdots\otimes K_t$}
\begin{theorem}[Application of Theorem \ref{th8.1}: Retract $k$-rational norm one tori $\T_K=R^{(1)}_{K/k}(\bG_m)$ of field extensions $K/k$ where $K=K_1\otimes\cdots\otimes K_t$]\label{th8.5}
Let the notation be given as in Theorem \ref{th8.1}. 
Assume that 
$\gcd(n_i,n_j)=1$ for any $1\leq i<j\leq t$. 
Then $K:=K_1\otimes\cdots\otimes K_t=K_1\cdots K_t$ is a finite separable 
field extension of $k$ with 
$[K:k]=n:=n_1\cdots n_t$, 
$G={\rm Gal}(L/k)\leq G_1\times\cdots\times G_{n_t}$, 
$G=nTm\leq S_n$ transitive and 
$H={\rm Gal}(L/K)\leq G$ with $H=\cap_{i=1}^t\varphi_i^{-1}(H_i)$ 
and $[G:H]=n$. 
If we choose $G_i$ and $H_i$ with $[G_i:H_i]=n_i$ $(1\leq i\leq t)$ and 
$\gcd(n_i,n_j)=1$ for any $1\leq i<j\leq t$ from the following table 
\begin{align*}
\begin{tabular}{c|cccccccc|cccccc} 
$G_i$ & $C_{m_i}$ & $D_{m_i}$ & $D_{m_i}$ & $MC_{u,2^d}$ & $A_5$ & $A_5$ & $\PSL_2(\bF_8)$ & $\PSL_2(\bF_8)$ & $G^\prime_i$ & $p_iTm_i$ & $S_5$ & $S_5$ & $S_5$ & $\PSL_2(\bF_7)$\\\hline
$H_i$ & $\{1\}$ & $\{1\}$ & $C_2$ & $\{1\}$ & $C_2^2$ & $A_4$ & $C_2^3$ & $C_2^3\rtimes C_7$ & $H^\prime_i$ & $H^{\prime\prime}_i$ & $H^{\prime\prime\prime}_i$ & $D_4$ & $A_4$ & $D_4$\\\hline
$n_i$ & $m_i$ & $2m_i$ & $m_i$ & $2^du$ & $15$ & $5$ & $63$ & $9$ & $n_i$ & $p_i$ & $30$ & $15$ & $10$ & $21$
\end{tabular}
\end{align*}
where 
\begin{align*}
MC_{u,2^d}=\langle\sigma,\tau\mid\sigma^u=\tau^{2^d}=1,
\tau\sigma\tau^{-1}=\sigma^{-1}\rangle\simeq C_u\rtimes C_{2^d}\ 
(d\geq 1, u\geq 3: {\rm odd}), 
\end{align*}
$G^\prime_i$ is a group whose all the Sylow subgroups are cyclic 
$($see also Remark \ref{rem3.6}$)$, 
$H^\prime_i\leq G^\prime_i\leq S_{n_i}$ 
with $[G_i^\prime:H_i^\prime]=n_i$, 
$H^{\prime\prime}_i={\rm Stab}_{p_iT_{m_i}}(1)$; 
the stabilizer of $1$ in $p_iT_{m_i}\leq S_{p_i}$ with 
$[p_iT_{m_i}:H^{\prime\prime}_i]=p_i$; prime and 
$H^{\prime\prime\prime}_i\simeq V_4\leq D(S_5)\simeq A_5\leq S_5$. 
Then the norm one torus 
$\T_{K_i}=R^{(1)}_{K_i/k}(\bG_m)$ of $K_i/k$  
is retract $k$-rational 
for any $1\leq i\leq t$ 
and hence 
the norm one torus 
$\T_K=R^{(1)}_{K/k}(\bG_m)$ of $K/k$ 
with $\widehat{\T}_K\simeq J_{G/H}$ is retract $k$-rational. 
In particular, if $k$ is a global field, then 
$\Sha^2_\omega(J_{G/H})=0$ and hence 
$\T_K$ has the weak approximation property and 
the Hasse norm principle holds for $K/k$.   
\end{theorem}
\begin{proof}
This follows from Theorem \ref{th8.1} with the aid of 
Theorem \ref{th1.11}, 
Theorem \ref{th1.12}, 
Theorem \ref{th1.14}, 
Theorem \ref{th1.17}, 
Theorem \ref{th1.18}, 
Theorem \ref{th1.23} 
and Theorem \ref{th1.24}. 
\end{proof}

\begin{example}[Example of Theorem \ref{th8.5}: retract $k$-rational tori $\T_K=R^{(1)}_{K/k}(\bG_m)$ for $G={\rm Gal}(L/k)=nTm\leq S_n$ $(2\leq n\leq 15)$ as in Theorem \ref{th1.23} (Table $3$)]\label{ex8.6}~\\
By Theorem \ref{th8.1}, 
when $n=n_1\cdots n_t$ with $\gcd(n_i,n_j)=1$ for any $1\leq i<j\leq t$, 
if $\T_{K_i}=R^{(1)}_{K_i/k}(\bG_m)$ $(1\leq i\leq t)$ is retract $k$-rational, 
then $\T_K=R^{(1)}_{K/k}(\bG_m)$ with $K=K_1\cdots K_t$ is retract $k$-rational. 
Then the Galois group $G={\rm Gal}(L/K)=nTm\leq S_n$ is a subdirect product of 
$G_i={\rm Gal}(L_i/k)=n_iTm_i\leq S_{n_i}$ $(1\leq i\leq t)$. 
By Lemma \ref{lem2.2}, all of the subdirect product 
$G=nTm\leq S_n$ $(n=10$, $14$, $15)$ of $G_1=n_1Tm_1\leq S_{n_1}$, 
$G_2=n_2Tm_2\leq S_{n_2}$ with $G\leq G_1\times G_2$, 
$n=n_1n_2$ and $\gcd(n_1,n_2)=1$ 
when $\T_{K_i}=R^{(1)}_{K_i/k}(\bG_m)$ $(i=1,2)$ 
is not stably but retract $k$-rational 
as in Theorem \ref{th1.23} (Table $3$) 
are given as follows:\\
%
(1) $n=10=n_1n_2=2\cdot 5$: $G=10Tm\leq S_2\times S_5$.\\ 
$10T5\simeq C_2\times F_{20}$, $10T22\simeq C_2\times S_5$,\\
$10T4\simeq F_{20}\leq C_2\times F_{20}$ with $C_2\simeq F_{20}/D_5$, 
$10T12\simeq S_5\leq C_2\times S_5$ with $C_2\simeq S_5/A_5$.\\ 
(2) $n=14=n_1n_2=2\cdot 7$: $G=14Tm\leq S_2\times S_7$.\\ 
$14T5\simeq C_2\times F_{21}$, $14T7\simeq C_2\times F_{42}$, 
$14T19\simeq C_2\times {\rm PSL}_3(\bF_2)$, 
$14T47\simeq C_2\times A_7$, 
$14T49\simeq C_2\times S_7$,\\ 
$14T4\simeq F_{42}\leq C_2\times F_{42}$ with $C_2\simeq F_{42}/F_{21}$, 
$14T46\simeq S_7\leq C_2\times S_7$ with $C_2\simeq S_7/A_7$.\\ 
(3) $n=15=n_1n_2=3\cdot 5$: $G=15Tm\leq S_3\times S_5$.\\ 
$15T8\simeq C_3\times F_{20}$, $15T11\simeq S_3\times F_{20}$, 
$15T24\simeq C_3\times S_5$, $15T29\simeq S_3\times S_5$,\\ 
$15T6\simeq C_{15}\rtimes C_4\leq S_3\times F_{20}$ with $S_3/C_2\simeq F_{20}/D_5$, 
$15T22\simeq (A_5\times C_3)\rtimes C_2\leq S_3\times S_5$ with $S_3/C_3\simeq S_5/A_5$. 

It follows from Proposition \ref{propconv} that 
if $G=G_1\times G_2$, then 
$\T_{K_1K_2}=R^{(1)}_{K_1K_2/k}(\bG_m)$ as in (1), (2), (3) 
above is not stably $k$-rational. 
For the remaining cases $G=10T4$, $10T12$, $14T4$, $14T46$, $15T6$, $15T22$ 
where $G\lneq G_1\times G_2$, 
$\T_{K_1K_2}=R^{(1)}_{K_1K_2/k}(\bG_m)$ as in (1), (2), (3) 
above is also not stably $k$-rational, see 
Theorem \ref{th1.23} (Table $3$). 
\end{example}
\begin{remark}\label{rem8.7}
In Theorem \ref{th1.23}, two cases 
$14T16\simeq \PSL_2(\bF_7)\rtimes C_2$ and 
$15T10\simeq S_5$ as in Table $3$ (not stably but retract $k$-rational cases) 
can not be obtained as a subdirect product 
$G=nTm\leq S_n$ 
of $G_1=n_1Tm_1\leq S_{n_1}$, 
$G_2=n_2Tm_2\leq S_{n_2}$ with $n=n_1n_2$ and $\gcd(n_1,n_2)=1$ 
as in Example \ref{ex8.4}. 
\end{remark}

\subsection{Stably/retract $k$-rational norm one tori $\T_K=R^{(1)}_{K/k}(\bG_m)$ of field extensions $K/k$ where $K=K_1\otimes\cdots\otimes K_t$ with $[K:k]=n\leq 30$}
By Theorem \ref{th8.1}, we obtain the following theorem for 
$18\leq n\leq 30$ which is not a prime power, 
i.e. $n=18$, $20$, $21$, $22$, $24$, $26$, $28$, $30$ 
(see Theorem \ref{th1.23} for $2\leq n\leq 16$ and Theorem \ref{th1.18} 
for $n=p$; prime): 

\begin{theorem}[{Application of Theorem \ref{th8.1}: Stably/retract $k$-rational norm one tori $\T_K=R^{(1)}_{K/k}(\bG_m)$ of field extensions $K/k$ where $K=K_1\otimes\cdots\otimes K_t$ with $[K:k]=n\leq 30$}]\label{th8.8}
Let the notation be given as in Theorem \ref{th8.1}.
%
Assume that $18\leq n\leq 30$ and $n$ is not a prime power, 
i.e. $n=18$, $20$, $21$, $22$, $24$, $26$, $28$, $30$. 
Assume also that 
$\gcd(n_i,n_j)=1$ for any $1\leq i<j\leq t$. 
Then $K:=K_1\otimes\cdots\otimes K_t=K_1\cdots K_t$ 
is a finite separable field extension of $k$ with 
$[K:k]=n:=n_1\cdots n_t$, 
$G={\rm Gal}(L/k)\leq G_1\times\cdots\times G_{n_t}$, 
$G=nTm\leq S_n$ transitive and 
$H={\rm Gal}(L/K)\leq G$ with 
$H=\cap_{i=1}^t\varphi_i^{-1}(H_i)$ and $[G:H]=n$.
Then 
all of the subdirect product 
$G=nTm\leq S_n$ $(n=18$, $20$, $21$, $22$, $24$, $26$, $28$, $30)$ of 
$G_1,\ldots,G_t$ 
with $G\leq G_1\times\cdots\times G_t$, 
$n=n_1\cdots n_t$, 
$G_i=n_iTm_i\leq S_{n_i}$ 
and $\gcd(n_i,n_j)=1$ $(1\leq i<j\leq t)$ 
when $R^{(1)}_{K_i/k}(\bG_m)$ $(1\leq i\leq t)$ is stably 
$($resp. retract$)$ $k$-rational 
are given 
as in Table $4$ $($resp. Table $4$ and Table $5)$.\\
{\rm (1)} $\T_K=R_{K/k}^{(1)}(\bG_m)$ is stably $k$-rational if $G$ is given as in Table $4$;\\
{\rm (2)} $\T_K=R_{K/k}^{(1)}(\bG_m)$ is not stably but 
retract $k$-rational if $G$ is given as in Table $5$. 

In particular, if $k$ is a global field and 
$\T_K=R^{(1)}_{K/k}(\bG_m)$ 
is stably $($resp. retract$)$ $k$-rational, then 
$\Sha^2_\omega(J_{G/H})=0$ and hence 
$\T_K$ has the weak approximation property and 
the Hasse norm principle holds for $K/k$. 
\end{theorem}
\begin{proof}
By Theorem \ref{th8.1}, 
when $n=n_1\cdots n_t$ with $\gcd(n_i,n_j)=1$ for any $1\leq i<j\leq t$, 
if $\T_{K_i}=R^{(1)}_{K_i/k}(\bG_m)$ $(1\leq i\leq t)$ is stably 
(resp. retract) $k$-rational, 
then $\T_K=R^{(1)}_{K/k}(\bG_m)$ with $K=K_1\cdots K_t$ is stably 
(resp. retract) $k$-rational. 
Then the Galois group $G={\rm Gal}(L/K)=nTm\leq S_n$ is a subdirect product of 
$G_i={\rm Gal}(L_i/k)=n_iTm_i\leq S_{n_i}$ $(1\leq i\leq t)$. 
By Lemma \ref{lem2.2}, 
all of the subdirect product 
$G=nTm\leq S_n$ $(n=18$, $20$, $21$, $22$, $24$, $26$, $28$, $30)$ of 
$G_1,\ldots,G_t$ 
with $G\leq G_1\times\cdots\times G_t$, 
$n=n_1\cdots n_t$, 
$G_i=n_iTm_i\leq S_{n_i}$ 
and $\gcd(n_i,n_j)=1$ $(1\leq i<j\leq t)$ 
when $\T_{K_i}=R^{(1)}_{K_i/k}(\bG_m)$ $(1\leq i\leq t)$ 
is stably (resp. retract) $k$-rational 
are given as in Table $4$ (resp. Table $4$ and Table $5$). 

For (2), we should check that $T=R^{(1)}_{K/k}(\bG_m)$ 
is not stably $k$-rational in the cases as in Table $5$. 
It follows from Proposition \ref{propconv} that 
if $G=G_1\times\cdots\times G_t$, then 
$\T_K=R^{(1)}_{K/k}(\bG_m)$ 
is not stably $k$-rational. 
For the cases with $G\lneq G_1\times\cdots\times G_t$ 
(the bold cases as in Table $5$), we split them into two cases:\\
(i) $G=21T10$, $21T58$, $22T4$, $22T45$, 
$26T6$, $26T8$, $26T83$, $28T12$, $28T360$. 
We can find 
$N=\langle {\rm Ker}(\varphi_1),{\rm Ker}(\varphi_2)\rangle \lhd G$, 
$N_1=\varphi_1({\rm Ker}(\varphi_2))\lhd G_1$, 
$N_2=\varphi_2({\rm Ker}(\varphi_1))\lhd G_2$ 
with $N=N_1\times N_2$ and $[G:N]=2$ 
which satisfy that 
$[J_{G/H}]^{fl}|_N=[J_{N/(H\cap N)}]^{fl}\neq 0$. 
This implies that $\T_K=R^{(1)}_{K/k}(\bG_m)$ is not stably $k$-rational;\\
(ii) $G=20T5$, $20T9$, $20T66$, $21T2$, $21T4$, $26T4$, $30T6$, 
$30T23$, $30T25$, $30T89$, $30T165$. 
We can confirm that 
$\T_K=R^{(1)}_{K/k}(\bG_m)$ is not stably $k$-rational 
by using the GAP (\cite{GAP}) function 
{\tt PossibilityOfStablyPermutationF} as in Hoshi and Yamasaki 
\cite[Algorithm F4, page 90]{HY17} 
(see also Hasegawa, Hoshi and Yamasaki \cite[Algorithm 4.1, page 930]{HHY20}). 
The related GAP functions are also available in \cite{RatProbAlgTori}. 
\end{proof}
%
\begin{center}
Table $4$: $\T_K=R_{K/k}^{(1)}(\bG_m)$ 
is stably $k$-rational where $G={\rm Gal}(L/k)=nTm\leq S_n$ 
$(n=18$, $20$, $21$, $22$, $24$, $26$, $28$, $30)$\vspace*{2mm}\\
\begin{tabular}{l} 
$G=nTm$ $(n=18$, $20$, $21$, $22$, $24$, $26$, $28$, $30)$: $\T_K=R_{K/k}^{(1)}(\bG_m)$ is stably $k$-rational\\\hline
$18T1\simeq C_{18}$, 
${\bf 18T5}\simeq D_9$, 
$18T13\simeq D_9\times C_2\simeq D_{18}$, 
$18T260\simeq \PSL_2(\bF_8)\times C_2$\\
$20T1\simeq C_{20}$, 
${\bf 20T2}\simeq Q_{20}\simeq C_5\rtimes C_4$, 
$20T6\simeq D_5\times C_4$, 
$20T63\simeq A_5\times C_4$\\
$21T1\simeq C_{21}$, 
$21T3\simeq D_7\times C_3$, 
${\bf 21T5}\simeq D_{21}$, 
$21T6\simeq C_7\times S_3$, 
$21T8\simeq D_7\times S_3$\\
$22T1\simeq C_{22}$, 
${\bf 22T2}\simeq D_{11}$, 
$22T3\simeq D_{11}\times C_2\simeq D_{22}$\\
$24T1\simeq C_{24}$,  
${\bf 24T8}\simeq C_3\rtimes C_8$, 
$24T32\simeq S_3\times C_8$\\
$26T1\simeq C_{26}$, 
${\bf 26T2}\simeq D_{13}$, 
$26T3\simeq D_{13}\times C_2\simeq D_{26}$\\
$28T1\simeq C_{28}$, 
${\bf 28T3}\simeq Q_{28}\simeq C_7\rtimes C_4$, 
$28T8\simeq D_7\times C_4$\\
$30T1\simeq C_{30}$, 
$30T2\simeq C_5\times S_3$, 
${\bf 30T3}\simeq D_{15}$, 
$30T4\simeq D_5\times C_3$, 
$30T5\simeq D_5\times C_6$,\\ 
$30T8\simeq D_5\times S_3$, 
$30T10\simeq D_5 \times S_3$, 
$30T12\simeq C_5\times D_6$, 
${\bf 30T13}\simeq D_5 \times S_3$,\\
$30T14\simeq D_{30}$,
$30T21\simeq D_5\times D_6$, 
$30T30\simeq A_5\times C_2$, 
$30T85\simeq A_5\times S_3$,\\ 
$30T87\simeq A_5\times C_6$, 
$30T177\simeq A_5\times D_6$
\end{tabular}
\end{center}\vspace*{2mm}

\begin{center}
Table $5$: $\T_K=R_{K/k}^{(1)}(\bG_m)$ 
is not stably but 
retract $k$-rational where $G={\rm Gal}(L/k)=nTm\leq S_n$ 
$(n=18$, $20$, $21$, $22$, $24$, $26$, $28$, $30)$\vspace*{2mm}\\
\begin{tabular}{l} 
$G=nTm$ $(n=18$, $20$, $21$, $22$, $24$, $26$, $28$, $30)$: $\T_K=R_{K/k}^{(1)}(\bG_m)$ is not stably but retract $k$-rational\\\hline
${\bf 20T5}\simeq F_{20}$,
${\bf 20T9}\simeq F_{20}\times C_2$,
$20T20\simeq F_{20}\times C_4$,
${\bf 20T66}\simeq A_5\rtimes C_4$,
$20T123\simeq S_5 \times C_4$\\
${\bf 21T2}\simeq F_{21}$, 
${\bf 21T4}\simeq F_{42}$, 
$21T7\simeq F_{21}\times C_3$, 
$21T9\simeq F_{42}\times C_3$, 
${\bf 21T10}\simeq C_7\rtimes (S_3\times C_3)$,\\
$21T11\simeq F_{21}\times S_3$,
$21T15\simeq F_{42}\times S_3$,
$21T22\simeq \PSL_2(\bF_7)\times C_3$,
$21T27\simeq \PSL_2(\bF_7)\times S_3$,\\
$21T44\simeq A_7\times C_3$, 
$21T56\simeq S_7\times C_3$, 
$21T57\simeq A_7\times S_3$, 
${\bf 21T58}\simeq A_7\rtimes S_3$, 
$21T74\simeq S_7\times S_3$\\
${\bf 22T4}\simeq F_{110}$, 
$22T5\simeq F_{55}\times C_2$, 
$22T6\simeq F_{110}\times C_2$, 
$22T13\simeq \PSL_2(\bF_{11})\times C_2$,\\ 
$22T27\simeq M_{11}\times C_2$,
${\bf 22T45}\simeq S_{11}$, 
$22T46\simeq A_{11}\times C_2$, 
$22T47\simeq S_{11}\times C_2$\\
${\bf 26T4}\simeq F_{52}$, 
$26T5\simeq F_{39}\times C_2$, 
${\bf 26T6}\simeq F_{78}$, 
$26T7\simeq F_{52}\times C_2$, 
${\bf 26T8}\simeq F_{156}$, 
$26T9\simeq F_{78}\times C_2$,\\ 
$26T10\simeq F_{156}\times C_2$, 
$26T48\simeq \mathrm{GL}_3(\bF_3)$, 
${\bf 26T83}\simeq S_{13}$, 
$26T84\simeq A_{13}\times C_2$, 
$26T85\simeq S_{13}\times C_2$\\
${\bf 28T12}\simeq C_7\rtimes C_{12}$, 
$28T13\simeq F_{21}\times C_4$, 
$28T26\simeq F_{42}\times C_4$, 
$28T87\simeq \PSL_2(\bF_7)\times C_4$,\\ 
${\bf 28T360}\simeq A_7\rtimes C_4$, 
$28T362\simeq A_7\times C_4$, 
$28T429\simeq S_7\times C_4$\\
${\bf 30T6}\simeq C_{15}\rtimes C_4$, 
$30T7\simeq F_{20}\times C_3$, 
$30T17\simeq (C_{15}\rtimes C_4)\times C_2$, 
${\bf 30T23}\simeq F_{20}\times S_3$,\\ 
$30T24\simeq F_{20}\times S_3$,
${\bf 30T25}\simeq S_5$, 
$30T26\simeq F_{20}\times C_6$, 
$30T32\simeq F_{20}\times C_3$, 
$30T51\simeq F_{20}\times D_6$,\\ 
$30T60\simeq S_5\times C_2$,
${\bf 30T89}\simeq S_5\rtimes C_3$,
$30T103\simeq S_5\times C_3$,
${\bf 30T165}\simeq S_5\times S_3$, 
$30T167\simeq S_5\times S_3$,\\ 
$30T168\simeq (A_5 \rtimes S_3)\times C_2$, 
$30T170\simeq S_5\times S_3$, 
$30T180\simeq S_5\times C_6$, 
$30T263\simeq S_5 \times D_6$
\end{tabular}
\end{center}\vspace*{2mm}

In Table $4$ and Table $5$, 
{\bf the bold} means that the case which can not be obtained 
as a direct product 
$G=nTm\leq S_n$ of 
$G_1,\ldots,G_t$ 
with $G= G_1\times\cdots\times G_t$, 
$n=n_1\cdots n_t$, 
$G_i=n_iTm_i\leq S_{n_i}$ 
and $\gcd(n_i,n_j)=1$ $(1\leq i<j\leq t)$. 
In Table $5$, 
$F_{pl}\simeq C_p\rtimes C_l$ $(2<l\mid p-1)$ 
is the Frobenius group of order $pl$ where $p$ is a prime number 
and $M_{11}$ is the Mathieu group of degree $11$.


\end{document}